\documentclass[microtype]{gtpart}

\usepackage{amsmath, amssymb, amsthm, amsfonts}
\usepackage{graphicx}
\usepackage{xcolor}

\usepackage[shortlabels]{enumitem}
\usepackage[normalem]{ulem}

\usepackage{hyperref}
\hypersetup{
linktoc=all,
colorlinks=true,
linkcolor=blue,
citecolor=blue,
linkbordercolor=purple,
citebordercolor=green}
\usepackage[capitalize]{cleveref}

\theoremstyle{plain}
\newtheorem{theorem}{Theorem}[section]
\newtheorem{proposition}[theorem]{Proposition}
\newtheorem{lemma}[theorem]{Lemma}
\newtheorem{corollary}[theorem]{Corollary}
\newtheorem{claim}{Claim}[theorem]
\newtheorem*{theorem*}{Theorem}
\newtheorem{remark}[theorem]{Remark}
\newtheorem{fact}[theorem]{Fact}
\newtheorem{question}[theorem]{Question}
\newtheorem{example}[theorem]{Example}
\newtheorem{illustrative_example}[theorem]{Example}

\newtheorem*{example*}{Example}

\theoremstyle{remark}

\theoremstyle{definition}
\newtheorem{definition}[theorem]{Definition}
\newtheorem{observation}[theorem]{Observation}

\DeclareMathOperator{\Aut}{Aut}
\DeclareMathOperator{\Isom}{Isom}
\DeclareMathOperator{\Sp}{Span}

\DeclareMathOperator{\id}{Id}
\DeclareMathOperator{\st}{stable}
\DeclareMathOperator{\diag}{diag}
\DeclareMathOperator{\Pos}{Pos}
\DeclareMathOperator{\diam}{diam}
\DeclareMathOperator{\Stab}{Stab}

\DeclareMathOperator{\Haus}{Haus}
\DeclareMathOperator{\BS}{BS}
\DeclareMathOperator{\Dell}{\mathcal{D}_\ell}
\DeclareMathOperator{\Cell}{\mathcal{C}_\ell}
\DeclareMathOperator{\Bell}{\mathcal{B}_\ell}
\DeclareMathOperator{\hil}{d_\Omega}

\DeclareMathOperator{\evmax}{\lambda_{max}}
\DeclareMathOperator{\evmin}{\lambda_{min}}
\DeclareMathOperator{\prim}{Prim}
\DeclareMathOperator{\orb}{orb}

\DeclareMathOperator{\partiali}{\partial_{\rm i}}

\DeclareMathOperator{\convcore}{\Cc^{c}_{\Omega}(\Gamma)}
\DeclareMathOperator{\smax}{\mathcal{S}_\Gamma}
\DeclareMathOperator{\pid}{\rho}

\DeclareMathOperator{\orblimset}{\Lc^{\orb}_{\Omega}}

\newcommand{\rkonelimset}{\Lambda^{G/Q}_{\Gamma}}

\newcommand{\Om}{\Omega}
\newcommand{\bdry}{\partial \Om}
\newcommand{\clOm}{\overline{\Om}}
\newcommand{\G}{\Gamma}
\newcommand{\g}{\gamma}
\newcommand{\RP}{\mathbb{P}(\mathbb{R}^{d+1})}
\newcommand{\Amin}{\mathcal{A}_{\min}(x_0)}

\DeclareMathOperator{\PGL}{PGL_{d+1}(\mathbb{R})}
\DeclareMathOperator{\GL}{GL_{d+1}(\mathbb{R})}
\DeclareMathOperator{\SL}{SL}
\DeclareMathOperator{\SO}{SO}
\newcommand{\PS}{\mathcal{P}\mathcal{S}}

\DeclareMathOperator{\CH}{ConvHull}
\DeclareMathOperator{\dist}{d}

\DeclareMathOperator{\CAT}{CAT}
\DeclareMathOperator{\BF}{BF}

\newcommand{\e}{\varepsilon}
\newcommand{\ev}{\lambda}
\newcommand{\T}{\widetilde}

\newcommand{\wt}{\widetilde}

\newcommand{\Rb}{\mathbb{R}}
\newcommand{\Zb}{\mathbb{Z}}
\newcommand{\Pb}{\mathbb{P}}
\newcommand{\Ab}{\mathbb{A}}
\newcommand{\Nb}{\mathbb{N}}
\newcommand{\Cb}{\mathbb{C}}

\newcommand{\Hb}{\mathbb{H}}

\newcommand{\Xb}{\mathbb{X}}

\newcommand{\Cc}{\mathcal{C}}
\newcommand{\Ac}{\mathcal{A}}
\newcommand{\Nc}{\mathcal{N}}
\newcommand{\Lc}{\mathcal{L}}
\newcommand{\Uc}{\mathcal{U}}
\newcommand{\Bc}{\mathcal{B}}

\newcommand{\Ic}{\mathcal{I}}

\newcommand{\Kc}{\mathcal{K}}

\newcommand{\Hc}{\mathcal{H}}

\begin{document}

\title[Rank One Hilbert Geometries]{Rank One Hilbert Geometries}
\author{Mitul Islam}
\address{Department of Mathematics, University of Michigan, MI 48109,  \newline
Current address: Max Planck Institute for Mathematics in the Sciences, 04103 Leipzig, Germany}
\email{mitul.islam@gmail.com}
\date{\today}

\begin{abstract}
We develop a notion of rank one properly convex domains (or Hilbert geometries) in the real projective space.  This is in the spirit of rank one non-positively curved Riemannian manifolds and CAT(0) spaces. We define rank one isometries for Hilbert geometries and characterize them as being equivalent to contracting elements (in the sense of geometric group theory). We prove that if a discrete subgroup of automorphisms of a Hilbert geometry contains a rank one isometry, then the subgroup is either virtually cyclic or acylindrically hyperbolic. This leads to several applications like infinite-dimensionality of the space of quasi-morphisms, counting results for conjugacy classes and genericity results for rank one isometries.
\end{abstract}

\maketitle

\tableofcontents

\section{Introduction}

A \emph{properly convex domain} in $\RP$ is an open subset $\Omega \subset \RP$ such that $\clOm$ is a  bounded convex domain in an affine chart.  Any such domain $\Om$ carries a canonical distance function $\hil$, called the \emph{Hilbert metric} on $\Om$, defined using projective cross-ratios (see \cref{sec:hil-geom}).  Then $\Omega$ equipped with its Hilbert metric constitutes a  \emph{Hilbert geometry}.   A motivating example is given by the open projective disk $\Omega_2:=\{[x:y:1] \in \Pb(\Rb^3) | x^2+y^2<1\}$, a properly convex domain in $\Pb(\Rb^3)$.  In fact, $(\Omega_2,{\rm d}_{\Omega_2})$ is the projective model of the 2-dimensional real hyperbolic space $\Hb^2$.

For a properly convex domain $\Omega$,  the group $\Aut(\Om):=\{ g \in \PGL : g \Om = \Om\}$ acts properly and isometrically on $(\Omega,\hil)$.  If $\Gamma \leq \Aut(\Omega)$ is a discrete subgroup, then the quotient space $\Omega/\Gamma$ is `locally modeled' on $(\Omega,\hil)$.  These are the main objects that we study in this paper. We make the following definition.

\begin{definition}\label{defn:hil-geom}
We say that $\Omega$ is a \emph{Hilbert geometry} if $\Omega \subset \RP$ is a properly convex domain.  Further, we say that a pair $(\Omega,\Gamma)$ is a \emph{Hilbert geometry} if $\Omega \subset \RP$ is a properly convex domain and $\Gamma \leq \Aut(\Omega)$ is a discrete subgroup.  A Hilbert geometry $(\Omega,\Gamma)$ is \emph{divisible} if $\G \leq \Aut(\Om)$ acts co-compactly on $\Om$.  
\end{definition}
\begin{example*}
Consider the projective model $\Omega_2$ of $\Hb^2$.  Here $\Aut(\Omega_2)={\rm PO}(2,1)$.   If $\Gamma \leq {\rm PO}(2,1)$ is any discrete subgroup,  then $(\Omega_2,\Gamma)$ is a Hilbert geometry and $(\Omega_2,\Gamma)$ is divisible when $\Gamma$ is a uniform lattice.
\end{example*}

The boundary of a Hilbert geometry $\Om$, denoted by $\bdry$, is the topological boundary of $\Om$ as a subset of $\RP$. The regularity of $\bdry$ strongly influences the geometric properties of $(\Om,\hil)$. For instance, consider the class of \emph{strictly convex} Hilbert geometries, i.e. Hilbert geometries $\Om$ such that $\bdry$ does not contain any non-trivial projective line segments. Benoist showed that strictly convex divisible Hilbert geometries $(\Om,\G)$ have $C^1$ boundaries and behave like compact Riemannian manifolds of negative curvature (more precisely, $\G$ is Gromov hyperbolic and the geodesic flow is Anosov) \cite{benoist_cd1}. This analogy between strictly convex Hilbert geometries and Riemannian negative curvature was subsequently studied by many authors with much success (see \cite{benoist_cd_survey} or \cite{marquis_survey} for a survey). 

On the other hand, the \emph{non-strictly convex} Hilbert geometries (i.e. when $\bdry$ contains non-trivial projective line segments) have remained elusive. There are only a few examples (see Section \ref{sec:eg-div-hil-geom}) and, until recently, only a limited number of results. Taking a cue from the strictly convex case, one hopes to liken non-strictly convex Hilbert geometries to Riemannian non-positive curvature, or more generally, $\CAT(0)$ spaces. This will be our guiding principle in this paper. But we remark that the similarity with CAT(0) geometry is superficial. In fact, an old theorem of Kelly-Strauss \cite{KS1958} states that $\Om$ is $\CAT(0)$ if and only if $\Om$ is the projective model of the real hyperbolic space.  Thus, one needs to use very different tools and techniques for working with Hilbert geometries as compared to CAT(0) spaces.

Our target in this paper is to classify Hilbert geometries into two broad classes: `rank one' and `higher rank'. The motivation for this classification comes from the success of the rank rigidity theorem for non-positively curved Riemannian manifolds  \cite{WB1985, burns_spatzier_rank_rigid}. Roughly, this theorem states that there is a dichotomy for irreducible compact Riemannian manifolds of non-positive curvature: either the manifold is `rank one', or it is a higher rank Riemannian locally  symmetric space. Similar rank rigidity theorems have been proven in other `non-positive curvature' settings \cite{SAGEEV_CAPRACE_2011, Ricks_2019} and conjectured for CAT(0) spaces. We remark  that the usual definition of rank for Riemannian manifolds uses  Jacobi fields and will not be useful for Hilbert geometries. This is because the geodesic flow on a generic non-strictly convex Hilbert geometry is only $C^0$.

We introduce a notion of rank one geodesics in $(\Om,\hil)$ using projective geometry. Consider an open projective line segment $(a,b) \subset \Omega$ with $a,b \in \bdry$. Then $(a,b)$ is a bi-infinite geodesic for the Hilbert metric $\hil$. We will say that $(a,b)$ is a \emph{rank one geodesic} provided it is not contained in a \emph{half triangle in $\Om$}, i.e. either $(a,c) \subset \Om$ or $(c,b) \subset \Om$ for any $c \in \bdry$ (see \cref{fig:rank_one} and Definitions \ref{defn:half-T} and \ref{defn:rank_one_geodesic}). The notion of a half triangle in Hilbert geometry is analogous to the notion of a half flat in $\CAT(0)$ geometry \cite[Section III.3]{ballmann_book}. Our above definition of a rank one geodesic is motivated by an analogous characterization of rank one geodesics in CAT(0) geometry. In a CAT(0) geodesic metric space, a rank one geodesic does not bound a half flat.

\begin{figure}
\centering
\begin{minipage}{0.45\textwidth}
\includegraphics[scale=0.45]{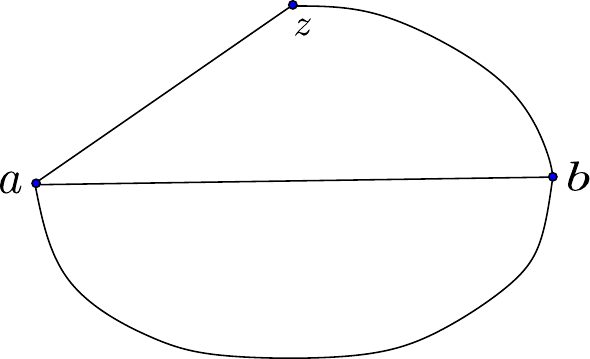}
\end{minipage}
\begin{minipage}{0.45\textwidth}
\includegraphics[scale=0.45]{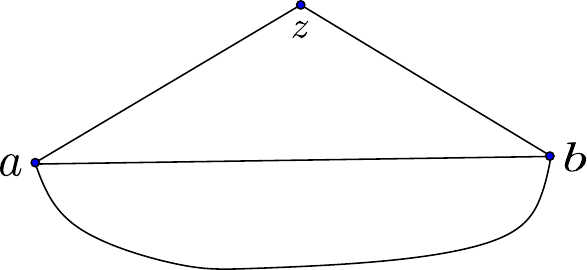}
\end{minipage}

\caption{In the left figure, $(a,b)$ is a \emph{rank one geodesic} while in the right figure, $(a,b)$ is contained in a \emph{half triangle in $\Om$}. However, \cref{prop:rank-one-properties} will show that neither of these can be a closed rank one geodesic (i.e. a  \emph{rank one axis}) in $\Omega/\G$. }
\label{fig:rank_one}
\end{figure}

We will say that an isometry $\g \in \Aut(\Om)$ is a \emph{rank one isometry} if $\g$ acts by  a translation along a rank one geodesic $\ell \subset \Om$, see  \cref{defn:rank-one-isometry}. We remark that acting by a translation along a rank one geodesic (i.e. having a \emph{rank one axis}) is much more special than simply translating along any projective geodesic (i.e having an \emph{axis}), see \cref{rem:rank_one_axis_vs_axis}. Our definition of rank one isometry is again analogous to a characterization of rank one isometries in CAT(0) geometry \cite{ballmann_axial_isometries, ballmann_orbihedra}. A rank one isometry $\g \in \Aut(\Om)$ has several properties reminiscent of hyperbolic isometries in $\Isom(\Hb^2)$: $\g$ is biproximal, has exactly two fixed points $\g^{\pm}$ in $\clOm$,  has a unique axis $(\g^+,\g^-) \subset \Om$, both fixed points are `visible' (i.e. $(\g^+,z) \cup (z,\g^-)$ for any $z \in \bdry-\{\g^+,\g^-\}$), and $\g$ has the so-called north-south dynamics on $\bdry$ (see \cref{prop:rank-one-properties} and \cref{cor:ns_dynamics}). In the case where $(\Om,\G)$ is divisible, it is quite easy to detect a rank one isometry: if $\g \in \Aut(\Om)$ has an axis and is biproximal, then $\g$ is a rank one isometry (see Proposition \ref{prop:biprox-equiv-no-half-T}).

We further the analogy between rank one isometries in Hilbert geometry and CAT(0) geometry by proving that rank one isometries (in our sense above) are contracting elements in the sense of Sisto \cite{sisto_contracting_rw}. Sisto introduced the notion of contracting elements to capture the essence of ``negative curvature'' in groups (see Section \ref{sec:defn-contracting-element}). He proved that if $\Lambda$ acts properly by isometries on a proper $\CAT(0)$ space $X$, then an element of $\Lambda$ is contracting if and only if it is rank one (in the sense of of CAT(0) geometry) \cite[Proposition 3.14]{sisto_contracting_rw}. Our first main result in the paper is an analogue of this result for Hilbert geometries. If $\Om$ is a Hilbert geometry, let $\PS^\Om:=\{ [x,y]  : x, y \in \Om\}$ where $[x,y]$ is a projective line segment joining $x$ and $y$. 

\begin{theorem}
\label{thm:contracting-iff-rank-one}
(see Part \ref{part:contracting_element_hil_geom})
If $\Om$ is a Hilbert geometry, then $\g \in \Aut(\Om)$ is a contracting element for $(\Om,\PS^\Om)$ if and only if $\g$ is a rank one isometry.
\end{theorem}

In the light of these analogies, one naturally expects that the presence of many rank one isometries would induce interesting ``negative curvature"-like behavior. To formalize this, we now introduce the notion of rank one Hilbert geometries. An example to keep in mind is $(\Omega_2,\G)$ where $\Om_2 \subset \Pb(\Rb^3)$ is the projective model of $\Hb^2$ and $\G \leq {\rm PO}(2,1)$ is an infinite discrete subgroup.
\begin{definition}\label{defn:rank-one-hil-geom}
A \emph{rank one Hilbert geometry} is a pair $(\Om,\G)$ where $\Om \subset \RP$ is a Hilbert geometry and $\G$ is a discrete subgroup of $\Aut(\Om)$ that contains a rank one isometry. 
\end{definition}
Morally, if a rank one group $\G$ as in \cref{defn:rank-one-hil-geom} is not virtually cyclic (i.e. does not contain a finite index cyclic subgroup), then it contains many rank one isometries and we expect the group $\G$ to appear quite `hyperbolic'.  But of course we cannot expect such a group $\G$ to always be Gromov hyperbolic  -- there are many examples to the contrary, see \cref{sec:eg-div-hil-geom}. The main result of this paper is to identify the notion of hyperbolicity that rank one groups satisfy. We prove that a rank one group is either virtually cyclic or an acylindrically hyperbolic group, see \cref{thm:acy-hyp} below.

The notion of acylindrically hyperbolic groups, introduced by Osin in \cite{osin_acy_hyp}, is a generalization of the notion of non-elementary  Gromov hyperbolic groups. Roughly speaking, a group is acylindrically hyperbolic if it admits a non-elementary  action on a (possibly non-proper) Gromov hyperbolic metric space with all but finitely many elements acting `hyperbolically' (see \cref{defn-acy-hyp-1}). This family includes many important classes of groups: mapping class groups of most finite-type surfaces, rank one $\CAT(0)$ groups that are not virtually abelian, relatively hyperbolic groups that are not virtually cyclic and have proper peripheral subgroups, and outer automorphism groups of free groups on at least two generators \cite[Appendix]{osin_acy_hyp}.  We prove the following.

\begin{theorem}
\label{thm:acy-hyp}
(see Section \ref{sec:acy_hyp})
If $(\Om, \G)$ is a rank one Hilbert geometry, then either $\G$ is virtually cyclic or $\G$ is an acylindrically hyperbolic group.
\end{theorem}

The acylindrical hyperbolicity of rank one Hilbert geometries $(\Omega,\G)$ have several applications. We defer this discussion until the subsection \hyperlink{applications}{\bf Applications} below. Instead we first mention some of the precursors to our above result. In all those previous results, the conclusion is that the group under consideration is either Gromov hyperbolic or relatively hyperbolic - both of which are acylindrically hyperbolic groups (see the previous paragraph).  Benoist showed that if $(\Om,\G)$ is divisible and $\partial \Om$ is strictly convex, then $\G$ is Gromov hyperbolic \cite{benoist_cd1}. If instead $\Om/\G$ is non-compact but has finite volume, then Cooper-Long-Tillmann showed that $\G$ is relatively hyperbolic (with respect to the cusp subgroups) \cite[Theorem 0.15]{cooper_long}. More generally, if $\Om/\G$ is geometrically finite, then Crampon-Marquis proved that $\G$ is relatively hyperbolic \cite[Theorem 1.8]{CM2014}. However they require that $\bdry$ is $C^1$  and not just strictly convex. 

Outside the strictly convex setting, \cite{IZ2019} has recently shown that if $\G$ acts co-compactly on $\Om$, then $\G$ is relatively hyperbolic (with peripheral subgroups free abelian of rank at least two)  if and only if the set of properly embedded simplices in $\Om$ of dimension at least two (cf. \ref{subsec:simplex}) forms an `isolated family'. The results in \cite{IZ2019} hold in greater generality -- whenever $\G$ acts convex co-compactly on $\Om$, see \cref{defn:cc} and \cref{sec:cc-rel-hyp} for further discussion.  We can interpret the above \cref{thm:acy-hyp} as a generalization of these aforementioned results in the general setup of (possibly non-strictly convex) Hilbert geometries. \cref{thm:acy-hyp} characterizes the existence of rank one isometries in $\G$ as a key factor that underpins the presence of these various weak forms of hyperbolicity for the group $\G$.

\subsection*{Zariski density and rank one} Before moving on to contrasting rank one against `higher rank' Hilbert geometries, we indulge in a short discussion about Hilbert geometries  $(\Om,\G)$ where $\G \leq \Aut(\Omega)$ is Zariski dense in $\SL_{d+1}(\Rb)$, i.e. `large' in an algebraic sense. For such groups $\G$, one can define a notion of proximal limit set $\rkonelimset \subset \RP$ (see \ref{defn:proximal_limset} and \ref{rem:proximal_limset_ZD}) that is independent of the properly convex domain $\Om$.  In \cref{sec:rank-one-hil-geom}, we prove that: if $x,y \in \rkonelimset \cap \bdry$ such that $(x,y) \subset \Om$ is a rank one geodesic, then the set of rank one isometries in $\G$ form a Zariski dense set in $\SL_{d+1}(\Rb)$ and $(x,y)$ can be approximated by the axes of rank one isometries (\cref{lem:limit-set-rank-one-hil-geom}). In particular, a Hilbert geometry $(\Om,\G)$, with $\G$ Zariski dense in $\SL_{d+1}(\Rb)$,  is rank one if and only if $\Omega$ contains a rank one geodesic $(x,y)$ with $x,y \in \rkonelimset \cap \bdry$, see \hyperlink{ans:answer_to_ques_in_sec_8}{Answer to Question 8.1}.

\subsection*{Rank one versus higher rank}
The class of rank one Hilbert geometries is quite rich. Besides the strictly convex Hilbert geometries, there are several examples of non-strictly convex divisible Hilbert geometries which are rank one, e.g. the 3-manifold groups constructed in \cite{benoist_cd4} from projective reflection groups. For more examples, see Section \ref{sec:eg-div-hil-geom}. In Appendix \ref{sec:eg-rank-one}, we discuss this further and also generalize the notion of rank one for convex co-compact actions. 

On the other hand, there are several examples of Hilbert geometries that are not rank one, or in other words, have `higher rank'. Projective simplices of dimension at least two and  symmetric domains of real rank at least two are the key examples of `higher rank' Hilbert geometries (see \cref{subsec:simplex} and \ref{sec:eg-div-hil-geom}). The former are examples of reducible `higher rank' while the latter are examples of irreducible `higher rank' domains (see  \cref{defn:irreducible}). At this point, it is natural to ask whether these are all the `higher rank' divisible Hilbert geometries, akin to the case of Riemannian non-positive curvature. Recently, A. Zimmer \cite{Z2019} has proven that this is indeed the case. 

We will now briefly discuss Zimmer's result for context. Zimmer calls $\Om$ a \emph{higher rank Hilbert geometry} if any $(p,q) \subset \Omega$ is contained in a  properly embedded projective simplex $S$ in $\Om$ of dimension at least two (cf. \ref{subsec:simplex}). Under some assumptions, he proves that his notion of higher rank is exactly complementary to our notion of rank one. We remark that Zimmer does not  develop a theory of rank one geometries. He focuses only on higher rank geometries and  proves that an irreducible divisible Hilbert geometry $(\Om,\G)$ is  higher rank (in his sense) if and only if it does not satisfy the notion of rank one (in the sense introduced in this paper). 

\begin{theorem}[{Part of \cite[Theorem 1.4]{Z2019}}]
\label{thm:rank_rigidity_detailed}
Suppose $(\Om,\G)$ is  a divisible Hilbert geometry and $\Omega$ is irreducible. Then the following are equivalent:
\begin{enumerate}
\item $\Om$ has higher rank (in the sense of Zimmer \cite[Definition 1.1]{Z2019}).
\item $\Om$ is a symmetric domain of real rank at least two.
\item  $\Om$ does not contain any rank one geodesics (in the sense of this paper, \cref{defn:rank_one_geodesic}).
\item $\G$ does not contain any rank one isometries (in the sense of this paper, \cref{defn:rank-one-isometry}).
\end{enumerate}
\end{theorem}

\hypertarget{applications}{\subsection*{Applications}}

We now return to our discussion about rank one geometries. There is a sizable literature exploring different properties of acylindrically hyperbolic groups. By virtue of \cref{thm:acy-hyp}, we can use these to establish several interesting results about rank one Hilbert geometries. We remark that in the ensuing discussion, we usually do not require the additional assumption of divisibility.

\subsection{Second bounded cohomology and quasi-morphisms.} 
A quasi-morphism of a group $G$ is a function $f: G \to \Rb$ such that $\sup_{g,h \in G}|f(gh)-f(g)-f(h)|$ is finite. We say that two quasi-morphisms are equivalent if they differ by a bounded function or a homomorphism of $G$ into $\Rb$. The set of all equivalence classes of quasi-morphisms of $G$ constitute $\T{QH}(G)$, which is a $\Rb$-vector space. More generally, if $\rho:G \to \Uc(E)$ is a unitary representation of $G$ on a complete normed $\Rb$-vector space $(E,|| \cdot||)$, then we can define $\T{QC}(G;\rho)$ (see Section \ref{sec:second_bdd_cohom}).

Bestvina-Fujiwara proved in \cite{bestvina_fujiwara_symm_space} that if $M$ is a compact non-positively curved Riemannian manifold, then - under some mild assumptions - $\T{QH}(\pi_1(M))$ is infinite dimensional if and only if $M$ is a rank one Riemannian manifold. We prove a similar cohomological characterization of rank one Hilbert geometries.

\begin{theorem}(see Section \ref{sec:second_bdd_cohom}) \label{thm:rk-1-hil-infinite-QH}
If $(\Omega,\G)$ is a rank one Hilbert geometry, $\G$ is torsion-free and $\G$ is not virtually cyclic, then
\begin{enumerate}
\item $\dim \big( \T{QH} (\G) \big) = \infty$, and 
\item if $p \in(1, \infty)$ and $\rho^p_{\rm reg}: \G \to \Uc \big( \ell^p(\G)\big)$ is the regular representation, then $\dim \big( \T{QC}(\G;\rho^p_{\rm reg}) \big) =\infty$. 
\end{enumerate}
\end{theorem}

\noindent  We prove a more general Theorem \ref{thm:rk-1-hil-infinite-QC-rho}. On the other hand, if $\G \leq G$ is a lattice in a higher rank simple Lie group $G$, then a result of Burger-Monod implies that $\T{QH}(\G)=0$ \cite[Theorem 21]{BM2002}. Then Theorem \ref{thm:rk-1-hil-infinite-QH} and the rank rigidity result (\cref{thm:rank_rigidity_detailed}) implies:

\begin{corollary}(see Section \ref{sec:second_bdd_cohom})
\label{cor:qh-rigid}
If $(\Omega,\G)$ is a divisible Hilbert geometry and $\Om$ is irreducible, then $\dim (\T{QH} (\G)) = \infty$ if and only if $(\Om,\G)$ is a rank one Hilbert geometry. Otherwise $\dim (\T{QH}(\G))=0$.
\end{corollary}

\subsection{Counting of conjugacy classes.} 
\label{subsec:intro-counting}
For $g \in \Aut(\Omega)$, define the translation length (cf. \ref{sec:translation_length}) $\tau_{\Omega}(g):=\inf_{x \in \Omega} \hil(x,gx)$  and the stable translation length $$\tau_{\Omega}^{\st}(g):=\lim_{n \to \infty} \frac{\hil(x,g^nx)}{n}.$$ Note that $\tau_{\Omega}^{\st}(g)$ is independent of the base point $x \in \Omega$.  Now suppose that $(\Omega,\G)$ is a rank one Hilbert geometry. For $g \in \Gamma$, let $[c_g]$ denote the conjugacy class of $g$ in $\G$. Both $\tau_{\Omega}$ and $\tau_{\Omega}^{\st}$ are well-defined on the set of conjugacy classes in $\G$. Then for $t > 0$, define 
\begin{align*}
\Cc(t)&:=\#\{ [c_g] : g \in \G, ~\tau_{\Omega}([c_g]) \leq t \} \quad \text{ and } \\
\Cc^{\st}(t)&:= \# \{ [c_g] : g \in \G, \tau_{\Omega}^{\st}([c_g]) \leq t \}.
\end{align*}

Here $\Cc(t)$ (resp. $\Cc^{\st}(t)$) counts the number of conjugacy classes in $\Gamma$ whose translation length (resp. stable translation length)  is at most $t$. For divisible rank one Hilbert geometries, we prove an asymptotic growth formula for $\Cc(t)$ and $\Cc^{\st}(t)$. To state our result, we will require the critical exponent of $\Gamma$ which is defined as
\begin{equation*}
\omega_{\G}:=\limsup_{n \to \infty}\dfrac{\log \# \{ g \in \G: \hil(x,gx) \leq n\}}{n}
\end{equation*}
for some (and hence any) base point $x \in \Omega$. 

\begin{theorem}
\label{thm:counting-conjugacy}
(see Section \ref{sec:counting-conjugacy})
Suppose $(\Om, \G)$ is a divisible rank one Hilbert geometry and $\G$ is not virtually cyclic. Then there exists a constant $D'$ such that if $t\geq 1$, 
\begin{align*}
\dfrac{1}{D'} \dfrac{\exp (t\omega_{\G})}{t} ~\leq~ \Cc(t) ~\leq~ D' \dfrac{\exp (t\omega_{\G})}{t}.
\end{align*}
The function $\Cc^{\st}(t)$ also satisfies a similar growth formula as above.
\end{theorem}
\begin{remark}\label{rem:counting}\
\begin{enumerate}
\item An element $g\in\Gamma$ is called primitive if $g \neq h^n$ for any $h \in \Gamma$ and $|n| \geq 2$. If $\Cc_{\prim}(t)$ is the number of conjugacy classes $[c_g]$ of primitive elements in $\Gamma$ such that $\tau_{\Omega}([c_g]) \leq t$, then $\Cc_{\prim}(t)$ satisfies a similar growth formula as $\Cc(t)$.
\item In \cite[Proposition 1.5]{PLB2021}, Blayac establishes finer counting results for (a related notion of) rank one Hilbert geometries using very different techniques.
\end{enumerate}
\end{remark}
 Counting of conjugacy classes is often connected to counting of closed geodesics. However this connection is subtle for Hilbert geometries since there could be elements in $\G$ that do not act by a translation along any projective line in $\Omega$ (i.e. do not have an axis, cf. \cref{ex:triangle} (B)).

\subsection{Genericity and random walks.} Suppose $(\Omega,\G)$ is a rank one Hilbert geometry, $\G$ is not virtually cyclic and $\G$ is finitely generated. If $S$ is a finite symmetric generating set of $\G$, let $W_n(S)$ be the set of words of length $n$ in the elements of $S$. A simple random walk on $\G$ (with support $S$) is a sequence of $\G$-valued random variables $\{X_n\}_{n\in \Nb}$ with laws $\mu_n$ defined by: if $n\geq 1$ and $g \in \G$
\begin{align*}
\mu_n(\{g\})=\dfrac{\# \{ w \in W_n(S) : w \text{ represents } g \}}{\#W_n(S)}. 
\end{align*}
Using results in \cite{sisto_contracting_rw}, we show that rank one isometries in $\G$ are exponentially generic from the viewpoint of simple random walks. This roughly means that the probability that a long word, written down by randomly choosing generators of $\G$, is not a rank one isometry is small. In particular, this probability decays exponentially in the length of the word.
\begin{proposition}(see Section \ref{sec:other-proofs})
\label{prop:generic-rw}
Suppose $(\Omega,\G)$ is a rank one Hilbert geometry, $\G$ is not virtually cyclic and $\G$ is finitely generated. Then the rank one isometries in $\G$ are exponentially generic:  if $\{X_n\}_{n \in \Nb}$ is a simple random walk on $\G$, then there exists a constant $C \geq 1$ such that for all $n \geq 1$ 
\begin{equation*}\Pb \left[~X_n \text{ is not a rank one isometry} ~\right] \leq C e^{-n/C}.
\end{equation*}
\end{proposition}

\subsection{More consequences of acylindrical hyperbolicity}  
\begin{proposition} 
\label{prop:appl-acy-hyp}
(see Section \ref{sec:other-proofs}) If $(\Om,\G)$ is a rank one Hilbert geometry and $\G$ is not virtually cyclic, then: 
\begin{enumerate}
\item $\G$ is SQ-universal, i.e. every countable group embeds in a quotient of $\G$. 
\item if $\G$ is the Baumslag-Solitar group $\BS(m,n)$, then $m=n=0$ and $\G$ is the free group on two generators.
\end{enumerate}
\end{proposition}

\subsection{Morse geodesics and Morse boundary.}
Roughly speaking, the Morse geodesics \cite{CORDES2017} in a geodesic metric space identify the ``hyperbolic directions". As a corollary to Theorem \ref{thm:contracting-iff-rank-one}, we prove that the axis of a rank one isometry is a Morse geodesic. 

\begin{proposition}(see Section \ref{sec:other-proofs})
\label{prop:rank-one-axis-morse}
If $\Om$ is a Hilbert geometry and $\g \in \Aut(\Om)$ is a rank one isometry, then the axis $\ell_\g$ of $\g$ is $\Kc$-Morse for some Morse gauge $\Kc: [1,\infty)\times[0,\infty)\to [0,\infty)$, i.e. if $\alpha$ is a $(\lambda,\varepsilon)$-quasi-geodesic  with endpoints on $\ell_\g$, then $\alpha \subset \Nc_{\Kc(\lambda,\varepsilon)}(\ell_\g)$.
\end{proposition}

 In \cite{CORDES2017}, Cordes introduced a notion of Morse boundary for proper geodesic metric spaces. In the cases of proper CAT(0) spaces and hyperbolic metric spaces, the Morse boundary coincides with the contracting boundary and the Gromov boundary respectively. Theorem \ref{thm:contracting-iff-rank-one} and Proposition \ref{prop:rank-one-axis-morse} implies that the Morse boundary $\partial_M \Om$ of a rank one Hilbert geometry $(\Om,\G)$ is non-empty. This inspires the following question (that we will not answer in this paper).
\begin{question}
Describe the Morse boundary $\partial_M \Om$ of a rank one Hilbert geometry $(\Om, \G)$. 
\end{question}

\subsection*{Recent developments}
Since this paper first appeared on the arXiv, there has been tremendous new developments in the field. We mention a few of them. Blayac has developed the Patterson-Sullivan theory for rank one Hilbert geometries in \cite{PLB2021}. In \cite{BV2023}, Blayac-Viaggi  has constructed examples of divisible rank one Hilbert geometries $(\Om,\G)$ in every dimension $d \geq 3$  where  $\G$ is not Gromov hyperbolic. In their examples, $\G$ is Zariski dense in $\SL_{d+1}(\Rb)$ and relatively hyperbolic with peripheral subgroups isomorphic to $\Zb \times H$, $H$ possibly non-Abelian \cite[Theorem 1.3]{BV2023}.

\subsection*{Outline of the paper}
We discuss the preliminaries in Part \ref{part:prelims}. Section \ref{sec:dynamics-auto} of Part \ref{part:prelims} is of particular interest as it discusses the geometry of $\omega$-limit sets of automorphisms in $\Aut(\Om)$. Part \ref{part:rank-one-hil-geom} develops the notion of rank one Hilbert geometries. We define rank one isometries and study their geometric properties in Section \ref{sec:rank-one-isom} and Section \ref{sec:rk-1-thin-triangles}. Our main tools here are the lemmas proven in Section  \ref{sec:axis}. In Section \ref{sec:rank-one-hil-geom}, we study rank one groups which are Zariski dense.

In Part \ref{part:contracting_element_hil_geom}, we prove Theorem \ref{thm:contracting-iff-rank-one} (in Sections \ref{sec:rank-one-implies-contracting} and \ref{sec:contracting-implies-rank-one}) and \cref{thm:acy-hyp} (in Section \ref{sec:acy_hyp}). Part \ref{part:applications} discusses several applications of our results, like computing the dimension of the space of quasi-morphisms, counting of conjugacy classes and genericity from the viewpoint of random walks. We discuss generalizations, examples and non-examples of rank one Hilbert geometries in Appendix \ref{sec:eg-rank-one}. In Appendix \ref{appendix:contracting-BF}, we discuss the equivalence of  two notions of contracting elements.

\subsection*{Acknowledgement}
I would like to heartily thank my advisor Ralf Spatzier. This work has benefited greatly from my conversations with Andrew Zimmer, Harrison Bray, Samantha Pinella, Thang Nguyen, Ludovic Marquis and Beatrice Pozzetti. I also thank the anonymous reviewers for their comments, corrections and constructive criticism. I was partially supported by NSF Grant 1607260 during the course of this work.

\part{Preliminaries}
\label{part:prelims}

\section{Notation}
\label{sec:notation}
We set the following notation as standard for the rest of the paper.
\begin{enumerate}
\item If $v \in \Rb^{d+1}\setminus \{0\}$, let $[v]$ or $\pi(v)$ denote its image in $\RP$. Conversely, if $u \in \RP$, we will use $\T u$ to denote a lift of $u$ (i.e. $\pi(\T u)=u$). 
\item If $A \in \GL$, let $[A]$ denote its image in $\PGL$  while $\T B \in \GL$ will denote a lift of $B \in \PGL$.
\item If $W \leq \Rb^{d+1}$ is a non-zero linear subspace, let $\Pb(W)$ denotes its projectivization. 
\item If $g \in \GL$, the eigenvalues of $g$ (over $\Cb$) are denoted by $\ev_1(g),\dots,\ev_{d+1}(g)$. We index them in the non-increasing order of their absolute values, i.e. $|\ev_1(g)| \geq \cdots \geq |\ev_{d+1}(g)|.$\\
Let $\evmax(g):=|\ev_1 (g)|$ and $\evmin(g):=|\ev_{d+1}(g)|.$
\item If $g \in \PGL$ and $1 \leq i \neq j \leq d+1$, $\left|\dfrac{\ev_i}{\ev_j}(g) \right|:=\left| \dfrac{\ev_i(\T g)}{\ev_j(\T g)} \right|$ for some (hence any) lift $\wt g \in \GL$ of $g$.
\end{enumerate}

\section{Hilbert Geometries}
\label{sec:hil-geom}

\subsection{Properly convex domains} 
\label{subsec:prop-convex}
An open set $\Om \subset \RP$ is called a \emph{properly convex domain} if there exists a co-dimension one subspace $H \leq \Rb^{d+1}$ such that $\clOm$ is a bounded (Euclidean) convex domain in the affine chart $\Ab:=\RP \setminus \Pb(H)$.

\begin{remark}
If $L \leq \Rb^{d+1}$ is a 2-dimensional linear subspace, then $\Pb(L) \not \subset \clOm$ for any properly convex domain $\Om$. This elementary observation that a properly convex domain cannot contain an entire projective line will be used in many of the proofs in this paper.
\end{remark}

 For a non-empty set $X \subset \Rb^{d+1}$, let $\Sp (X)$ denote the linear span of $X$. If $X'\subset \RP$ is non-empty, define $$\Sp(X'):= \Sp \left( \{ \T x \in \Rb^{d+1} : \pi(\T x) \in X'\} \right).$$ Suppose $\Om$ is a properly convex domain. If $x,y \in \clOm$, let $[x,y]$ denote the unique closed connected subset of $\Pb(\Sp\{x,y\}) \cap \clOm$ that joins $x$ and $y$. We will call $[x,y]$ \emph{the projective line segment} between $x$ and $y$ (note that the notion of a projective line segment dependens on $\Om$, but we assume that $\Om$ will be clear from context and suppress it for brevity).  We introduce the notation: $(x,y):=[x,y]\setminus \{x,y\}$, $[x,y):=[x,y] \setminus \{y\}$ and $(x,y]:=[x,y] \setminus \{x\}.$ We will call $(x,y)$ an \emph{open projective line segment}. When we say that $[x,y]$ (or $(x,y)$) is \emph{non-trivial}, we mean $x \neq y$. 

We have a notion of convexity and convex hull in a properly convex domain $\Om$. A set $Y\subset \clOm$ is \emph{convex} if $[y_1,y_2] \subset Y$ for all $y_1, y_2 \in Y$. If $X \subset \clOm$ is a non-empty set, then $\CH_{\clOm}(X)$ is the smallest closed convex subset of $\clOm$ that contains $X$. We define $$\CH_{\Om}(X):=\Om \cap \CH_{\clOm}(X).$$

\subsection{Hilbert metric} 
Suppose $\Om$ is a properly convex domain and $\Ab$ is affine chart that contains $\clOm$ as a compact subset. We equip $\Ab$ with the Euclidean norm $|\cdot|$. If $x, y\in \Om$, then there exist $a,b \in \bdry$ such that $\Pb(\Sp(\{x,y\})) \cap \clOm=[a,b]$ where the four points appear in the order $a,x ,y,b$.  The cross-ratio of these four points is given by 
\begin{equation*}
 \left[ a, x, y, b \right]:= \left( \dfrac{|b-x||y-a|}{|b-y||x-a|}\right).
 \end{equation*}
The \emph{Hilbert metric} on $\Omega$ is defined by
\begin{equation*}
 \hil(x,y):=\frac{1}{2}\log ~([ a, x, y, b ]).
 \end{equation*}
 
\begin{observation}
\label{rem:distance_comparison_between_domains}
If $\Omega' \subset \Omega$ are properly convex domains and $x,y \in\Omega'$, then $\dist_{\Omega}(x,y) \leq \dist_{\Omega'}(x,y)$.
\end{observation}

 If $x,y \in \Om$, then $[x,y]$ is a geodesic in $(\Om,\hil)$ joining $x$ and $y$. In order to emphasize this fact, we will often refer to the projective line segment $[x,y]$ as the \emph{projective geodesic segment} between $x$ and $y$. If $(x,y) \subset \Om$ with $x,y \in \bdry$, then $(x,y)$ is a bi-infinite geodesic in $(\Om,\hil)$ and we will call it a \emph{(bi-infinite) projective geodesic}. 
 
 The space $(\Om,\hil)$ is a proper, complete and geodesic metric space and we will call $\Om$ a \emph{Hibert geometry}, see \cref{defn:hil-geom}. The group $\Aut(\Om) :=\{ g \in \PGL : g \Om = \Om \}$ acts properly and isometrically on $(\Om,\hil)$. However note that the projective geodesic may not be the unique geodesic between points in $(\Om,\hil)$. Consider, for example, the two-dimensional simplex $T_2:=\Pb(\Rb^+e_1 \oplus \Rb^+e_2 \oplus \Rb^+e_3)$ with its Hilbert metric $d_{T_2}$. Then generic points $x \neq y \in T_2$ have uncountably many geodesics (for the Hilbert metric $d_{T_2}$) joining them  \cite[Proposition 2]{harpe_hilbert_metric}.

\begin{definition}
\label{defn:cone_over_om}
For a Hilbert geometry $\Om$, $\pi^{-1}(\Om):=\{v \in \Rb^{d+1} : \pi(v) \in \Omega\}$ has two connected components.  The \emph{cone above (or over) $\Om$}, denoted by $\T \Om$, is a connected component of $\pi^{-1}(\Om)$.
\end{definition} 
Then $\pi^{-1}(\Om)=\T \Om \sqcup (-\T \Om)$. Note that if $g \in \Aut(\Om)$, then there is a lift $\T g \in \GL$ of $g$ that preserves $\wt \Omega$, i.e.  $\T g \cdot \T \Om=\T \Om$. Indeed, if a lift $\T g$ does not preserve $\T \Om$, then $\T g \cdot \T \Om=-\T \Om$ and hence $-\T g$ preserves $\T \Om$. We have the following elementary observation about such lifts of automorphisms. 
\begin{observation}
\label{lem:lift_has_pos_eigenvalues}
Suppose $\wt \Omega$ is a cone above $\Omega$ and $\wt g \in \GL$ preserves $\wt \Om$. If $\wt{a} \in \wt \Om$ satisfies $\wt g \cdot \wt a= \ev \cdot  \wt{a}$, then $\ev>0$.  
\end{observation}
\begin{proof}
Clearly $\ev \neq 0$. As $\wt g$ preserves $\wt \Om$, $\wt g \cdot \wt a \in \wt \Om$ which implies that $\ev \cdot \wt a \in \wt \Om$. Now if $\ev<0$, then $\wt a \in (-\wt \Om)$.  But then $\wt a \in \wt \Om \cap -\wt \Om=\emptyset$, a contradiction.\end{proof}

 \subsection{Projective simplices}
 \label{subsec:simplex}

The standard $k$ dimensional projective simplex in $\RP$ is $$T_k:=\{[x_1:\dots:x_{k+1}:0:\dots:0] \in \RP ~|~  x_1,\dots,x_{k+1}>0 \}.$$   A $k$-dimensional projective simplex is a subset of $\RP$ of the form $gT_k$ for some $g \in \PGL$. If $\Om \subset \RP$ is a properly convex domain and  $S \subset \Om$ is a projective simplex, then we say that $S$ is a \emph{properly embedded simplex in $\Om$} if and only if $\partial S \subset \partial \Om$.

The Hilbert metric $\dist_{T_k}$ on $T_k$ is given by
\begin{align*}
\dist_{T_k}([x_1:\dots:x_{k+1}:0:\dots:0],[y_1:\dots:y_{k+1}:0:\dots:0])=\max_{1\leq i,j \leq k+1}\dfrac{1}{2} \left| \log \dfrac{x_iy_j}{x_jy_i} \right|.
\end{align*}
Then $(T_k,\dist_{T_k})$ is quasi-isometric to the real Euclidean space of dimension $k$. For a more elaborate discussion, see \cite[Section 5]{IZ2019}, \cite{N1988} or \cite{harpe_hilbert_metric}. The group $\Aut(T_d)$ is generated by the group of permutation matrices in $\PGL$ and the group $\{[\diag(\ev_1,\dots,\ev_{d+1})] \in \PGL : \ev_1,\dots,\ev_{d+1}>0\}$.

\begin{lemma}
\label{lem:simplex_translation_dist}
Suppose $g:=[\diag(\ev_1,\dots,\ev_{d+1})] \in \Aut(T_d)$ where $\ev_i>0$ for all $i=1,\dots,d+1$. Let $\ev_{\max}:=\max_{1\leq i\leq d+1} \ev_i$ and $\ev_{\min}:=\min_{1 \leq i\leq d+1}\ev_i$. Then $\dist_{T_d}(x,gx)=\dfrac{1}{2}\log\dfrac{\ev_{\max}}{\ev_{\min}}$ for any $x \in T_d$.
\end{lemma}
\begin{proof}
Fix $x=[x_1:\dots:x_{d+1}]\in T_d$. Using the formula for $\dist_{T_d}$,
\begin{equation*}
\dist_{T_d}(x,gx)=\max_{1\leq i,j \leq d+1}\dfrac{1}{2} \left| \log \dfrac{x_i\ev_jx_j}{x_j\ev_ix_i} \right|=\max_{1\leq i,j \leq d+1}\dfrac{1}{2} \left| \log \dfrac{\ev_j}{\ev_i} \right|=\dfrac{1}{2}\log\dfrac{\ev_{\max}}{\ev_{\min}}.  \qedhere
\end{equation*}
\end{proof}

\subsection{Examples of Hilbert geometries.} 
\label{sec:eg-div-hil-geom}

The projective open ball $\Om_d:=\{ [x_1:\dots:x_d:1 | \sum_{i=1}^d x_i^2 < 1\}$ in $\RP$ is the simplest example of a  divisible strictly convex Hilbert geometry. In fact $\Om_d$ with its Hilbert metric is isometric to $\Hb^d$ and is well-known as the Beltrami Klein model of the real hyperbolic space. Moreover $\Aut(\Om_d)={\rm PO}(d,1).$ There are several examples of divisible strictly convex Hilbert geometries that are not isometric to $\Hb^d$: in dimension 4, there is a construction due to Benoist \cite[Proposition 3.1]{B2006} while Kapovich constructed examples in all dimensions above 4 \cite{kapovich2007}.

 Among non-strictly convex (divisible) Hilbert geometries, the simplest example is the standard $d$-simplex $T_d$, see \ref{subsec:simplex}. But this example is reducible, a term which we now define. Recall that a convex cone in $\Rb^{d+1}$ is a set $C \subset \Rb^{d+1}$ such that $r_1v_1+r_2v_2 \in C$ whenever $v_1, v_2 \in C$ and $r_1, r_2 >0$.

\begin{definition}
\label{defn:irreducible}
A properly convex domain $\Om \subset \RP$ is \emph{reducible} if there exist convex cones $C_1 \subset \Rb^{d_1}$ and $C_2 \subset \Rb^{d_2}$ with $d_1,d _2 \geq 1$ such that $\Om=\Pb(C_1 \oplus C_2)$.  Otherwise, $\Om$ is \emph{irreducible}.
\end{definition}

 An irreducible non-strictly convex (divisible) example is $\Pos_d$ ($d \geq 3$), the set of positive-definite real symmetric $d \times d$ matrices of unit trace. It is a properly convex domain in $\Rb^{d(d+1)/2}$ and is a projective model for the symmetric space of $\SL_d(\Rb)$. The notion of symmetric domains generalize $\Pos_d$. A symmetric domain $\Om$ is a properly convex domain such that: for each $x \in \Om$, there exists an order two isometry $s_x \in \Aut(\Om)$ where $x$ is the unique fixed point of $s_x$ in $\Om$. Symmetric domains of real rank at least two are real projective analogues of higher rank Riemannian symmetric spaces of non-positive curvature, see \cite{benoist_cd_survey, Z2019} for details. As one might expect, symmetric domains are very special in the theory of properly convex domains. The rank rigidity theorem \ref{thm:rank_rigidity_detailed} mentioned in the introduction is also a result in this spirit. Benoist proved the following result.
 
 \begin{theorem}[{\cite[Theorem 5.2]{benoist_cd_survey}}]
 \label{thm:benoist_zariski_dense}
 Suppose $\Om \subset \RP$ is an irreducible properly convex domain that is not a symmetric domain.  If $\G \leq \SL_{d+1}(\Rb)$ is a discrete subgroup that acts co-compactly on $\Om$, then $\G$ is Zariski dense in $\SL_{d+1}(\Rb)$. 
 \end{theorem}

 Besides simplices and the symmetric domains of real rank at least two, only a few examples of divisible non-strictly convex Hilbert geometries are known. These are low dimensional examples; see for instance \cite{benoist_cd4, CLM2016}, which rely on Coxeter group constructions, or \cite{BDL2018}, which uses `cusp-doubling' construction for certain three manifolds.

\subsection{Closest-point projection for the Hilbert metric.} 
Suppose $\Om$ is a Hilbert geometry. If $r>0$, we set $$\Bc_{\Om}(x,r):=\{ y \in \Om : \hil(x,y) < r \}.$$ 
\begin{lemma}[{\cite[Lemma1.7]{cooper_long}}]
\label{lem:convexity_of_balls}
$\Bc_{\Om}(x,r)$ is a relatively compact and convex set.
\end{lemma}
If $C \subset \Om$ is a closed convex set, we define the closest-point projection on $C$ as: if $x \in \Om$, 
\begin{align*}
\Pi_C(x):=C \cap \overline{\Bc_{\Om}(x, \hil(x,C))}.
\end{align*}

As the intersection of two closed convex sets is again a closed convex set, \cref{lem:convexity_of_balls} immediately implies the following.
\begin{observation}
Suppose $\Om$ is a Hilbert geometry, $C$ is a closed convex set and $x \in \Om$. Then $\Pi_C(x)$ is a compact convex set.
\end{observation}
\begin{corollary}
\label{cor:P-sigma}
Suppose $\sigma: \Rb \to  (\Om,\hil)$ is a unit speed parametrization of the bi-infinite  projective geodesic $\sigma(\Rb)$ with $\sigma(\pm \infty) \in \bdry$. If $x \in \Omega$, then there exist  $T_x^-, T_x^+ \in \Rb$ with $T_x^-\leq T_x^+$ such that $$\Pi_{\sigma(\Rb)}(x)= \left[ \sigma(T_x^-),\sigma(T_x^+) \right].$$ 
\end{corollary}
\begin{proof}
Any compact convex subset of the bi-infinite projective geodesic $\sigma(\Rb)$ is of the form $[\sigma(T), \sigma(T')]$ with $T \leq T'$. 
\end{proof}

\subsection{Faces of properly convex domains}
Suppose $\Om$ is a Hilbert geometry. We define the relation $\sim_{\Omega}$: if $p, q \in \clOm$, then $p \sim_{\Omega} q$ if and only if  there exists an open projective line segment in $\clOm$ that contains both $p$ and $q$. The relation $\sim_{\Om}$ is an equivalence relation (see \cite[Section 3.3]{CM2014}). The equivalence class of $p \in \clOm$ is called the (open) face of $p$ and is denoted by $F_{\Omega}(p)$.

\begin{proposition} \label{prop:faces}
Suppose $\Omega \subset \RP$ is a Hilbert geometry. 
\begin{enumerate}
\item If $x \in \bdry$, then $F_{\Omega}(x) \subset \bdry.$ 
\item $F_{\Omega}(x)= \Omega$ if and only if $x \in \Omega.$
\item If $x, y \in \bdry$, then either $[x,y] \subset \bdry$ or $(x,y) \subset \Omega.$ 
\item Suppose $[x, y] \subset \bdry$, $a \in F_{\Omega}(x)$ and $b \in F_{\Omega}(y)$. Then, $[a,b] \subset \bdry.$
\end{enumerate}
\end{proposition}
\begin{proof} For part (1), note that if $y \in \Omega$, then $[y,x]$ cannot be extended beyond $x$ in $\clOm.$ Thus $F_{\Om}(x) \subset \clOm-\Om=\bdry$. Part (2) follows from part (1). \\
\noindent (3) If $[x,y] \not \subset \bdry$, then choose any $z \in (x,y) \cap \Omega.$ So, $F_{\Omega}(z)=\Omega$. Then $(x,y) \subset F_{\Omega}(z) \subset \Omega.$  \\
\noindent (4) It suffices to prove this for $b=y$, i.e. to prove that $[a,y] \subset \bdry.$ Suppose, for a contradiction, that $(a,y) \subset \Omega$. Then $a \neq x$. Pick $a' \in F_{\Omega}(x)$ such that $x \in (a,a')$.  As $(a,y) \subset \Omega$, pick $w\in (a,y)$. Then $(w,a')\subset \Omega$. Thus $\CH_{\Omega}\{a',y,a\}$ is a non-empty set, as it contains $(w,a')\subset \Omega$. Hence, $\CH_{\Omega}\{a',y,a\} \subset \Omega$. This implies that $a,a'$ and $y$ span a 2-simplex in $\clOm$ and the interior of this 2-simplex is contained in $\Omega$.   As $x \in (a,a')$, $(x,y)$ is contained in the interior of this 2-simplex. Thus $(x,y) \subset \Omega$, a contradiction. 
\end{proof}

\subsection{Distance estimates}

\begin{proposition}[{\cite[Proposition 5.2]{IZ_flat_torus}}]
\label{line-in-bdry-0}
Suppose $\Om$ is a Hilbert geometry and $\{x_n\}$ and $\{y_n\}$ are sequences in $\Om$  such that $x:=\lim_{n \to \infty}x_n$ and $y:=\lim_{n \to \infty}y_n$ exist in $\clOm.$ If $$\liminf_{n \to \infty} \hil(x_n,y_n) < \infty,$$ then $y \in F_{\Om}(x)$ and $$\dist_{F_{\Om}(x)}(y,x) \leq \liminf_{n \to \infty} \hil(x_n,y_n).$$ 
\end{proposition}
Note that if $\liminf_{n \to \infty} \hil(x_n,y_n)=0$, then the above proposition implies that $y=x$.

Next we need the notion of Hausdorff distance. If $(X,d)$ is a metric space, then the Hausdorff distance between $A,B \subset X$ is defined by
\begin{align*}
d^{\Haus}(A,B)=\max \left\{ \sup_{a \in A} d(a,B), \sup_{b \in B} d(b,A)\right\}. 
\end{align*}
 
\begin{proposition}[{\cite[Proposition 5.3]{IZ_flat_torus}, \cite[Lemma 1.8]{cooper_long}}]
\label{prop:iz-dist-estimate}
Suppose $\Om$ is a Hilbert geometry and $x_1,x _2, y_1, y_2 \in \clOm$ such that $F_{\Om}(x_1)=F_{\Om}(x_2)$ and $F_{\Om}(y_1)=F_{\Om}(y_2)$. If $(x_1,y_1) \subset \Om$, then 
\begin{equation*}
\hil^{\Haus}\left( (x_1,y_1), (x_2,y_2) \right) \leq \max \{\dist_{F_{\Om}(x_1)}(x_1,x_2),  \dist_{F_{\Om}(x_2)}(y_1, y_2)\}.
\end{equation*}
\end{proposition}

 In particular, if $x_i, y_i \in \Om$, then $\hil^{\Haus}\left( [x_1,y_1], [x_2,y_2] \right) \leq \max \{\hil (x_1,x_2),  \hil(y_1, y_2)\}.$

\subsection{Translation length} 
\label{sec:translation_length}
Supppose $\Om$ is a Hilbert geometry and $g \in \Aut(\Om)$, then its  \emph{translation length} is defined by $\tau_{\Om}(g):=\inf_{x \in \Omega} \hil(x,gx).$

\begin{remark}[{\cite[Observation 7.2]{IZ_flat_torus}}]
\label{rem:translation_length_ineq}
Suppose $g \in \Aut(\Om)$ and $W\leq \Rb^{d+1}$ is a $g$-invariant subspace of dimension $\geq 2$ such that $\Omega \cap \Pb(W)$ is non-empty. Then $\tau_{\Om}(g) \leq \tau_{\Om \cap \Pb(W)}(g|_W)$.
\end{remark}

\begin{proposition}[{\cite[Proposition 2.1]{cooper_long}}]
\label{prop:translation_length}
If $g \in \Aut(\Om)$, then $\tau_{\Om}(g)=\dfrac{1}{2}\log \left| \frac{\ev_1}{\ev_{d+1}}(g) \right|.$
\end{proposition}
Note that this differs from the formula in \cite{cooper_long} by a factor of $1/2$. This is because our definition of Hilbert metric has the  factor of $1/2$. We further remark that if $\wt g \in \GL$ is any lift of $g$, then $\tau_{\Om}(g)=\frac{1}{2}\log\frac{\ev_{\max}(\wt g)}{\ev_{\min}(\wt g)}$.

\subsection{Minimal translation sets.}
\label{subsec:min_translation}
Suppose $\Om$ is a Hilbert geometry and $\G \leq \Aut(\Om)$. If $H \leq \G$ is a subgroup, then the minimal translation set of $H$ in $\Om$ is 
\begin{align*}
{\rm Min}_{\Om}(H):=\bigcap_{h \in H}\{ x \in \Om : \hil(x,h\cdot x) = \tau_{\Om}(h)\}.
\end{align*}

\begin{example}\
\begin{enumerate}
\item If $g=[\diag(\ev_1,\dots,\ev_{d+1})] \in \Aut(T_d)$ with $\ev_i>0$ for all $i=1,\dots,d+1$, then \cref{lem:simplex_translation_dist} implies that $T_d={\rm Min}_{T_d}(\langle g 
\rangle)$.
\item The minimal translation set could be empty, e.g. if $u$ is a parabolic isometry in $PO(2,1)$, then $\tau_{\Hb^2}(u)=0$ and the minimal translation set of $\langle u \rangle$ is empty. 
\end{enumerate}
\end{example}

We will need the following result connecting eigenspaces with minimal translation sets. 
\begin{lemma}
\label{lem:min_trans_set_contains_eigenspace}
Suppose $a,b,c \in \bdry$ are three distinct fixed points of $g \in \Aut(\Omega)$.Then $$\CH_{\Omega}\{ a,b,c\}  \subset {\rm Min}_{\Omega}(\langle g \rangle).$$
\end{lemma}
\begin{proof}Without loss of generality, we can assume that $\CH_{\Omega}\{ a,b,c\}$ is non-empty since the result is obviously true otherwise. Let $T:=\CH_{\Omega}\{a,b,c\}$. 
Fix a cone $\wt \Om$ over $\Omega$ and a lift $\wt g$ of $g$ that preserves $\wt \Omega$. Let $V=\Sp\{a,b,c\}$ and $g_0:=\wt{g}|_V$. Let $\wt a, \wt b, \wt c \in \wt \Omega$ be lifts of $a,b,c$ respectively and fix the basis $\{ \wt a, \wt b, \wt c\}$ of $V$. In this basis, $g_0=\diag (t_1, t_2, t_3)$. By  \cref{lem:lift_has_pos_eigenvalues}, we can assume that $t_1, t_2, t_3>0$. Since $T$ is a 2-simplex in $\Pb(V)$, \cref{lem:simplex_translation_dist}  implies that  $$d_T(x,g_0x)=\dfrac{1}{2} \log  \dfrac{\max\{t_1,t_2,t_3\}}{\min\{t_1,t_2,t_3\}}$$ for any $x \in T$.

Suppose $\ev_1(\wt g), \dots, \ev_{d+1}(\wt g)$ are the eigenvalues of $\wt g$, indexed in the non-increasing order of their modulus. Then
\[
|\ev_{1}(\wt g)| \geq \max\{t_1,t_2,t_3\} \geq \min \{t_1,t_2,t_3\} \geq |\ev_{d+1}(\wt g)|.
\]
By \cref{prop:translation_length}, $\tau_{\Omega}(g)=\dfrac{1}{2}\log \left| \ \dfrac{\ev_1(\wt g)}{\ev_{d+1}(\wt g)} \right|$. 
Then  $d_T(x,g_0x) \leq \tau_{\Omega}(g)$ for any $x\in T$. 

As $\Omega \cap \Pb(V) \supset T$, \cref{rem:distance_comparison_between_domains} implies that  $d_T(y',y) \geq d_{\Omega \cap \Pb(V)}(y',y)$ for any $y',y \in T$. Then $d_T(x,g_0x) \geq \tau_{\Omega \cap \Pb(V)} (g_0)$ for any $x \in T$. Then \cref{rem:translation_length_ineq} implies that $d_T(x,g_0x) \geq \tau_{\Omega}(g)$. Thus $d_T(x,g_0x)=\tau_{\Omega}(g)$ for any $x \in T$. Hence $T \subset {\rm Min}_{\Omega}(g)={\rm Min}_{\Omega}(\langle g \rangle)$.
\end{proof}

The next result concerns translation length and minimal translation sets of compact subgroups. This is essentially restatement of \cite[Lemma 2.1]{marquis_survey} which shows that every compact subgroup of $\Aut(\Omega)$ has a fixed point in $\Omega$.
\begin{lemma}[{\cite[Lemma 2.1]{marquis_survey}}]
\label{lem:compact_subgroup}
Suppose $\Om$ is a Hilbert geometry and $H \leq \Aut(\Omega)$ is a compact subgroup. Then $\tau_{\Omega}(h)=0$ for all $h \in H$ and  
${\rm Min}_{\Omega}(H)=\{ x \in \Omega : H \cdot x=x \}\neq \emptyset$. 
\end{lemma}

\subsection{Centralizers} 
Suppose $\Om$ is a Hilbert geometry and $\G \leq \Aut(\Om)$. If $H \leq \G$ is a subgroup, the centralizer of $H$ in $\G$ is 
\begin{align*}
C_{\G}(H):=\bigcap_{h \in H}\{ g \in \G : ghg^{-1}=h\}.
\end{align*}
We will need to following result on co-compactness of centralizer subgroups. 
\begin{theorem}[{\cite[Theorem 1.10]{IZ_flat_torus}}]
\label{thm:centralizer}
Suppose $\Om$ is a Hilbert geometry, $\Cc \subset \Om$ is a closed convex subset, and $\G \leq \Aut(\Om)$ is a discrete subgroup that acts co-compactly on $\Cc$. If $A \leq \G$ is an Abelian subgroup, then $C_{\G}(A)$ acts co-compactly on $\CH_{\Om}({\rm Min}_{\Cc}(A))$ where $${\rm Min}_{\Cc}(A):=\Cc \cap {\rm Min}_{\Om}(A).$$
\end{theorem}

\subsection{Proximality}
\label{sec:proximality}
 We call $g \in \GL$ proximal if $g$ has a unique eigenvalue of maximum modulus and the multiplicity of this eigenvalue is 1, or equivalently if $$|\ev_1(g)|>|\ev_2(g)|.$$ We will say that $g \in \GL$ is biproximal if both $g$ and $g^{-1}$ are proximal, i.e $|\ev_1(g)|>|\ev_2(g)|$ and $|\ev_d(g)|> |\ev_{d+1}(g)|$. Note that the notion of proximality is invariant under scaling a matrix by non-zero real numbers. Then $\g \in \PGL$ is proximal (resp. biproximal) if some (hence any) lift of $\g$ is proximal (resp. biproximal).

\section{Dynamics of Automorphisms}
\label{sec:dynamics-auto}

\subsection{$\omega$-limit sets of automorphisms}
\label{subsec:defn-omega-lim-set}

Let $\Omega \subset \Pb(\Rb^{d+1})$ be a Hilbert geometry and let $\g \in \Aut(\Omega)$ with $\tau_{\Omega}(\g)=\frac{1}{2}\log \left| \frac{\ev_1}{\ev_{d+1}}(\g) \right|>0$. Recall that for any $X \subset \Omega$, $\overline{X}$ denotes the closure of $X$ in $\clOm$. We define the $\omega$-limit set of $\g$ as
\begin{align*}
\omega(\gamma,\Omega):= \bigcup_{x \in \Omega} \left( \overline{\{ \g^n x : n \in \Nb \}} \cap \bdry \right),
\end{align*}  
Thus $\omega(\gamma,\Omega)$ is the union of all accumulation points in $\bdry$ of all $\{ \gamma^n : n \in \Nb\}$-orbits in $\Omega$. 

\begin{example}
\label{example:omega_limit_in_T2}
Let $\Omega=T_2$ and $\gamma=[\diag(1,2,2)]$. Then for any $x=[x_1:x_2:x_3] \in T_2$, $\lim_{n \to \infty} \gamma^nx=[0:x_2:x_3]$. Thus $$\omega(\gamma,T_2)=\{[0:x_2:x_3]\in \Pb(\Rb^3) | x_2,x_3>0\}.$$ Thus $\omega(\gamma,T_2)$ is the open projective line segment $(\pi(e_2),\pi(e_3)) \subset \partial T_2$. Also note that $\omega(\gamma,T_2)= E_{\g}^+-\{\pi(e_2),\pi(e_3)\}$, where $E_{\g}^+=\Pb (\Sp\{e_2,e_3\}) \cap \bdry$. Here $\Pb (\Sp\{e_2,e_3\})$ is the projectivization of the direct sum of the eigenspace of $\g$ that correspond to eigenvalues of maximum modulus. This observation that $\omega(\g,T_2) \subset E_{\g}^+$ holds more generally, as we will see in \cref{prop:omega-limit-gamma}.
\end{example}

\begin{remark}
We now compare the notion of $\omega$-limit set with that of the full orbital limit set  introduced in \cite{DGK2017}. Given an infinite discrete subgroup $H \leq \Aut(\Omega)$, the full orbital limit set of $H$ is defined in \cite{DGK2017} as
$$\orblimset(H):= \cup_{x \in \Omega} \left( \overline{H \cdot x} \cap \bdry \right).$$
If $\gamma \in \Aut(\Omega)$ and $\tau_{\Omega}(\g)>0$, then $\{ \g^n : n \in \Nb \}$ is an infinite discrete subsemi-group of $\Aut(\Omega)$. Then $\omega(\gamma,\Omega)$ can be interpreted as the full orbital limit set of the subsemi-group $\{ \g^n : n \in \Nb \}$.
\end{remark}

\subsection{Geometry of $\omega$-limit sets of automorphisms}
\label{subsec:geom-omega-lim-set}

For the rest of this subsection, fix a Hilbert geometry $\Omega \subset \Pb(\Rb^{d+1})$ and $\g \in \Aut(\Omega)$ with $\tau_{\Omega}(\g)>0$. Fix a lift $\wt \g$ of $\g$. Our goal here is \cref{prop:omega-limit-gamma} -- a description of $\omega(\gamma,\Omega)$ using the real Jordan decomposition of $\wt \g$. We first give an intuitive idea. Suppose $c_1,\dots,c_q$ are all the eigenvalues (with repetitions) of $\wt \g$ of modulus $\ev_{\max}(\wt \g)$. If $c_1,\dots,c_q \in \Rb$, then $\omega(\gamma,\Omega)$ is contained in the projective subspace spanned by the eigenvectors corresponding to the eigenvalues of modulus $\ev_{\max}(\wt \g)$. In the notation of \cref{defn:E_g} below, this subspace is precisely $\Pb(E_{\wt \g})$. Now suppose that among the $c_i$-s, there is a complex conjugate pair of eigenvalues $\mu,\overline{\mu} \in \Cb-\Rb$. Then, in the above subspace, we need to replace the eigenvectors  for $\mu$ and $\overline{\mu}$ with a 2-dimensional $\g$-invariant projective real subspace on which $\g$ acts by a rotation ($E_{\mu}$ in the notation of \cref{defn:E_g}). This is the key intuition behind \cref{prop:omega-limit-gamma}. The references for this section are \cite[II.1]{M1991} and \cite[Section 2]{cooper_long}).

Now we start the formal discussion. First we introduce some notation. If $\mu \in \Rb$, define 
\begin{align*}
J_{\mu}:=
\begin{pmatrix}
\mu & 1 & 0 & \dots & 0 \\
0 & \mu & 1 & \dots & 0 \\
\vdots & \vdots & \vdots & \vdots & \vdots \\
0 & \dots &  0 & \dots  &1 \\
0 & \dots &  0 & \dots  &\mu
\end{pmatrix}.
\end{align*}
If $\mu=\alpha + i \beta \in \Cb-\Rb$, define
\begin{align*}
J_{\mu}:=
\begin{pmatrix}
R(\mu) & \id_2 & 0 & \dots & 0 \\
0 & R(\mu) & \id_2 & \dots & 0 \\
\vdots & \vdots & \vdots & \vdots & \vdots \\
0 & \dots &  0 & \dots  & \id_2 \\
0 & \dots &  0 & \dots  &R(\mu)
\end{pmatrix}
\end{align*}
where $\id_2$ is the $2 \times 2$ identity matrix and $R(\mu):=\begin{pmatrix} \alpha & -\beta \\ \beta & \alpha \end{pmatrix}$.

 Consider the real Jordan decomposition of $\wt \g$. This says that there is a $\wt \g$ invariant decomposition $\Rb^{d+1}=V_{\mu_1} \oplus \dots \oplus V_{\mu_n}$ into real vector subspaces and, with an appropriate choice of basis for $V_{\mu_j}$, 
\begin{align*}
\wt \g=
\begin{pmatrix}
 J_{\mu_1} & \dots & 0 \\ \vdots & \vdots & \vdots \\ 0 & \dots & J_{\mu_{n}}
\end{pmatrix}. 
\end{align*}
We remark that $V_{\mu}=V_{\overline{\mu}}$, as conjugate pairs of eigenvalues correspond to the same invariant subspace in the real Jordan decomposition. Without loss of generality, we can assume that $\mu_1,\dots,\mu_{l} \in \Rb$ and $\mu_{l+1},\dots, \mu_{n} \in \Cb-\Rb$. Then $\mu_1,\dots,\mu_{l}$, $\mu_{l+1},\overline{\mu_{l+1}},\dots,\mu_{n}, \overline{\mu_{n}}$ are eigenvalues (possibly with repetitions) of $\wt \g$ over $\Cb$ and the multiplicity of $\mu_i$ is determined  by the Jordan block $J_{\mu_i}$.

Note that $\ev_{\max}(\wt \g)$ and $\ev_{\min}(\wt \g)$ are the maximum and the minimum of the set $\{|\mu_i| : 1\leq i \leq n\} $. Now we focus on the eigenvalues of maximum modulus. By re-indexing the $\mu_i$-s, we now assume that $\mu_1,\dots,\mu_m$ are precisely those $\mu_i$ that satisfy $|\mu_i|=\ev_{\max}(\wt \g)$. We further assume that among them, $\mu_1,\dots,\mu_k \in \Rb$ and $\mu_{k+1},\dots, \mu_m \in \Cb-\Rb$. Then $\mu_1,\dots,\mu_k, \mu_{k+1},\overline{\mu_{k+1}},\dots, \mu_m,\overline{\mu_m}$ are eigenvalues (possibly with repetitions) of $\wt \g$  of modulus $\ev_{\max}(\wt \g)$ and their multiplicities are determined by the Jordan block structure of $\wt \g$.

\begin{definition}
\label{defn:E_g}
\emph{If $\mu_j \in \Rb$, let $E_{\mu_j}$ be the eigenvector for $\wt{\g}$ in $V_{\mu_j}$ with eigenvalue $\mu_j$. If $\mu_j \in \Cb-\Rb$, let $E_{\mu_j}$ be the two-dimensional $\wt \g$-invariant subspace of $V_{\mu_j}$ such that $\wt \g|_{E_{\mu_j}}$ is conjugated to $R(\mu_j)$. Define } 
\begin{align*}
 E_{\T \g} := \bigoplus_{1 \leq j \leq m}  E_{\mu_j}=\bigoplus_{|\mu_j|=\ev_{\max}(\wt \g)}  E_{\mu_j}.
\end{align*}
\label{defn:subspaces_L_and_K}
\emph{We also define} $$L_{\T \g}  := \bigoplus_{|\mu_j|=\ev_{\max}(\T \g)}  V_{\mu_j}   \qquad \text{  and  } \qquad K_{\T \g}  := \bigoplus_{|\mu_j|< \ev_{\max}(\T \g)} V_{\mu_j}.$$
\end{definition}
Then  $\wt \g|_{E_{\wt \g}}$ is conjugated in ${\rm GL}(E_{\wt \g})$ to 
\begin{align}
\ev_{\max}(\wt{\gamma}) \cdot 
\label{eqn:action_of_g_on_Eg}
\begin{pmatrix}
M_k & 0 & \dots &  0 \\
0 & R\left( \frac{\mu_{k+1}}{\ev_{\max(\wt \g)}} \right) & \dots & 0 \\
\vdots & \vdots  & \vdots  & \vdots \\
0 & 0 & 0 & R \left( \frac{\mu_{m}}{\ev_{\max(\wt \g)}} \right)
\end{pmatrix},
\end{align} 
where $M_k$ is a $k \times k$ diagonal matrix with each diagonal entry $\pm 1$. Thus $\langle \wt \g|_{E_{\wt \g}} \rangle$ is conjugated in ${\rm GL}(E_{\wt \g})$ to a cyclic subgroup of 
$\{\pm \id\}^k \times O(E_{\mu_{k+1}}) \times \dots \times O(E_{\mu_m}) < O(E_{\wt \g}).$
Here, $O(W)$ denotes the group of orthogonal transformations preserving a linear subspace $W \subset \Rb^d$. 
\begin{claim}
There exists a sequence $\{m_k\}$ in $\Nb$ with $m_k \to \infty$ such that 
$$\lim_{k \to \infty} \frac{1}{\ev_{\max}(\wt \g)^{m_k}} \left( \wt \g|_{E_{\wt \g}} \right)^{m_k}=\id|_{E_{\wt \g}}.$$
\end{claim}
\begin{proof}[Proof of Claim]
Let $k_\g:=\frac{1}{\ev_{\max}(\wt \g)}\wt \g|_{E_{\wt \g}}$ and $\Kc:=\overline{\langle k_{\g}\rangle}$. By equation \eqref{eqn:action_of_g_on_Eg}, there exists $h \in {\rm GL}(E_{\wt \g})$ such that $h \cdot \Kc \cdot h^{-1}$ is a compact subgroup of $\{\pm \id\}^k \times O(E_{\mu_{k+1}}) \times \dots \times O(E_{\mu_m})$. Thus $\Kc':=h \cdot \Kc \cdot h^{-1}$ is a Lie subgroup of $O(E_{\wt \g})$. Hence either $\id$ is an isolated point of $\Kc'$ or there exists a neighborhood $U$ of $\id$ in $O(E_{\wt \g})$ such that $U\cap \Kc' \subset \Kc'$. 

In the latter case, it is obvious that there exists a monotonic sequence of integers $\{m_p\}$  such that $\lim_{p \to \infty} k_{\g}^{m_p}=\id|_{E_{\wt \g}}.$ Up to passing to a subsequence, we can assume that $m_p \to \infty$ or $m_p \to -\infty$. If $m_p \to \infty$, the claim is proved. Otherwise, choose the sequence $-m_p$. 

In the former case (i.e. when $\id$ is an isolated point of $\Kc'$), it implies $\left( k_{\g} \right) ^s=\id|_{E_{\wt \g}}$ for some $s \in \Nb$. Then $m_p:=sp$ proves the claim. 
\end{proof}

We will now discuss the dynamics of $(\wt \gamma)^n$ on $\Pb(\Rb^{d+1})$. The results are quite standard and the proofs are fairly elementary computations using Jordan blocks, see \cite[II.1]{M1991} or \cite[Lemma 2.5]{cooper_long}) for instance.  

\begin{observation} \
\label{obs:iteration_jordan_blocks}
\begin{enumerate}
\item For a generic point $v \in V_{\mu}$, all accumulation points of  $\left\{ \frac{1}{|\mu|^n}\left( \wt \g|_{V_{\mu}} \right)^n v : n \in \Nb \right\}$ 
lie in $E_{\mu}$.
\item Let $W=V_{\mu_1} \oplus V_{\mu_2}$ and $|\mu_1|>|\mu_2|$. Then, for any $w \in W-V_{\mu_2}$, all accumulation points of 
$\left\{ \frac{1}{|\mu_1|^n}\left(\wt \g|_{W} \right)^n w : n \in \Nb \right\}$ 
lie in $E_{\mu_1}$.
\item Let $W'=V_{\mu} \oplus V_{\mu'}$ with $|\mu|=|\mu'|$.  Then for a generic point $w' \in W'$, all accumulation points of 
$\left\{ \frac{1}{|\mu|^n}\left(\wt \g|_{W} \right)^n w' : n \in \Nb \right\}$ lie in $E_{\mu}$ if $\dim V_{\mu} > \dim V_{\mu'}$. If $\dim V_{\mu} = \dim V_{\mu'}$, then the accumulation points lie in $E_{\mu} \oplus E_{\mu'}$. 
\end{enumerate}
\end{observation}

Recall the notation from \cref{defn:E_g}. Then the above observations imply the following result.
\begin{fact}
\label{rem:iteration_using_jordan_blocks}
 If $w \in \RP \setminus \Pb(K_{\T \g})$, then the accumulation points of $\{\g^n w : n > 0\}$ lie in $\Pb(E_{\T \g})$. In particular, if $w' \in \Pb(L_{\wt \g})$, then all accumulation points of $\{ \g^n w' : n>0\}$ also lie in $\Pb(E_{\wt \g})$.  
\end{fact}
\begin{remark}
\label{rem:attracted_to_most_powerful_jordan}
In fact, a finer conclusion is possible in \cref{rem:iteration_using_jordan_blocks}.  Following \cite{cooper_long}, call a real Jordan subspace $V_{\mu_i}$ \emph{most powerful} if $|\mu_i|=\ev_{\max}(\wt \g)$ and $\dim (V_{\mu_i})=\max\{ \dim (V_{\mu_j}): |\mu_j|=\ev_{\max}(\wt \g)\}$. Let $F_{\wt \g}$ be the direct sum of the $E_{\mu_j}$-s that correspond to the most powerful Jordan subspaces $V_{\mu_j}$. Then, $F_{\wt \g} \subset E_{\wt \g}$. For any $w \in \RP \setminus \Pb(K_{\T \g})$ as above, the accumulation points of $\{\g^n w:n>0 \}$ actually lie in $\Pb(F_{\wt \g})$, see part (3) of the above Observation or \cite[Proposition 2.5]{cooper_long}. We record this finer conclusion for completeness, but we will not need it in this paper.
\end{remark}

\begin{claim} 
\label{claim:about_jordan_subspace_intersections}
$\Omega \cap \Pb(K_{\T \g})=\emptyset$, $\Pb(E_{\wt \g}) \cap \clOm \subset \bdry$ and
$\omega(\gamma,\Omega) \subset \Pb(E_{\wt \g}) \cap \bdry$. 
\end{claim} 
 \begin{proof}[Proof of Claim]
We first note that $\Omega \cap \Pb(K_{\T \g})=\emptyset$. Otherwise,  \cref{rem:translation_length_ineq} implies that 
\begin{align*}
\tau_{\Omega \cap \Pb(K_{\T \g})}(\g)=\log \left( \dfrac{\ev_{\max}(\T \g|_{K_{\T \g}})}{\ev_{\min}(\T \g|_{K_{\T \g}})} \right) < \log \left( \dfrac{\ev_{\max}(\T \g)}{\ev_{\min}(\T \g)} \right) = \tau_{\Omega}(\g),
\end{align*} 
a contradiction. Now suppose, if possible, that $\Pb(E_{\wt \g}) \cap \Omega$ is non-empty. By \cref{rem:translation_length_ineq}, $\tau_{\Omega}(\g) \leq \tau_{\Pb(E_{\wt \g}) \cap \Omega}(\g|_{E_{\wt \g}})=0$, a contradiction. Finally, since $\Omega \cap \Pb(K_{\T \g})=\emptyset$, $\Omega \subset \RP \setminus \Pb(K_{\T \g})$. Then \cref{rem:iteration_using_jordan_blocks} implies that $\omega(\g, \Omega) \subset \Pb(E_{\T \g}).$ Moreover $\omega(\g, \Omega) \subset \bdry$ by definition. 
  \end{proof}
  
 Note that these subspaces in \cref{defn:E_g} as well as the discussion above are independent of the lift $\T \g$ of $\g$ that we fix. Thus we introduce the following definitions.
  
\begin{definition}
\label{defn:E_L_K_plus}
\emph{If $\g \in \Aut(\Om)$, fix some (hence any) lift $\wt \g \in \GL $ of $\g$  that preserves the cone $\wt \Om$ above $\Om$ and define}
\begin{align*}
E_{\g}^+:= \Pb \left( E_{\T \g} \right),~ L_{\g}^+:= \Pb \left( L_{\T \g} \right)~ \text{ and } ~K_{\g}^+:= \Pb \left( K_{\T \g}\right), 
\end{align*}
\emph{where the subspaces $E_{\T \g}$, $L_{\T \g}$ and $K_{\T \g}$ are as in \cref{defn:E_g}. We also define}
\begin{align*}
E_{\g}^-:=E_{\g^{-1}}^+, ~ L_{\g}^-:=L_{\g^{-1}}^+ ~\text{ and }~ K_{\g}^-:=K_{\g^{-1}}^+.
\end{align*}
\end{definition}

\begin{remark}
\label{rem:L_minus}
Note that a linear subspace $V \subset \Rb^{d+1}$ is a real Jordan subspace for $\wt \g$ with eigenvalue $\mu$ if and only if $V$ is a real Jordan subspace for $\wt \g^{-1}$ with eigenvalue $\mu^{-1}$. Indeed, this follows because $\ker \left( \wt \g - \mu \id \right)^k=\ker\left( \wt \g^{-1}-\mu^{-1}\id \right)^k$ for any $k \in \Nb$. 
Thus , if $V_{\mu}$-s are the real Jordan subspaces for $\wt \g$ as above, then 
\[ 
E_{\wt \g^{-1}}=\bigoplus_{|\mu|=\ev_{\min}(\wt \g)}E_{\mu}, ~L_{\wt \g^{-1}}=\bigoplus_{|\mu|=\ev_{\min}(\wt g)} V_{\mu} ~\text{ and }~ K_{\wt \g^{-1}}=\bigoplus_{|\mu| > \ev_{\min}(\wt g)} V_{\mu}.
\]
\end{remark}

The key upshot of the discussion in this subsection is the following proposition. 
\begin{proposition}\label{prop:omega-limit-gamma}
If $\Omega$ is a Hilbert geometry, $\g \in \Aut(\Omega)$ and $\tau_\Om(\g) >0$, then:
\begin{enumerate}
\item $\omega(\g, \Omega) \subset E_{\g}^+$.
\item the action of $\g$ on $E_{\g}^+$ is conjugated into the projective orthogonal group ${\rm PO}(E_{\g}^+)$. 
\item there exists a sequence of positive integers $\{m_k\}$ with $m_k \to \infty$ such that 
\begin{equation*}
\lim_{k \to \infty} \Big( \g \Big|_{E_{\g}^+} \Big)^{m_k}=\id \Big|_{E_{\g}^+}.
\end{equation*}
\end{enumerate}
\end{proposition}
\begin{remark}  A similar proposition is true if we replace $\g$ by $\g^{-1}$ and $E_\g^+$ by $E_\g^-$. Moreover, it is possible that $\omega(\gamma,\Omega) \subsetneq E_{\g}^+ \subset \bdry$, see \cref{example:omega_limit_in_T2}. We finally remark that a finer conclusion is possible here: $\omega(\g,\Omega) \subset \Pb(F_{\wt \g}) \subset E_{\g}^+$, where $F_{\wt \g}$ is as defined in \cref{rem:attracted_to_most_powerful_jordan}. We will not need this finer conclusion, but we record it for completeness.
\end{remark}

\subsection{$\omega$-limit sets and faces in a properly convex domain }

We continue our discussion about $\omega$-limit sets from the previous subsection. Our goal now is to prove a result about the faces $F_{\Omega}(x)$ for  $x  \in E_\g^{\pm}.$  This result will be used in Section \ref{sec:contracting-implies-rank-one}. Before formulating the precise result, we give an illustrative example.

\begin{illustrative_example}
Let $g=\diag(\ev_1,\ev_2,\ev_3)$ where $\ev_1>\ev_2>\ev_3>0$ and  let $g$ preserve a properly convex domain $\Omega \subset \Pb(\Rb^3)$. Suppose $\pi(e_3) \in \bdry$ and let $F:=F_{\Omega}(\pi(e_3))$. We will show that $\pi(e_2) \not \in F$. Suppose, on the contrary, that $\pi(e_2) \in F$. Then $I_t:=[\pi(e_3-t e_2),\pi(e_3+te_2)] \subset F$ for some $t>0$. It is an elementary observation that $I_t$ gets expanded by the action of $g$ and $\overline{\cup_{k=1}^{\infty} g^k I_t}=\Pb(\Sp \{e_2,e_3\})$. Thus $\Pb(\Sp\{e_2,e_3\}) \subset \overline{F}  \subset \clOm$, which contradicts that $\Om$ is a properly convex domain. Thus $\pi(e_2) \not \in F$. By a similar reasoning $\pi(e_1) \not \in F$. 

The takeaway from this example should be the following: since $\pi(e_3)$ is an eigenvector corresponding to an eigenvalue of modulus $\ev_{\min}(g)$, the corresponding face $F_{\Om}(\pi(e_3))$ cannot intersect any eigenspace whose eigenvalue has modulus greater that $\ev_{\min}(g).$ The above philosophy works even if we replace eigenspaces by Jordan blocks and is the key idea behind the next result.
\end{illustrative_example}

We now state the precise version of the result. Recall the notation  $L_{\g}^{-}$ from previous section (cf. \cref{defn:E_L_K_plus} and \cref{rem:L_minus}): for any $\g \in \Aut(\Omega)$, $$L_{\g}^-= \Pb \left( \oplus_{|\mu|=\ev_{\min}(\wt \g)} V_{\mu} \right).$$
As in the previous subsection,  $\wt \g$ is some (hence any) lift of $\g$ and $V_{\mu}$ is the real Jordan subspace of $\wt \gamma$ for the eigenvalue $\mu$. Thus $L_{\g}^-$ is the direct sum of all the Jordan subspaces corresponding to the eigenvalues of $\wt \g$ of minimum absolute value.

\begin{lemma}
\label{lem:face_in_Lminus}
Suppose $\Omega$ is a Hilbert geometry and $\g \in \Aut(\Om)$ with $\tau_{\Omega}(\g) >0.$ If $y \in E_{\g}^-$, then $F_{\Omega}(y) \subset L_{\g}^-$.
\end{lemma}
\begin{proof}
Suppose, for a contradiction, that $v \in F_{\Omega}(y) - L_{\g}^-$. Fix a lift $\wt \g$ of $\g$. Since $y\in E_{\g}^-$, \cref{prop:omega-limit-gamma} part (3) implies that we can find a sequence $\{d_k\}$ of positive integers with $d_k \to  \infty$ such that $\left( \g|_{E_{\g}^-} \right)^{d_k} \to \id|_{E_{\g}^-}$.  

Up to passing to a subsequence of $\{d_k\}$, we can assume that $\g^{d_k} v \to v_{\infty} \in \overline{\Om}$. As $v \not \in L_{\g}^-$, \cref{obs:iteration_jordan_blocks} part (2) implies that there exists $c>\ev_{\min}(\wt \g)$ such that the accumulation points of $\left\{ \left( \frac{\wt \g}{c} \right) ^{d_k} v : k \geq 1 \right\}$ do not lie in $L_{\g}^-$. Thus $v_{\infty} \not \in L_{\g}^-$ and $\lim_{k \to \infty} \left( \dfrac{c}{\ev_{\min}} \right)^{d_k} = \infty$. We can then fix lifts $\wt y$, $\wt v$ and $\wt v_{\infty}$  such that
\[
\left( \dfrac{\wt \g}{\ev_{\min}(\wt \g)} \right)^{d_k} \wt y \to \wt y ~~\text{ and }~~ \left( \dfrac{\wt \g}{c} \right)^{d_k} \wt v \to \wt v_{\infty}.
\]
We claim that 
\[\Pb(\Sp\{y,v_{\infty}\}) \subset \clOm.\]
To prove this claim, it suffices to show that $\pi(\wt y+t \wt v_{\infty}) \in \clOm$ for any real number $t \neq 0$. Fix $0 \neq t \in \Rb$. Define $$s_k:=t \cdot \dfrac{\ev_{\min}^{d_k}}{c^{d_k}+t \ev_{\min}^{d_k}}.$$ Then $s_k \to 0$ as $k \to \infty$. In fact, for $k$ large enough, $s_k$ belongs to $(0,1)$ or $(-1,0)$ accordingly as $t>0$ or $t<0$. Set 
$$w_k:=\pi((1-s_k) \wt{y}+s_k \wt{v} )=\pi\left( \wt y + \dfrac{s_k}{1-s_k} \wt v \right)=\pi \left( \wt y + t \dfrac{\ev_{\min}^{d_k}}{c^{d_k}} \wt v \right),$$ 
since $\frac{s_k}{1-s_k}=t \dfrac{\ev_{\min}^{d_k}}{c^{d_k}}.$  Then $w_k \in \Pb(\Sp \{y,v\})$ and $\lim_{k \to \infty}w_k = y$. Thus, for $k$ large enough, $w_k \in F_{\Om}(y) \cap \Pb(\Sp \{y,v\})$ because $v \in F_{\Om}(y)$. Moreover, $w_k$ lies on opposite sides of $y$ in $F_{\Om}(y) \cap \Pb(\Sp \{y,v\}) $ accordingly as $t>0$ or $t<0$. Thus the following computation will show that $\g^{d_k}$ expands small neighborhoods of $y$ in $\Pb(\Sp \{y,v\}) \cap F_{\Om}(y)$ to large subintervals of the projective line $\Pb(\Sp\{y,v_{\infty}\})$. More precisely, we observe that
\begin{align*}
\lim_{k\to \infty} \g^{d_k} w_k &=\lim_{k \to \infty}\pi \left( (1-s_k)\frac{\wt \g^{d_k}}{\ev_{\min}(\wt \g)^{d_k}}\wt y + s_k \dfrac{\wt \g^{d_k}}{\ev_{\min}^{d_k}} \wt v  \right)  =\lim_{k \to \infty}\pi \left( \frac{\wt \g^{d_k}}{\ev_{\min}(\wt \g)^{d_k}}\wt y + \dfrac{s_k}{1-s_k}\dfrac{c^{d_k}}{\ev_{\min}^{d_k}} \dfrac{\wt \g^{d_k}}{c^{d_k}}\wt v \right) \\
&=\lim_{k \to \infty} \pi \left( \frac{\wt \g^{d_k}}{\ev_{\min}(\wt \g)^{d_k}}\wt y+t \frac{\wt \g^{d_k}}{c^{d_k}}\wt v \right)=\pi(\wt y + t \wt v_{\infty}).
\end{align*} 

Thus $\pi(\wt y+t \wt v_{\infty}) \in \clOm$ since $w_k \in F_{\Om}(y)$ for $k$ large enough. Since $t \neq 0$ is arbitrary, $\Pb(\Sp\{y,v_{\infty}\}) \subset \clOm.$ This proves the claim.

However, if $\clOm$ contains the non-trivial projective line $\Pb(\Sp\{y,v_{\infty}\})$, then $\Om$ cannot be properly convex.  This is a contradiction.
\end{proof}

\begin{corollary}\label{no-face-shared-by-E+-E-}
Suppose $\Om \subset \RP$ is a Hilbert geometry and $\g \in \Aut(\Om)$ with $\tau_{\Omega}(\g) >0.$ 
\begin{enumerate}
\item If $y \in E_\g^-$, then $F_{\Om}(y) \cap E_\g^+ = \emptyset.$ 
\item If $ y \in E_\g^-$, $z \in F_{\Omega}(y)$ and $\{i_k\}$ is a sequence in $\Zb$ such that $z_\infty:=\lim_{k \to \infty} \g^{i_k}z$ exists, then $z_{\infty} \in E_{\g}^-$.
\end{enumerate}
\end{corollary}
\begin{proof}
By \cref{lem:face_in_Lminus}, $F_{\Omega}(y) \subset L_{\g}^-$. Since $\tau_{\Omega}(\g)>0$, $L_{\g}^- \cap E_{\g}^+$ is empty by definition and this proves the first part. For the second part, note that  $z \in F_{\Omega}(y)$ implies that $z \in L_{\g}^-$. On $\Sp(L_{\g}^-)$, all eigenvalues of $\wt \g$ have the same modulus $\ev_{\min}(\wt \g)$. Then \cref{obs:iteration_jordan_blocks} part (3) implies that all accumulation points of $\{ \g^n z : n \in \Nb\}$ lie in $E_{\g}^-$. By a similar reasoning, all accumulation points of $\{ \g^{-n} z : n \in \Nb\}$ also lie in $E_{\g}^-$. This proves the second part. 
\end{proof}

\begin{remark}
\label{rem:analogues_of_face_results}
Analogues of \cref{lem:face_in_Lminus} and \cref{no-face-shared-by-E+-E-} hold for $F_{\Omega}(x)$ where $x \in E_\g^+$. One has to replace $\g$ with $\g^{-1}$ to obtain the analogous results.
\end{remark}

\part{Rank one Hilbert geometries}
\label{part:rank-one-hil-geom}

\section{Axis of Isometries}
\label{sec:axis}

\begin{definition}
\label{defn:axis}
Suppose $\Omega \subset \RP $ is a Hilbert geometry and $g \in \Aut(\Omega)$. An \emph{axis} of $g$ is a non-trivial projective line segment $\ell_g:=\Pb(V_g) \cap \Om$ where $V_g \leq \Rb^{d+1}$ is a two-dimensional $g$-invariant linear subspace.
\end{definition}

We will show that if $g$ has an axis and $\tau_{\Omega}(g)>0$, then $g$ acts by a translation along its axis $\ell_g$ and the endpoints of $\ell_g$ correspond to eigenvectors with eigenvalues of maximum and minimum modulus respectively. Recall the notation $E_g^+, E_g^-\subset \RP$ from \cref{defn:E_L_K_plus}.
\begin{lemma}
\label{lem:axis_inside}
Suppose $\Omega \subset \RP$ is a Hilbert geometry, $g \in \Aut(\Om)$ with $\tau_{\Omega}(g)>0$ and $g$ has an axis $\ell_g=\Pb(V_g) \cap \Omega$.  If $\wt g$ is a lift of $g$ in $\GL$, then:
\begin{enumerate}
\item $\wt g|_{V_g}$ has two distinct eigenvalues $\lambda_+ > \lambda_-$,
\item there exist $\wt{g_+}, \wt{g_-} \in \Rb^{d+1}$ such that $\wt g  \cdot \wt{g_{\pm}}=\ev_{\pm} \cdot \wt{g_{\pm}}$
and $\ell_g=(g_+,g_-)$ where $g_{\pm}=\pi(\wt{g_{\pm}})$,
\item $|\ev_+|=\ev_{\max}(\wt g)$, $|\ev_-|=\ev_{\min}(\wt g)$, $\tau_{\Omega}(g)= \log \left( \left| \dfrac{\ev_+}{\ev_-} \right| \right)>0$,
\item $g_+ \in E_g^+$ and $g_- \in E_g^-$.
\end{enumerate}
\end{lemma}
\begin{remark}
If the lift $\wt g$ preserves the cone $\wt \Om$ above $\Om$ and $\wt{g_{\pm}} \in \wt \Om$, then $\ev_{\pm}>0$, see \cref{lem:lift_has_pos_eigenvalues}. Then $\ev_{+}=\ev_{\max}(\wt g)$ and $\ev_-=\ev_{\min}(\wt g)$.
\end{remark}
\begin{proof}
Let $\ell_g=(a,b)$. Note that $g$ preserves the set  $\{a,b\}=\ell_g \cap \bdry$. Fix any lift $\wt{g}$ of $g$. In the basis $\{\Rb \cdot a,\Rb\cdot b\}$ of $V_g$, there exist $c_1,c_2 \neq 0$ such that $\wt g|_{V_g}$ is either 
$$\begin{pmatrix}
c_1 & 0 \\ 0 & c_2
\end{pmatrix} ~~\text{ or }~~ \begin{pmatrix}
0 & c_1 \\ c_2 & 0
\end{pmatrix}.$$
In the latter case, both eigenvalues of $\wt g|_{V_g}$ have the same modulus and $\tau_{\Omega \cap \Pb(V_g)}(g|_{\Pb(V_g)})=0$, by Proposition \ref{prop:translation_length}. But then \cref{rem:translation_length_ineq} implies $$0\leq \tau_{\Omega}(g) \leq \tau_{\Omega \cap \Pb(V_g)}(g|_{\Pb(V_g)})=0,$$ a contradiction. Thus we are in the former case and $g$ is diagonalizable with eigenvalues $c_1$ and $c_2$. Note that $c_1 \neq c_2$, since otherwise the same reasoning as above implies that $\tau_{\Omega}(g)=0$. Then set $\ev_+:=\max\{c_1,c_2\}$ and $\ev_-:=\min \{c_1,c_2\}$ and this proves part (1). For part (2), let $\wt{g_{\pm}}$ be the eigenvectors of  $\wt g$ in $V_g$ with eigenvalues $\ev_{\pm}$. Then note that by previous discussion, the set $\{\pi(\wt{g_+}),\pi(\wt{g_-})\}$ equals the set $\{a,b\}$. Thus $\ell_g=(\pi(\wt{g_+}),\pi(\wt{g_-}))$ and $\wt g|_{V_g}=\diag(\ev_+,\ev_-)$ in this basis.

For part (3), first note that \cref{rem:translation_length_ineq} implies $\tau_{\Om}(g)  \leq \tau_{\Om \cap \Pb(V_g)}\left(g\big|_{\Omega \cap \Pb(V_g)}\right)$. Proposition \ref{prop:translation_length} then implies that $\log \dfrac{\ev_{\max}}{\ev_{\min}} (\wt{g})  \leq \log \left| \dfrac{\ev_+}{\ev_-} \right|. $ Since $|\ev_+| \leq \ev_{\max}(\wt{g})$ and $|\ev_-| \geq \ev_{\min}(\wt{g})$, then $\left| \frac{\ev_+}{\ev_-} \right| \leq \frac{\ev_{\max}}{\ev_{\min}} (\wt{g})$. Thus
\begin{align*}
\left| \dfrac{\ev_+}{\ev_-} \right|=  \dfrac{\ev_{\max}}{\ev_{\min}} (\wt{g}).
\end{align*} 
Then, $|\ev_+|=|\ev_-|\cdot \dfrac{\ev_{\max}}{\ev_{\min}}(\wt{g}) \geq \ev_{\max}(\wt{g})$,
which implies that $|\ev_+|=\ev_{\max}(\wt{g})$. Similarly, $|\ev_-|=\ev_{\min}(\wt{g})$. This proves part (3). Then part (4) follows by definition of $E_g^+$ and $E_g^-$.
\end{proof}

\begin{corollary}
\label{cor:biproximal_with_axis_implies_unique_axis}
Suppose $g \in \Aut(\Om)$ with $\tau_{\Om}(g)>0$ and $g$ has an axis. If $\#(E_g^+)=\#(E_g^-)=1$, then $g$ has a unique axis given by $(E_g^+,E_g^-) \subset \Om$. In particular, if $g$ is biproximal (cf. \cref{sec:proximality}) and has an axis, then the axis of $g$ is unique. 
\end{corollary}
\begin{proof}
Immediate from \cref{lem:axis_inside} parts (2) and (4) and the hypothesis that $\#(E_g^+)=\#(E_g^-)=1$. For the in particular part, it suffices to note that if $g$ is biproximal, then $\#(E_g^+)=\#(E_g^-)=1$. 
\end{proof}
\begin{remark}
Although biproximality of $g$ implies $\#(E_g^+)=\#(E_g^-)=1$, its converse fails in general. For example, consider $g=\begin{pmatrix} 0.25 & 0 & 0 \\ 0 & 2 & 1 \\ 0 & 0 & 2 \end{pmatrix}$. However, we will show later in \cref{lem:axis_implies_proximal} that if $g$ has an axis, then $g$ is biproximal $\iff$ $\#(E_g^+)=\#(E_g^-)=1$. 
\end{remark}

An isometry $g \in \Aut(\Omega)$ may not have an axis (cf. \cref{ex:triangle} part (B) below). Hence we introduce the notion of a pseudo-axis.
\begin{definition} 
Suppose $\Omega \subset \RP$ is a Hilbert geometry and $g \in \Aut(\Om)$. A \emph{pseudo-axis} of $g$ is a non-trivial projective line segment $\sigma_g :=\Pb(W_g) \cap \clOm$ where $W_g \leq \Rb^{d+1}$ is a two-dimensional $g$-invariant linear subspace such that  $\Pb(W_g) \cap \Om =\emptyset$.
\end{definition}

\begin{observation}
\label{obs:axis}
If $\tau_{\Omega}(g)>0$, then either $g$ has an axis or a pseudo-axis.
\end{observation}

This observation is immediate from the following result of Benoist (also see \cite[Proposition 2.2]{marquis_survey}). Here $\T \Om$ denotes a cone above $\Om$ (cf. \ref{defn:cone_over_om} and the remark that follows).

\begin{proposition}[{\cite[Lemma 3.2]{benoist_cd3}}]
\label{prop:pseudo_axis}
Suppose $\Omega$ is a Hilbert geometry, $g \in \Aut(\Omega)$ and $\tau_{\Omega}(g)>0$. Let $\wt{g}$ be a lift of $g$ that preserves $\T \Om$. Then $\wt g$ has a real positive eigenvalue that equals $\ev_{\max}(\wt g)$ and there exists $v$ such that $\wt g \cdot v = \ev_{\max}(\wt g) \cdot v$ and $\pi(v) \in \clOm$. A similar result holds if we replace $\ev_{\max}(\wt g)$ by $\ev_{\min}(\wt g)$. 
\end{proposition}

\begin{remark}
If $\wt g$ is an arbitrary element of $\GL$, then $\ev_{\max}(\wt g)$ doesn't have to be an eigenvalue of $\wt g$. In fact, $\wt g$ may only have complex non-real eigenvalues of modulus $\ev_{\max}(\wt g)$. So the key point of the above proposition is that preserving the cone $\wt \Om$ above $\Om$ imposes a strong restriction, namely that $\wt g$ has a positive real eigenvalue that equals $\ev_{\max}(\wt g)$. 

However, the proposition does not imply anything about the number (or nature) of the other  eigenvalues whose modulus is $\ev_{\max}(\wt g)$. In \cref{ex:triangle} part (A), the matrix $g_2^{-1}$ has a repeated eigenvalue $\frac{1}{\ev_2}$ of maximum modulus. Moreover $\wt g$ can have complex eigenvalues of modulus $\ev_{\max}(\wt g)$, see \cref{ex:cone_over_H2}.
\end{remark}

We will now discuss a few examples to illustrate the notions introduced. An isometry may have a unique axis, infinitely many axes, or no axes at all. An isometry can have pseudo-axes without having an axis and vice versa.

\begin{example}
\label{ex:axis-H2}
\emph{(Unique axis, no pseudo-axes)}
Consider the Hilbert geometry $\Om:=\{[x:y:1] | x^2+y^2<1 \}$ in $\Pb(\Rb^3)$. It is the projective model of $\Hb^2$ and $\Aut(\Om)={\rm PO}(2,1)$. If $g \in \SO(2,1)$ has $\tau_{\Om}([g])>0$ (i.e. $g$ is a hyperbolic isometry in $\Isom(\Hb^2)$), then $[g]$ has a unique axis. 
\end{example}

\begin{example} 
\label{ex:triangle}
Consider the two-dimensional simplex $T_2:=\{[x_1:x_2:x_3] | x_1, x_2, x_3>0 \}$. 
\begin{enumerate}[label=(\Alph*)]
\item \emph{(Uncountably many axes, several pseudo-axes)} Let $g_2 :=[ \diag(\ev_1,\ev_2,\ev_2)]$ where $\ev_1 > \ev_2 > 0$. For $0 < t < 1$, let $Q_t:=([e_1],[te_2+(1-t)e_3])$.  Then, $\{Q_t\}_{t\in (0,1)}$ is an uncountable family of axes of $g_2$. There are three pseudo-axes: $[e_1,e_2]$, $[e_2,e_3]$, and $[e_1,e_3]$. 
\item \noindent \emph{(Several pseudo-axes, no axis)} Let $g_1 := [\diag(\ev_1,\ev_2,\ev_3)]$ where $\ev_1 > \ev_2 > \ev_3>0$. The pseudo-axes of $g_1$ are $[e_1,e_2]$,  $[e_2,e_3]$ and $[e_1,e_3]$. But $g_1$ does not have an axis.
\end{enumerate}
\end{example}

\begin{example}
\label{ex:cone_over_H2}
Let $\Om_2 \subset \Pb(\Rb^3)$ be the projective disk model of $\Hb^2$ and fix a cone $\wt \Om_2$ over $\Om_2$. Define $\Om_*:=\{ [v:x] \in \Pb(\Rb^4) : v \in \wt \Om_2, x>0\}$, i.e. $\Om_* \subset \Pb(\Rb^4)$ is the properly convex domain obtained by the join of $\Om_2$ with a point. Let $g:=\begin{bmatrix} \ev A & 0 \\ 0 & 1/\ev^3 \end{bmatrix} \in \Aut(\Om_*)$ where $\ev>1$ and $A=\begin{bmatrix} \cos(\theta) & -\sin(\theta) & 0 \\ \sin(\theta) & \cos(\theta) & 0 \\ 0 & 0 & 1\end{bmatrix} \in {\rm SO}(2,1).$ Then $g$ has three eigenvalues $\ev,\ev e^{\pm i \theta}$ of maximum modulus. 

\noindent Note that $g$ has an axis $\ell_g:=(\pi(e_3),\pi(e_4)) \subset \Om_*$. The action of $g$ is by a translation along $\ell_g$ and a rotation (by angle $\theta$) around $\ell_g$. The axis $\ell_g$ is contained in properly embedded triangles in $\Om_*$. 
\end{example}

\subsection{Three Key Lemmas}
We conclude this section by establishing three lemmas that will be used in the next section. The first one is a consequence of Lemma \ref{lem:axis_inside}.

\begin{lemma}
\label{evector_in_bdry_2}
Suppose $\Omega \subset \RP$ is a Hilbert geometry, $g \in \Aut(\Om)$ with $\tau_{\Om}(g)>0$ and $a,b$ are fixed points of $g$ with $a \in E_g^+$ and $b\in E_g^-$. If $c$ is a fixed point of $g$ such that $c \in\clOm -( E_g^+ \cup E_g^-)$, then $[a,c] \cup [b,c] \subset \bdry$.
\end{lemma}
\begin{proof}
First observe that $c \in \bdry$. Otherwise, $\tau_{\Omega}(g)=d_{\Omega}(c,gc)=0$, a contradiction. 
Suppose $(a,c) \subset \Omega$. Then $(a,c)$ is an axis of $g$ with $a \in E_g^+$. Lemma \ref{lem:axis_inside} then implies that $c \in E_g^-$, a contradiction. Thus $[a,c] \subset \bdry$. A similar reasoning implies that $[c,b] \subset \bdry$.
\end{proof}

The next lemma shows that if $g \in \Aut(\Om)$, $\tau_{\Om}(g)>0$, $g$ has an axis $(a,b)$ and $\#(E_g^+)>1$, then  $F_{\Om}(a)$ contains a non-trivial projective line segment in $\bdry$. Before formulating the precise result and its proof, let us give an intuitive explanation of the main idea. Suppose $u \neq a \in E_g^+$ and let $\xi$ be a point in $(a,b) \subset \Om$. As $\Om$ is open, we can find a point $\xi' \in \Om \cap \Pb(\Sp\{\xi,u\})$ that is distinct from $\xi$. Then, up to extracting a suitable subsequence of $\{ g^n \}$, $g^{n_k}\xi \to a$ while $g^{n_k} \xi' \to a'$. As $\xi\neq \xi'$ and $u,a \in E_g^+$, one can check that $a'\neq a$ (see the proof below). Then a property of the Hilbert metric (\cref{line-in-bdry-0}) implies that $a' \in F_{\Om}(a)$. This is the gist of the proof below.

\begin{lemma} \label{evector_in_bdry_1}
Suppose $\Omega \subset \RP$ is a Hilbert geometry, $g \in \Aut(\Om)$ with $\tau_{\Om}(g)>0$ and $g$ has an axis $(a,b)$ where $a \in E_g^+$ and $b \in E_g^-$. If $u \in E_{g}^+ \setminus \{a\}$, then there exist $x_u^- \neq x_u^+ \in \bdry$ such that $a \in (x_u^-,x_u^+)$ and 
\begin{align*}
F_{\Om}(a) \cap \Pb(\Sp\{a,u\}) =(x^-_{u},x^+_{u}).
\end{align*}
\end{lemma}

\begin{remark} Suppose we have the same setup as \cref{evector_in_bdry_1}. By a similar reasoning, if $v \in E_g^-\setminus \{b\}$, then $F_{\Om}(b) \cap \Pb(\Sp\{b,v\})=(x_v^-,x_v^+)$ where $x_v^- \neq x_v^+$. 
\end{remark}

\begin{proof}
Let us fix a cone $\T \Om$ over $\Om$. Then, we fix lifts $\T g, \T a, \T b, \T u$ of $g, a, b, u$ such that $\T a, \T b, \T u \in \T \Om$ and $\T g \cdot \T \Om =\T \Om$. Note that $\T a$ is an eigenvector of $\T g$ corresponding to the eigenvalue $\evmax(\T g)$ or $-\evmax(\T g)$. Since $\T g$ preserves $\T \Om$, \cref{lem:lift_has_pos_eigenvalues}  implies that $\T g \cdot \T a = \evmax(\T g) \cdot \T a.$ Similarly, $\T g \cdot \T b=\evmin(\T g) \cdot \T b$. Since $u \in E_{g}^+$, Proposition \ref{prop:omega-limit-gamma} part (3) implies that there exists an unbounded sequence of positive integers $\{m_k\}$  such that 
\begin{align*}
\left( \dfrac{\T g}{\evmax(\T g)} \right)^{m_k} \T u= \T u.
\end{align*}
For $t \in \Rb$, let $\T{p}_t:=\dfrac{\T{a}+\T{b}}{2}+t \T{u}$ and $p_t:=\pi(\T p_t)$. Since $(a,b)$ is an axis, $p_0 \in \Om$. Then, as $\Om$ is an open set, there exists $\e_0>0$ such that $\T{p_t} \in \T{\Om}$ for all $t \in (-\e_0,\e_0).$  Fix $t \in (-\e_0,\e_0)$. Then
\begin{align*}
\lim_{k \to \infty} g^{m_k}p_t &=\lim_{k \to \infty} \pi \Bigg( \Big( \dfrac{\T{g}}{\evmax(\T g)}\Big) ^{m_k} \T{p}_t \Bigg)= \lim_{k \to \infty} \pi \Bigg( \dfrac{\T{a}}{2} +  \Big( \dfrac{\evmin(\T g)}{\evmax(\T g)} \Big)^{m_k}\dfrac{\T{b}}{2} + t \Big( \dfrac{\T{g}}{\evmax(\T g)}\Big)^{m_k} \T{u} \Bigg) \\
&=\pi \Big( \T a + 2t  \T u \Big) \in \clOm.
\end{align*}
Then, $\lim_{k \to \infty} g^{m_k}p_0=a$ and  $\lim_{k \to \infty} g^{m_k}p_t \neq a$ whenever $t \neq 0$. By  Proposition \ref{line-in-bdry-0}, 
\[
\lim_{k \to \infty} g^{m_k}p_t \in F_{\Omega}(a)
\] because $\lim_{k\to \infty}\hil(g^{m_k}p_0,g^{m_k}p_t)=\hil(p_0,p_t)$. 
Thus there exist $x_u^+ \neq x_u^- \in \bdry$ such that $F_{\Om}(a) \cap \Pb(\Sp\{a,u\})) =(x^-_{u},x^+_{u})$.
\end{proof}

The next lemma shows that if $\g \in \Aut(\Om)$ has an axis and $\# (E_{\g}^-)=1$, then $\g^{-1}$ is a proximal element in $\PGL$ (see \cref{sec:proximality}). Before stating the precise version of the result, we give an illustrative example to explain the main idea behind it.

\begin{illustrative_example} Let $\mu>\ev>0$. Suppose $g=\begin{pmatrix}\mu & 0 & 0 \\ 0 & \ev & 1 \\ 0 & 0 & \ev \end{pmatrix}$ preserves $\Om \subset \Pb(\Rb^3)$ and $\pi(e_1),\pi(e_2) \in \bdry$. Here $g$ satisfies $\#(E_g^-)=1$ but $g^{-1}$ is not proximal. The main takeaway from this example will be that such a matrix $g$ cannot have an axis in $\Om$, i.e. the only candidate for an axis, namely $(\pi(e_1),\pi(e_2))$, cannot lie in $\Om$. 

To proceed, we will first explain that $\pi(e_3)$ cannot lie in $\clOm$. For this, first note that $$g^{\pm k}\pi(e_3)=\pi(k\ev^{k-1}e_2 + \ev^ke_3).$$ Hence $g^{\pm k} \pi(e_3) \to \pi(e_2)$ as $k \to \infty$, but they approach $\pi(e_2)$ from `opposite directions' in the projective line $\Pb(\Sp \{ e_2,e_3\})$. That is, $g^{\pm k}$ `wraps' $[g^{-1}\pi(e_3),g \pi(e_3)]$ around $\Pb(\Sp \{ e_2,e_3\})$. Then, $\pi(e_3) \in \clOm$ will imply that $\Pb(\Sp \{ e_2,e_3\}) \subset \clOm$, which is a contradiction as $\Om$ is a properly convex domain. Thus $\pi(e_3) \not \in \clOm$.

Now we revisit our basic proposition: $(\pi(e_1),\pi(e_2))$ cannot lie in $\Om$. Suppose this is false and $(\pi(e_1),\pi(e_2)) \subset \Om$. Since $g^k(\pi(e_1),\pi(e_3)) \to (\pi(e_1),\pi(e_2))$,  we can find $\pi(y_k) \in \Om \cap (\pi(e_1),\pi(e_3))$ such that $g^k\pi(y_k)$ converges to the midpoint of $(\pi(e_1),\pi(e_2))$. Now unless $\pi(y_k) \to \pi(e_3)$, one can use the action of $g$ to show that $g^k\pi(y_k) \to \pi(e_1)$, a contradiction (see the computation in \cref{eqn:p_equals_a}). Thus $\pi(y_k)\to \pi(e_3)$ and hence $\pi(e_3) \in \clOm$. This contradicts the previous paragraph.
\end{illustrative_example}

The argument discussed above is the gist of the proof below. We now precisely formulate and prove our result.

\begin{lemma}
\label{lem:axis_implies_proximal}
Suppose $\Om \subset \RP$ is a Hilbert geometry, $\g \in \Aut(\Om)$ with $\tau_{\Om}(\g)>0$ and $\g$ has an axis. If $\# (E_{\g}^-)=1$, then $\left| \dfrac{\ev_d}{\ev_{d+1}}(\g) \right|>1$. 
\end{lemma}
\begin{remark}
Similar reasoning with $\g$ replaced by $\g^{-1}$ implies: if $\#(E_{\g}^+)=1$, then $\left| \frac{\ev_1}{\ev_{2}}(\g) \right|>1$.
\end{remark}

\begin{proof}
Suppose the axis of $\g$ is $(a,b)$ with $a \in E_{\g}^+$ and $b \in E_{\g}^-$. Let us fix $\T \Om$, a cone above $\Om$. Fix lifts $\T \g$, $\T a$ and $\T b$ where $\T a, \T b \in \T \Om$ and $\T \g \cdot \T \Om=\T \Om$. Set $\evmax:=\evmax(\T\g)$ and $\evmin:=\evmin(\T\g)$. Since $b \in E_\g^-$ is a fixed point and $\T b \in \T \Om$, \cref{lem:lift_has_pos_eigenvalues} implies that $\T \g \cdot \T b = \evmin \cdot ~\T b.$ Similarly, $\T \g \cdot \T a = \evmax \cdot ~\T a.$

Since $\#(E_{\g}^-)=1$, there is a one-dimensional eigenspace of $\wt \g$  (namely $\Rb \T b$) and a single Jordan block $J_{\min}$ corresponding to eigenvalues of modulus $\evmin$ (immediate from the definition, cf. \ref{defn:E_L_K_plus}). Thus in order to prove $\left|\dfrac{\ev_d}{\ev_{d+1}}(\g) \right|>1$, it is enough to show that the Jordan block $J_{\min}$ has size 1. Suppose this is false. Then there exists $\T w \in \Rb^{d+1}$ such that if $k\in \Zb$, then
\begin{align}
\label{eqn:gamma-k-w}
\wt{\g}^k\wt{w}=k\ev_{\min}^{k-1}\wt{b}+\ev_{\min}^k \wt{w}. 
\end{align}
Setting $w:=\pi(\wt{w})$, $\lim_{k \to \infty} \g^k w = b$. Since $\g^ka=a$ for all $k$, $\lim_{k \to \infty} \g^k[a,w] = [a,b]$. Fix $p \in (a,b) \subset \Omega$. Then there exist  $y_k \in (a,w)$ such that 
\begin{align}
\label{eqn:yk-to-p}
\lim_{k \to \infty} \g^ky_k = p.
\end{align}
Since $p \in \Omega$ and $\Omega$ is open, $\g^ky_k \in \Omega$ for $k$ large enough. Thus, up to truncating finitely many terms of the sequence $\{y_k\}$, we can assume that for $k \geq 1$,  $$y_k \in (a,w)\cap \Omega.$$ We can fix lifts $\T y_k$ of $y_k$ in $\T \Om$ such that 
\begin{align}
\label{eqn:yk-ck-dk}
\wt y_k=c_k\wt{a}+d_k\wt{w}
\end{align} 
where $c_k,d_k \in [0,1]$. Thus, up to passing to a subsequence, we can assume that $c_\infty:=\lim_{k \to \infty} c_k$ and $d_\infty:=\lim_{k \to \infty} d_k$ exist. Then $\T{y_{\infty}}:=\lim_{k\to\infty} \T y_k$ exists and we set 
\begin{align*}
y_\infty:=\pi(\T{y_\infty})=\pi(c_\infty \T a + d_\infty \T w).
\end{align*}
We now claim that $y_\infty=\pi(\T w)=w$. If this is not true, then $c_\infty \neq 0$. Then, there exists $k_0 \in \Nb$ such that $c_k>(c_{\infty}/2)$ for all $k > k_0$ and $\lim_{k \to \infty} (d_k/c_k)=d_\infty/c_{\infty}$ exists in $\Rb$. Then using equation \eqref{eqn:yk-to-p} followed by \eqref{eqn:yk-ck-dk} and \eqref{eqn:gamma-k-w}, 
\begin{align}
\label{eqn:p_equals_a}
p=\lim_{k \to \infty} \g^k y_k &=\lim_{k \to \infty} \pi \left(  \dfrac{\wt{\g}^k \wt{y_k}}{c_k\ev_{\max}^k}  \right) \\
&= \lim_{k \to \infty} \pi  \left( \wt{a} + \dfrac{d_k}{c_k}  \left( \dfrac{k}{\ev_{\max}} \Big(\dfrac{\ev_{\min}}{\ev_{\max}} \Big)^{k-1} ~\wt{b} + \Big(\dfrac{\ev_{\min}}{\ev_{\max}} \Big)^{k} \wt{w} \right) \right) \nonumber \\
&=\pi (\wt{a})=a. \nonumber
\end{align}
This is a contradiction since $p \in \Omega$ while $a\in \bdry$. Thus $y_\infty=w$. 

Since $y_k \in \Om$ for $k\geq 1$, $w=y_\infty$ implies that $w \in \clOm$. Then for all $k \in \Zb$, 
\begin{equation}
\label{eqn:w_gw_in_omega}
[w, \g^kw] \subset \clOm.
\end{equation}
We now show that this implies $\Pb(\Sp\{w,b\})  \subset \clOm$. 
For $t>0$, let 
\begin{align*}
\Hc_t:=\left\{ \pi \left( \wt{w}+ r \wt{b} \right) : -t \leq r \leq t \right\}.  
\end{align*}
Then $\overline{\bigcup_{t >0} \Hc_t}=\Pb(\Sp\{b,w\}).$ Now observe that if $k \in \Zb$, then equation \eqref{eqn:gamma-k-w} implies that 
\begin{align*}
\g^k w=\pi \left( \dfrac{\T \g^k \T w}{\ev_{\min}^k} \right)=\pi \left( \wt{w}+\dfrac{k}{\ev_{\min}}\wt{b} \right).
\end{align*} 
Then, for every $t>0$, there exists $k_t \in \Nb$ such that $\Hc_t \subset [\g^{-(k_t-1)}w,w] \cup [w,\g^{k_t}w]$. Then, by equation \eqref{eqn:w_gw_in_omega}, $\Hc_t \subset \clOm$ for any $t>0$.  Thus $\Pb(\Sp\{w,b\}) =\overline{\bigcup_{t>0} \Hc_t} \subset \clOm.$ This is a contradiction as $\Om$ is a properly convex domain and hence $\clOm$ cannot contain a projective line.
\end{proof}

\section{Rank one Isometries: Definition and Properties}
\label{sec:rank-one-isom}

In this section, we introduce the notion of rank one isometries for Hilbert geometries. Our definition is analogous to the definition of rank one isometries for $\CAT(0)$ spaces \cite{ballmann_axial_isometries, ballmann_orbihedra}. The notion of half triangles that we introduce is analogous to the notion of half flats used in the $\CAT(0)$ setting. Refer to Figure \ref{fig:rank_one} for the next two definitions.

\begin{definition}\label{defn:half-T}
 Suppose $\Om\subset \RP$ is a Hilbert geometry.  Then the points $x, z, y \in \bdry$ form a \emph{half triangle in} $\Om$ if
 \begin{align*}
 [x,z] \cup [y,z] \subset \bdry \text{ and } (x,y) \subset \Omega.
 \end{align*}
For $x,y \in \bdry$, we will say that the projective geodesic $(x,y) \subset \Omega$ is \emph{not contained in any half triangle in $\Om$} if for any $z \in \bdry$, either $(x,z) \subset \Om$ or $(z,y) \subset \Om$. 
\end{definition}

\begin{definition}
\label{defn:rank_one_geodesic}
Suppose $\Om \subset \RP$ is a Hilbert geometry and $a,b \in \bdry$. The projective geodesic $(a,b)$ is a \emph{rank one geodesic} provided $(a,b) \subset \Om$ is not contained in any half triangle in $\Om$. 
\end{definition}

We now define rank one isometries for Hilbert geometries. An isometry in $\Aut(\Om)$ is rank one if it acts by a translation along a rank one geodesic (cf. \cref{fig:rank_one}).

\begin{definition} \label{defn:rank-one-isometry} 
Suppose $\Om\subset \RP$ is a Hilbert geometry. 
\begin{enumerate}
\item An element $\g \in \Aut(\Om)$ is a \emph{rank one isometry} if:
\begin{enumerate}
\item $\tau_{\Omega}(\g)=\log \left|\dfrac{\ev_1}{\ev_{d+1}}(\g)\right|>0$,
\item $\g$ has an axis, 
\item none of the axes $\ell_\g$ of $\g$ are contained in a half triangle in $\Om$.
\end{enumerate}
\item A bi-infinite projective geodesic $\ell \subset \Om$ is a \emph{rank one axis} if $\ell$ is the axis of a rank one isometry in $\Aut(\Om).$
\end{enumerate}
\end{definition}

\begin{remark}
\label{rem:rank_one_axis_vs_axis}
The prototypical example of a rank one isometry is a hyperbolic isometry $[\diag(\lambda,1,1/\lambda)]$ ($\ev>1$) in $\Isom (\Hb^2)$, see \cref{ex:axis-H2}. On the other hand, any element in $\Aut(T_d)$, where $T_d$ is a $d$-dimensional simplex, is a non-example. In fact, this non-example highlights the necessity of the half triangle condition in the definition of a rank one isometry, as we now explain.  Recall \cref{ex:triangle} part (A). In that example, $g_2=[\diag(\lambda_1,\lambda_2,\lambda_2)]$ has an axis $Q_t$ for each $0<t<1$ and $\tau_{T_2}(g_2)>0$. But all of these axes are contained in the projective triangle $T_2$ (and hence a half triangle). For another non-example, see \cref{ex:cone_over_H2}.
\end{remark}

Recall Definition \ref{defn:rank-one-hil-geom}: a \emph{rank-one Hilbert geometry} is a pair $(\Om,\G)$ where $\Om$ is a Hilbert geometry and $\G \leq \Aut(\Om)$ is a discrete subgroup that contains a rank one isometry. In Appendix \ref{sec:eg-rank-one}, we discuss examples and also a generalization for convex co-compact groups. 

We will now establish some key geometric and dynamical properties of rank one isometries. The essence here is that translating along a rank one axis is much more special than translating along any axis and \cref{prop:rank-one-properties} could be interpreted as strengthening \cref{lem:axis_inside} under the rank one assumption. Our results are reminiscent of Ballmann's results in rank one Riemannian non-positive curvature \cite{ballmann_axial_isometries, ballmann_book}. Recall the notation $E_{g}^{\pm}$ from \cref{defn:E_L_K_plus} and the notion of proximality from \cref{sec:proximality}.

\begin{proposition}\label{prop:rank-one-properties}
Suppose $\Om$ is a Hilbert geometry and $\g \in \Aut(\Om)$ is a rank one isometry with an axis $\ell_{\g}=(a,b)$ where $a \in E_{\g}^+$ and $b\in E_{\g}^-$. Then: 
\begin{enumerate}
\item $\g$ is biproximal,
\item $\ell_{\g}$ is the unique axis of $\g$ in $\Omega$, 
\item the only fixed points of $\g$ in $\clOm$ are $a$ and $b$, 
\item if $z' \in  \bdry \setminus \{ a,b \}$, then $(a,z') \cup (b,z') \subset \Om$ (cf. \cref{fig:rank_one}),
\item if $z \in \bdry \setminus \{ a,b\}$, then neither $(a,z)$ nor $(b,z)$ is contained in a half triangle in $\Om$.
\end{enumerate}
\end{proposition}
\begin{remark}
\label{rem:ns_dynamics}
If $\g$ is a  rank one isometry, then the above proposition shows that $\#(E_{\g}^{\pm})=1$ and we will henceforth use the notation $\g^{\pm}:=E_{\g}^{\pm}$. We will call $\g^+$ the \emph{attracting fixed point of $\g$} and $\g^-$ the \emph{repelling fixed point of $\g$.}  We choose this terminology because $\g$ has \emph{north-south dynamics} on $\bdry$, see \cref{cor:ns_dynamics}.
\end{remark}
\begin{proof}
Let us fix $\T \Om$, a cone  over $\Om$. For the rest of this proof, fix lifts $\T \g$, $\T a$ and $\T b$ where $\T a, \T b \in \T \Om$ and $\T \g \cdot \T \Om=\T \Om$. Set $\evmax:=\evmax(\T\g)$ and $\evmin:=\evmin(\T\g)$. Since $a\in E_{\g}^+$ is a fixed point of $\g$, then $\T a$ is an eigenvector of $\wt \g$ corresponding to the eigenvalue $\evmax$ or $-\evmax$. By \cref{lem:lift_has_pos_eigenvalues},  $$\T \g ~\cdot~ \T a = \evmax ~\cdot ~ \T a.$$ Similarly, $\T \g ~ \cdot ~ \T b = \evmin ~ \cdot ~ \T b$. 

\medskip

\noindent (1) By the hypothesis, $\#(E_{\g}^{\pm}) \geq 1$. In order to prove that $\g$ is biproximal, we first prove that: 
\begin{claim}
\label{claim:unique-eigenvalue}
$\#(E_{\g}^+)=\# (E_{\g}^-)=1$.
\end{claim}
\noindent \emph{Proof of Claim. }
It suffices to prove the claim for $E_{\g}^+$ since the same arguments with $\g$ replaced by $\g^{-1}$ implies the result for $E_{\g}^-$. Now suppose the claim is false and there exists $u \in E_{\g}^+ \setminus \{a\}$. Then Lemma \ref{evector_in_bdry_1} implies that there exist $z^-, z^+ \in \bdry$ such that $a \in (z^-,z^+)$ and 
\begin{align*}
F_{\Om}(a) \cap \Pb(\Sp\{a,u\})) =(z^-,z^+).
\end{align*}
Then, $\Ic_z:=\big[ z_-,z_+ \big]$ is the maximal projective line segment in $\bdry$ containing both $z_-$ and $z_+.$

Since $\g$ is a rank one isometry, its axis $(a,b)$ cannot be contained in a half triangle in $\Om$. But $[a,z_+] \subset \bdry$ which implies that $(z_+,b) \subset \Om$. Similarly, $(z_-,b)\subset \Om.$ Choose $x_{\pm} \in (z_{\pm},b) \cap \Om $. By Proposition \ref{prop:omega-limit-gamma} part (3), there exists a sequence $\{m_k\}$ of positive integers with $m_k \to \infty$ such that 
\begin{align*}
\lim_{k \to \infty} \left( \g \big|_{E_{\g}^+} \right)^{m_k}=\id_{E_{\g}^+}.
\end{align*} 
Since $z_+ \in \Pb(\Sp\{a,u\})$, $z_+ \in E_{\g}^+$. Fix a lift $\wt{z_+} \in \T \Om$ of $z_+$. Then $\lim_{k \to \infty} \left( \frac{\T{\g}}{\ev_{\max}} \right)^{m_k}\T{z_+}=\T{z_+}.$

On the other hand, $$\lim_{k \to \infty}\left( \frac{\T \g}{\ev_{\max}}\right)^{m_k} \T{b}=\lim_{k \to \infty}\left( \frac{\ev_{\min}}{\ev_{\max}}\right)^{m_k} \T{b} = 0$$ as $\ev_{\max} > \ev_{\min}$. Then, since $x_+ \in (z_+,b)$, 
$$\lim_{k \to \infty} \g^{m_k}x_+ = z_+.$$ 
Similarly $$\lim_{k \to \infty} \g^{m_k}x_- = z_-.$$  
Since $\lim_{k \to \infty} \hil(\g^{m_k}x_+,\g^{m_k}x_-) = \hil(x_+,x_-)$, Proposition \ref{line-in-bdry-0} implies that $z_+  \in F_{\Omega}(z_-)$. Thus there is an open projective line segment in $\bdry$ containing both $z_+$ and  $z_-$. This contradicts the maximality of $\Ic_z$ and finishes the proof of Claim \ref{claim:unique-eigenvalue}.

By the above claim $\#(E_{\g}^+)=\#(E_{\g}^-)=1$ where $\tau_{\Om}(\g)>0$ and $\g$ has an axis $(a,b)$. Then Lemma \ref{lem:axis_implies_proximal} implies that $\g$ is biproximal.

\noindent (2)  This follows from biproximality of $\g$ and \cref{cor:biproximal_with_axis_implies_unique_axis}.

\noindent (3)  Suppose $c$ is a fixed point of $\g$ in $\bdry$ that is distinct from both $a$ and $b$. By part (1) of this Proposition, $\g$ is biproximal. Thus $c \not \in E_{\g}^+ \cup E_{\g^-}$. Then,  by Lemma \ref{evector_in_bdry_2}, $[a,c] \subset \bdry$ and $[b,c] \subset \bdry.$ Thus, the axis $\ell_{\g}=(a,b)$ of $\g$ is contained in a half triangle, contradicting that $\g$ is a rank one isometry.

\noindent (4)  Let $v \in \bdry \setminus \{a, b\}$. Then $v \not \in \Pb(\Sp\{a,b\})$ as $(a,b) \subset \Om$. Suppose $[a,v] \subset \bdry$. Since $\g$ is biproximal, there exists a $\g$-invariant decomposition of $\Rb^{d+1}$ given by: $$\Rb^{d+1}=\Rb \T{a} \oplus \Rb \T{b} \oplus \T{E}.$$ Choose any lift $\T v$ of $v$ in $\T{\Om}$. Then $\T v$ decomposes as 
\begin{align*}
\T{v}=c_1 \T{a} + c_2 \T{b} + \T{v_0}
\end{align*}
where $c_1, c_2 \in \Rb$ and $\T{v_0} \neq 0$. If $c_2 \neq 0$, then $\lim_{n \to \infty} \g^{-n}v = b$, that is, $\lim_{n \to \infty}\g^{-n}[a,v] = [a,b]$. Since $[a,v] \subset \bdry$ by assumption, $[a,b] \subset \bdry.$ This is a contradiction since $(a,b) \subset \Om$. Thus, $c_2=0$. 

Set $\ev_{\T E}:= \ev_{\max}(\T{\g}\big|_{\T E})$. Since $\g$ is biproximal,  $\ev_{\T E} < \evmax$. Then, for every $n >0$, 
\begin{align*}
\Big( \dfrac{\T{\g}}{\ev_{\T E}} \Big)^{-n}\T{v}=c_1 \Big( \dfrac{\evmax}{\ev_{\T E}}\Big)^{-n} \T a + \Big( \dfrac{\T{\g} \big|_{\T E}}{\ev_{\T E}} \Big)^{-n} \T{v_0}.
\end{align*}
Up to passing to a subsequence, we can assume that $v_{\infty}:= \lim_{n \to \infty} \g^{-n}v$ exists. Note that $v_\infty \in \clOm ~\cap~ \Pb\big( \T E \big).$ But $\clOm \cap \Pb(\wt E)$ is a $\g$-invariant non-empty convex compact subset of $\Rb^{d-1}$ and  Brouwer's fixed point theorem implies that $\g$ has a fixed point in $\clOm \cap \Pb(\wt E)$. But $\clOm \cap \Pb(\wt E) \cap \{ a, b\} =\emptyset$. This contradicts part (3). Hence, $(a,v) \subset \Om$. Similarly we can show that $(b,v) \subset \Om$.

\noindent (5) This is a consequence of part (4).
\end{proof}

\begin{corollary}
\label{cor:ns_dynamics}
Suppose $\g \in \Aut(\Om)$ is a rank one isometry. Then $\g$ has \emph{north-south dynamics} on $\bdry$, that is, 
$$\left( \g\Big|_{\clOm-\{\g^{\mp}\}} \right)^{\pm n} \to \g^{\pm}$$
as $n \to \infty$ and the convergence is locally uniform on compact subsets of $\clOm-\{ \g^{\mp}\}$.
\end{corollary}
\begin{proof}
The proof is very similar to part (4) of \cref{prop:rank-one-properties}. By the above proposition, $\g$ is biproximal. Thus there exists a $\g$-invariant decomposition $\Rb^{d+1}=\Rb \g^+ \oplus H_{\g} \oplus \Rb \g^-$ where $\g^{\pm}=E_{\g}^{\pm}$. Moreover, $\g^n$ converges to the constant map $\g^+$ locally uniformly on compact subsets of $\RP - \Pb(H_{\g} \oplus \Rb \g^-)$ as $n \to \infty$. 

We claim that $\Pb(H_{\g} \oplus \Rb \cdot \g^-) \cap \clOm=\{\g^- \}$. If the claim is false, then pick $v \in \Pb(H_{\g} \oplus \Rb \cdot \g^-) \cap \clOm$ such that $v \neq \g^-$. Up to passing to a subsequence, we can assume that $v_\infty=\lim_{n \to \infty} \g^n v$ exists in $\clOm$. Since $v \in \Pb(H_{\g} \oplus \Rb \g^-)-\{\g^-\}$, similar reasoning as in part (4) implies that $v_\infty \in \Pb(H_{\g})$. Thus $v_{\infty} \in \clOm \cap \Pb(H_{\g})$. Again, as in part (4), Brouwer's fixed point theorem will imply the existence of a fixed point of $\g$ in $\clOm \cap \Pb(H_{\g})$ which is distinct from $\g^{\pm}$. This contradicts  \cref{prop:rank-one-properties} part (3). This finishes the proof of the claim.

By the claim and the first paragraph of the proof, $\g^n$ converges to the constant map $\g^+$ locally uniformly on compact subsets of $\clOm - \{\g^-\}$ as $n \to \infty$. The proof for $\g^{-n}$ follows by  similar reasoning.
\end{proof}

Now we prove a simpler characterization of rank one isometries for co-compact actions.

\begin{proposition}\label{prop:biprox-equiv-no-half-T} 
Suppose $\Om$ is a Hilbert geometry, $\G \leq \Aut(\Omega)$ is a discrete subgroup that acts co-compactly on $\Omega$ and $\g \in \G$ where $\tau_{\Om}(\g)>0$. If $\g \in \G$ has an axis,  then the following are equivalent:
\begin{enumerate}
\item $\g$ is biproximal.
\item none of the axes of $\g$ are contained in a half triangle in $\Om$.
\item $\g$ is a rank one isometry.
\end{enumerate}
\end{proposition}
\begin{proof} 
Note that $(2) \iff (3)$ is by definition (cf. \ref{defn:rank-one-isometry}) and $(3) \implies (1)$ is Proposition \ref{prop:rank-one-properties} part (1). We will prove $(1) \implies (2)$, under the assumption that $\Omega/ \Gamma$ is compact.

\noindent Let $(a,b)$ be the axis of $\g$ with $a \in E_{\g}^+$ and $b \in E_{\g}^-$. We first show that $\g$ has no other fixed points in $\bdry$ except $a$ and $b$. If this is not true, let $v$ be such a fixed point of $\g$. Since $\g$ is biproximal, $v \not \in E_{\g}^+ \cup E_{\g}^-$. Then Lemma \ref{evector_in_bdry_2} implies that 
\begin{align}
\label{eqn:avb_in_bdry}
[a,v] \cup [v,b]  \subset \bdry.
\end{align} 
Since $(a,b) \subset \Om$, $\CH_{\Omega}\{a,v,b\}$ is a non-empty set.

 Let $A_\g:=\langle \g \rangle$. Recall the notation ${\rm Min}_{\Om}(A_\g)=\bigcap_{h \in A_\g}\{ x \in \Om : \hil(x, h \cdot x)=\tau_{\Om}(h) \}$ from \cref{subsec:min_translation}. \cref{lem:min_trans_set_contains_eigenspace} implies that
\begin{align}
\label{eqn:claim_CH_contained_in_min_trans}
\CH_{\Omega}\{a,v,b\} \subset {\rm Min}_{\Om}(A_\g).
\end{align}

The group $\G$ acts co-compactly on $\Om$. Then, Theorem \ref{thm:centralizer} implies that $C_{\Gamma}(A_\g)$ acts co-compactly on $\CH_{\Omega}({\rm Min}_{\Om}(A_\g))$.  Fix $p \in (a,b)$ and choose $v_n \in [p,v)$ such that $\lim_{n \to \infty} v_n=v$. By equation \eqref{eqn:claim_CH_contained_in_min_trans}, $v_n \in {\rm Min}_{\Om}(A_\g)$. Then there exists $h_n \in C_{\G}(A_\g)$ such that $q:=\lim_{n \to \infty} h_nv_n$ exists in $\Omega$. Thus $\lim_{n \to \infty} \hil(h_n^{-1}q,v_n)=0$. Then Proposition \ref{line-in-bdry-0} implies that, up to passing to a subsequence, 
\begin{equation*}
\lim_{n \to \infty} h_n^{-1}q=\lim_{n \to \infty}v_n=v.
\end{equation*}
Pick a point $q' \in (a,b)$. Up to passing to a subsequence, $v':=\lim_{n \to \infty} h_n^{-1}q'$ exists in $\clOm$. Since $\lim_{n \to \infty}\hil(h_n^{-1}q,h_n^{-1}q')=\hil(q,q')$, Proposition \ref{line-in-bdry-0} implies that $v \in F_{\Omega}(v')$.  Now we show that $v'\in\{a,b\}$. Since $h_n \in C_{\G}(A_\g)$, $h_n(a,b)$ is an axis of $\g$. As $\g$ is biproximal and has an axis, \cref{cor:biproximal_with_axis_implies_unique_axis} implies that $h_n(a,b)=(a,b)$. Then, since $q'\in (a,b)$, $v'=\lim_{n \to \infty}h_n^{-1}q' \in \{ a, b\}$. Hence $$v  \in  F_{\Omega}(a) \cup F_{\Omega}(b).$$ 

If possible, let $v \in F_{\Omega}(a)$. By equation \eqref{eqn:avb_in_bdry}, $[a,v] \cup [v,b] \subset \bdry$. Now, by Proposition \ref{prop:faces} part (4), $v \in F_{\Omega}(a)$ and $[v,b] \subset \bdry$ implies that $[a,b] \subset   \bdry$. This is a contradiction as $(a,b) \subset \Om$. Thus, $v \not \in F_{\Om}(a)$. So $v$ must be in $F_{\Omega}(b)$. By a similar reasoning, we now observe that $v \not \in F_{\Om}(b)$. Thus we have a contradiction. 

So we have shown that if $\g$ has an axis $(a,b)$ and is biproximal, then $\g$ has no fixed points in $\bdry$ other than $a$ and $b$.  Then the proof of part (4) of Proposition \ref{prop:rank-one-properties} goes through verbatim. Thus $(a,z) \cup (z,b)\subset \Omega$ for all $z \in \bdry \setminus \{a,b\}$, that is, the axis $(a,b)$ is not contained in any half triangle in $\bdry$. This finishes the proof.
 \end{proof}

 \section{Rank one axis and thin triangles}
\label{sec:rk-1-thin-triangles}

In this section, we prove that any projective geodesic triangle in $\Om$ with one of its sides on a rank one axis $\ell$ is $\Dell$-thin for some constant $\Dell$. 

\begin{proposition} 
\label{prop:Dl-thin}
Suppose $\Om$ is a Hilbert geometry. If $\ell \subset \Om$ is a rank one axis, then there exists a constant $\Dell \geq 0$ such that:  if $\Delta(x,y,z):=[x,y]\cup [y,z] \cup [z,x]$ is a projective geodesic triangle in $\Omega$ with $[y,z] \subset \ell$, then $\Delta(x,y,z)$ is $\Dell$-thin. 
\end{proposition}
\begin{remark}
The thinness constant $\Dell$ in the above theorem depends only on the axis $\ell$ (and not on the rank one isometry that has $\ell$ as its axis).
\end{remark}

But first let us introduce some relevant definitions and technical results that we will need.

\subsection{Thin Triangles}
\begin{definition} 
\label{defn:thin_triangles}
Suppose $(X,d)$ is a geodesic metric space.
\begin{enumerate}
\item  A geodesic triangle $T$ with vertices $y_1, y_2, y_3$ is a union of geodesics $\sigma_1 \cup \sigma_2 \cup \sigma_3$ where $\sigma_i$ is a geodesic joining $y_{i}$ and $y_{i+1}$, where the indices $i \in\{1,2,3\}$ and are counted modulo 3. 
\item A geodesic triangle $T:=\sigma_1\cup\sigma_2\cup \sigma_3$ is called $D$-thin for some $D \geq 0$ if $$\sigma_i \subset \{ x \in X : d(x,\sigma_{i-1} \cup \sigma_{i+1}) < D\}$$ where the indices $i \in\{1,2,3\}$ and are counted modulo 3.
\end{enumerate}
\end{definition}

The following is an elementary observation about thin triangles that we use later in the paper.
\begin{observation}
\label{obs:thinness_implies_special_points}
Suppose $(X,d)$ is a geodesic metric space and $T:=\sigma_1 \cup \sigma_2 \cup \sigma_3$ is a geodesic triangle with vertices $y_1,y_2,y_3$ and each $\sigma_i$ is a continuous geodesic path joining $y_i$ and $y_{i+1}$ (the indices $i \in \{1,2,3\}$ and counted modulo 3). If $T$ is $D$-thin, then there exist $x_i \in \sigma_i$ for $i=1,2,3$ such that $\max\{d(x_1,x_2),d(x_1,x_3)\} \leq D$. 
\end{observation}
\begin{proof}
By slight abuse of notation, let $\sigma_1:[0,b] \to X$ denote the continuous parametrization of the geodesic path $\sigma_1$ for some $b \geq 0$. Without loss of generality, we assume that $\sigma_1(0)=y_1$. Since $T$ is $D$-thin, 
\begin{align}
\label{eqn:consequence_of_thinness}
\sigma_1([0,b]) \subset \{ x \in X: d(x,\sigma_2 \cup \sigma_3) < D\}.
\end{align} 
Note that $d(\sigma_1(0),\sigma_3)=0$ as $y_1 \in \sigma_1 \cap \sigma_3$. Let $E:=\{ t : t \in [0,b], d(\sigma_1(t),\sigma_3) < D\}.$ Then $0 \in E$ and $s_0:=\sup E$ exists. We can find a sequence $\{t_n\}$ in $E$ such that $t_n \to s_0$. Then, by continuity of $\sigma_1$, 
$$d(\sigma_1(s_0),\sigma_3)=\lim_{t_n \to s_0 } d(\sigma_1(t_n),\sigma_3)\leq D.$$ 
Now note that $d(\sigma_1(s_0),\sigma_2)\leq D$. Indeed, if $t>s_0$, then $d(\sigma_1(t),\sigma_3) \geq D$ by definition of $s_0$. Then \cref{eqn:consequence_of_thinness} implies that $d(\sigma_1(t),\sigma_2) < D$. Then by continuity of $\sigma_1$, $$d(\sigma_1(s_0),\sigma_2)=\lim_{t \to s_0^+} d(\sigma_1(t),\sigma_2) \leq D.$$ Then set $x_1:=\sigma_1(s_0)$ and for $i=2,3$, let $x_i \in \sigma_i$ be such that $d(x_1,x_i)=d(x_1,\sigma_i)$.
\end{proof}

Suppose $(\Omega,\hil)$ is a Hilbert geometry. Then there are some special geodesic triangles in $\Omega$, namely the ones whose edges are projective geodesic segments.
\begin{definition}
If $v_1, v_2, v_3 \in \Omega$, a projective geodesic triangle (with vertices $v_1$, $v_2$ and $v_3$) is $$\Delta( v_1, v_2,v_3):= [v_1,v_2] \cup [v_2,v_3] \cup [v_3, v_1].$$ 
\end{definition}
We will say that $\Delta(v_1,v_2,v_3)$ is $D$-thin if it is $D$-thin in the sense of \cref{defn:thin_triangles}. There is a simple criteria to determine whether a projective geodesic triangle is $D$-thin. This proof comes from \cite{IZ2019} and we include it here for the convenience of the reader.

\begin{lemma}\label{lem:thin-triangle-criterion}
Suppose $\Om$ is a Hilbert geometry, $R \geq 0$, and $\Delta(x,y,z)$ is projective geodesic triangle such that $[y,z] \subset \Nc_{R} \big( [x,y] \cup [x,z]\big)$. Then $\Delta(x,y,z)$ is $(2R)$-thin.
\end{lemma}
\begin{proof}
Since $[y,z] \subset \Nc_{R} \big( [x,y] \cup [x,z]\big),$ there exist $m_{yz} \in [y,z]$, $m_{xy} \in [x,y]$ and $m_{xz} \in [x,z]$ such that $\hil (m_{yz},m_{xy}) \leq R$ and $\hil(m_{yz},m_{xz}) \leq R.$ Indeed, the existence of three such points follows by a similar reasoning as in the proof of \cref{obs:thinness_implies_special_points}. Then, by Proposition \ref{prop:iz-dist-estimate}, 
\begin{align*}
\hil^{\Haus}([y,m_{yz}],[y,m_{xy})  &\leq R, \\
\hil^{\Haus}([z,m_{yz}],[z,m_{xz}]) &\leq R, \\
\hil^{\Haus}([x,m_{xy}],[x,m_{xz}]) &\leq 2R.
\end{align*}
Hence, $\Delta(x,y,z)$ is $(2R)$-thin. 
\end{proof}

\subsection{Proof of \cref{prop:Dl-thin}} \

Now we prove \cref{prop:Dl-thin}. Fix a Hilbert geometry $\Om$ and a rank one axis $\ell \subset \Om$. The remark following \cref{prop:Dl-thin} will be a consequence of the proof -- the proof only uses the fact that there is some rank one isometry $\gamma$ that acts along $\ell$ by a translation; it does not rely on $\gamma$ in any other manner. Lemma \ref{lem:thin-triangle-criterion} reduces  \cref{prop:Dl-thin} to the following.

\begin{proposition}
If $\ell \subset \Om$ is a rank one axis, then there exists a constant $\Bell$ with the following property: if $\Delta(x,y,z)$ is an projective geodesic triangle in $\Omega$ with $[y,z] \subset \ell$, then $[y,z] \subset \mathcal{N}_{\Bell} \big( [x,y] \cup [x,z] \big).$ Moreover, this constant $\Bell$ depends only on the rank one axis $\ell$ (and not on the rank one isometry whose axis is $\ell$).
\label{Dl_1}
\end{proposition}
\begin{proof}[Proof of Proposition]
The moreover statement will again follow from the proof since the proof is independent of the choice of the rank one isometry which has $\ell$ as its axis. Now we begin the proof of the first part.

Suppose the claim is false. Then for each $n \geq 0,$ there exists a projective geodesic triangle $\Delta (a_n,b_n,c_n) \subset \Om$ with $[a_n,b_n]\subset \ell$, $c_n \in \Om$ and $e_n \in (a_n,b_n)$ such that $$\hil \left( e_n,[c_n,a_n] \cup [c_n,b_n] \right) \geq n.$$

Since $\ell$ is a rank one axis, there exists a rank one isometry $\g'$ whose axis is $\ell.$ Thus, translating $\Delta (a_n,b_n,c_n)$ by elements in $\langle \g' \rangle$ and passing to a subsequence, we can assume that $e:=\lim_{n \to \infty} e_n$ exists and $e \in \ell.$ Up to passing to a subsequence, we can assume that $a:=\lim_{n \to \infty} a_n$, $b:=\lim_{n \to \infty} b_n$ and $c:=\lim_{n \to \infty}c_n$ exist. Observe that:
\begin{align*}
\hil(e,[a,c]\cup[c,b])=\lim_{n \to \infty}\hil(e_n,[a_n,c_n] \cup [c_n,b_n]) \geq \lim_{n \to \infty} n=\infty.
\end{align*}
Thus $[a,c]\cup[c,b] \subset \bdry$. But $(a,b) \subset \Om$ since $e \in (a,b) \cap \Om$. Thus $a, c, b$ form a half triangle in $\Om$. But since $[a_n, b_n] \subset \ell$, $[a,b] \subset \overline{\ell}$. Since $a,b \in \bdry$ and $\ell \subset \Om$, $\overline{\ell}=[a,b]$. Thus $\ell=(a,b)$.  So the rank one axis $\ell$ is contained in a half triangle in $\Om$, a contradiction. 
\end{proof}

\section{Rank one Hilbert geometry: Zariski density and Limit sets}
\label{sec:rank-one-hil-geom}

Recall the definition of rank one geodesics from \cref{defn:rank_one_geodesic}. In this section we would like to address the following question.
\begin{question}
\label{ques:does_rk1_geod_imply_rk1}
Suppose $(\Omega,\Gamma)$ is a Hilbert geometry and $\Omega$ contains a rank one  geodesic. When does this imply that $(\Omega,\Gamma)$ is a rank one Hilbert geometry? 
\end{question}

It is a natural question that aims to understand how the geometry of a properly convex domain influences the algebraic properties of a `large' group acting on it. Under certain assumptions, Zimmer answers this question in \cite{Z2019}. 

\begin{proposition}
\label{prop:partial_answer_to_ques}
Suppose $\Omega$ is an irreducible Hilbert geometry and $\G \leq \Aut(\Omega)$ acts co-compactly on $\Om$. Then $(\Om,\G)$ is a rank one Hilbert geometry if and only $\Om$ contains a rank one geodesic.
\end{proposition}

This is immediate from \cref{thm:rank_rigidity_detailed}. So our main goal in this section is to answer \cref{ques:does_rk1_geod_imply_rk1} without the assumptions of irreducibility or co-compactness as above. Instead, we will work with groups that satisfy the following assumption.

\begin{quote}
\hypertarget{assumption}{\textbf{Assumption:}} \emph{$\Gamma \leq \SL_{d+1}(\Rb)$ is a discrete Zariski dense subgroup of $\SL_{d+1}(\Rb)$ and there exists a properly convex domain $\Om\subset \RP$ such that $\Gamma \cdot \Omega=\Omega$.}
\end{quote} 
In this assumption, Zariski density may be interpreted as an assurance that the group $\G$ is `large'. We will work with $\SL_{d+1}(\Rb)$ in this section instead of $\PGL$. Indeed, given $\Gamma \leq \PGL$, we can pass to a subgroup of index at most 2 and assume that $\G \leq \SL_{d+1}(\Rb)$. In \cref{sec:hypothesis_star}, we will formulate a hypothesis on the proximal limit set $\rkonelimset$ (cf. \ref{defn:proximal_limset}) that we call \hyperlink{hypoth_star}{\bf Hypothesis ($\star$)} and use it to provide \hyperlink{ans:answer_to_ques_in_sec_8}{\bf an answer to \cref{ques:does_rk1_geod_imply_rk1}.}

\subsection*{Notation} For the rest of this section, let $G:=\SL_{d+1}(\Rb)$,  $P \leq G$ be the subgroup of upper-triangular matrices and $Q$ be the stabilizer in $G$ of $[e_1]=[1:0:\ldots:0] \in \Pb(\Rb^{d+1})$. Fix the standard inner product on $\Rb^{d+1}$ and let $K:={\rm SO}(d+1)$.

Let $\e_i$ be the evaluation of the $i$-th diagonal entry of a diagonal matrix. Then we take $\Delta:=\{\e_i-\e_{i+1}:1 \leq  i \leq d \}$ to be the set of positive simple roots.  
For any $\theta=\{\e_{i_1}-\e_{i_1+1},\dots, \e_{i_k}-\e_{{i_k}+1}\} \subset \Delta$, let $P_{\theta}$ denote the subgroup of block upper triangular matrices in $G$ with square diagonal blocks of size $i_1,i_2-i_1,\dots, i_k-i_{k-1}$, $d-i_k$ respectively. In particular, $P_{\Delta}=P$ and $G/P$ is the full flag variety while $P_{\{\e_1-\e_2\}}=Q$ and $G/Q\cong \RP$.

\subsection{Limit sets in flag varieties}
We will require the notion of limit sets of discrete subgroups of $G$ in flag varieties, in particular $G/P$ and $G/Q$. This has been defined and studied by various authors in different degrees of generality: Guivarch \cite{guivarch1990} (for subgroups of $\SL_{d+1}(\Rb)$ acting proximally and strongly irreducibly on $\Rb^d$), Benoist \cite{B1997} (for Zariski dense subgroups of reductive groups) and Gu\'{e}ritaud-Guichard-Kassel-Wienhard  \cite{GGKW2017} (for arbitrary subgroups of reductive groups). We use the definition from \cite[Section 5.1]{GGKW2017}

First recall the notion of singular value decomposition (or more generally, Cartan decomposition in $G$): for any $g \in G$, there exist $k_1,k_2 \in \SO(d+1)$ and $A_g=\diag(a_1,\dots,a_{d+1})$ with $a_1 \geq \dots \geq a_{d+1} >0$ such that $$g=k_1A_gk_2.$$ The Cartan decomposition defines the Cartan projection $\mu(g):=\diag(\log(a_1),\dots,\log(a_{d+1})).$ It maps $G$ into the space of trace zero diagonal matrices of size $(d+1) \times (d+1)$.

Let $\theta \subset \Delta$. If $g \in G$ has a singular value decomposition $g=k_1 A_g k_2$, define $E_{\theta}: G \to G/P_{\theta}$ by $$E_{\theta}(g): =k_1 \cdot eP_{\theta}.$$ The map $E_{\theta}$ does not depend on the choices of $k_1, k_2$ provided $\alpha(\mu(g))>0$ for all $\alpha \in \theta$ \cite[Section 5.1]{GGKW2017}. 

\begin{definition}[{\cite[Definition 5.1]{GGKW2017}}]
\label{defn:proximal_limset}
Suppose $\Gamma_0$ is a discrete subgroup of  $G$. The limit set $\Lambda_{\Gamma_0}^{G/P_{\theta}}$ of $\Gamma_0$ in $G/P_{\theta}$ is defined to be the set of all accumulation points of sequences $\{E_{\theta}(\gamma_n)\}_{n \in \Nb}$ where $\{\gamma_n\}_{n\in \Nb}$ is any sequence in $\Gamma_0$ such that $\alpha(\mu(\gamma_n)) \to \infty$ for all $\alpha \in \theta$. 
\end{definition}

\begin{remark} \label{rem:proximal_limset_ZD}
Suppose $\Gamma_0$ is Zariski dense in $G$. 
\begin{enumerate}
\item Then $\Lambda_{\Gamma_0}^{G/P_{\theta}}$ is non-empty  and is the closure of the set of attracting fixed point of proximal elements in $G/P_{\theta}$ (\cite{B1997}, \cite[Section 5.1]{GGKW2017}). Here, an element $g \in G$ is called \emph{proximal}\footnote{This coincides with the notion of proximality discussed in \cref{sec:proximality} when $\theta=\{\e_1-\e_2\}$. } in $G/P_{\theta}$  provided $\alpha(\mu(g))>0$ for all $\alpha \in \theta$. Moreover, $g$ is proximal in $G/P_{\theta}$ if and only if $g$ has a unique attracting fixed point\footnote{A fixed point $x \in X$ of a smooth map $f:X \to X$ is attracting if $||Df_x||<1$.} in $G/P_{\theta}$ \cite[Definition 2.25]{GGKW2017}.
\item Suppose $\theta=\{\e_1-\e_2\}$ so that $P_{\theta}=Q$. Then $\Lambda_{\Gamma_0}^{G/Q}$ is the unique minimal closed $\G$-invariant subset of $G/Q$ \cite[Lemma 4.2]{BQ2016}. This may not be true  for arbitrary choices of $\theta$, see \cite[Remark 4.4]{BQ2016}.
\end{enumerate}
\end{remark}

\begin{lemma}
\label{lem:proximal_limset_inside_bdry}
Suppose $\G \leq \SL_{d+1}(\Rb)$ satisfies the \hyperlink{assumption}{\bf Assumption}. Then $\rkonelimset \neq \emptyset$ and $\rkonelimset \subset \partial \Omega$ is the unique minimal closed $\G$-invariant subset of $\bdry$.  
\end{lemma}
\begin{proof} Note that $\bdry$ is a closed $\G$-invariant set and the unique attracting fixed point of any proximal element (in $G/Q$) of $\G$ lies in $\bdry$. The lemma then follows from the above \cref{rem:proximal_limset_ZD}.
\end{proof}

If we do not assume Zariski density, then we may still have non-empty limit set (in an appropriate $G/P_{\theta}$) but with some unusual properties. The following is such an example.
\begin{example}
\label{example:proximal_limset_proj_torus}
Consider the discrete subgroup $\G':=\left \{ \diag(2^{m_1},\dots,2^{m_{d+1}}) : \sum_{i=1}^{d+1} m_i=0 \right\}$ of  $\Aut(T_d)$ and $T_d/\G'$ is a $d$-dimensional torus. Although $\G'$ is not Zariski dense in $\SL_{d+1}(\Rb)$, the proximal limit set in $\RP$ is non-empty and in fact $\Lambda^{G/Q}_{\G'}=\{ [e_1],\dots,[e_{d+1}]\}$. Thus $\Lambda_{\G'}^{G/Q}$ is a proper subset of $\partial T_d$. Note that $(T_d,\G')$ is not a rank one Hilbert geometry, see \cref{rem:rank_one_axis_vs_axis}.
\end{example}

In the light of \cref{lem:proximal_limset_inside_bdry} and this example, it is natural to ask when does $\rkonelimset$ equal $\bdry$.

\begin{remark}
\label{rem:proximal_limit_equals_bdry}
In general, if $\G$ only satisfies the \hyperlink{assumption}{\bf Assumption}, then $\rkonelimset$ can be a proper subset of $\bdry$. For example, let $\G \leq {\rm PO}(2,1)$ be a Zariski dense convex co-compact Kleinian group. Then $\rkonelimset=\overline{\G \cdot x} \cap \partial \Hb^2$ where $x \in \Hb^2$. Unless $\G$ is co-compact, $\rkonelimset \neq \partial \Hb^2$. However, under the additional co-compactness assumption, we often have equality. Blayac \cite[Theorem 1.3]{PLB2021-boundary} has recently shown that  if $(\Omega,\Gamma)$ is a divisible rank one Hilbert geometry, then $\rkonelimset=\bdry$. 
\end{remark}

\subsection{Hypothesis ($\star$) and an answer to \cref{ques:does_rk1_geod_imply_rk1}}
\label{sec:hypothesis_star}

We now introduce a special hypothesis under which we can answer \cref{ques:does_rk1_geod_imply_rk1}.

\begin{quote}
\label{quote:hypoth_star}
\emph{Suppose $\G \leq G$ is a discrete subgroup and $\Omega\subset \Pb(\Rb^{d+1})$ is a properly convex domain such that $\Gamma \cdot \Omega=\Omega$.} \emph{We will say that $(\Omega,\Gamma)$ satisfies \hypertarget{hypoth_star}{\emph{\bf Hypothesis $(\star)$}} if  there exists a rank one geodesic $(a', b') \subset \Omega$ with its endpoints $a',b' \in \rkonelimset \cap \partial \Omega$.}
\end{quote}

We will show that for any Zariski dense discrete subgroup $\G$, this hypothesis is equivalent to the rank one property. One implication is easy and does not require Zariski density.

\begin{lemma}
\label{lem:rank_one_implies_hypoth_star}
Suppose $\G \leq G$ is a discrete subgroup that preserves a properly convex domain $\Om$ and $(\Om,\G)$ is a rank one Hilbert geometry. Then $(\Omega,\G)$ satisfies \hyperlink{hypoth_star}{\emph{\bf Hypothesis $(\star)$}}.
\end{lemma}
\begin{remark}
Note that in this lemma, we do not assume that $\G$ is Zariski dense in $G$.
\end{remark}
\begin{proof}
Since $(\Om,\G)$ is a rank one Hilbert geometry, we can find a rank one isometry $\g \in \G$. Let $\g^{\pm} \in \bdry$ be the attracting and the repelling fixed points of $\g$. Then $\g^{\pm} \in \rkonelimset$ by definition of $\rkonelimset$.  Also $(\g^+,\g^-)$ is the axis of $\g$ and hence a rank one geodesic, see \cref{defn:rank-one-isometry} and \cref{prop:rank-one-properties}.
\end{proof}

Next we will seek a converse of the above lemma and this will require the Zariski density assumption on $\G$. But first we recall the notion of loxodromic elements. We will call $g \in G$ \emph{loxodromic} if 
$$|\ev_1(g)|> \dots > |\ev_{d+1}(g)|.$$ If $g$ is loxodromic, then it has unique attracting fixed points in both $G/Q$ and $G/P$. We will denote by $a_{g}^{\pm} \in G/Q$ (resp. $\Xb_g^{\pm} \in G/P$) the unique attracting fixed point of $g^{\pm 1}$ in $G/Q$ (resp. $G/P$).  With this notation, $\Pi_{PQ}(\Xb_{g}^{\pm})=a_g^{\pm}$ where $$\Pi_{PQ}:G/P \to G/Q$$ 
is the natural smooth projection map. Also recall that if $g \in \Aut(\Omega)$ is a rank one isometry, then we denote by $g^{\pm}$ the attracting and the repelling fixed points of $g$ (\cref{rem:ns_dynamics}).

\begin{lemma}\label{lem:limit-set-rank-one-hil-geom}
 Suppose $\G \leq G$ satisfies the \hyperlink{assumption}{\bf Assumption}.   If there exists a rank one geodesic $(a',b') \subset \Omega$  with its endpoints $a',b ' \in \rkonelimset \cap \bdry$, then:
\begin{enumerate}
\item there exist rank one isometries $\{g_n\}$ in $\G$ such that $\lim_{n \to \infty}g_n^+=a'$ and $\lim_{n \to \infty} g_n^-=b'$, 
\item $(\Om,\G)$ is a rank one Hilbert geometry,
\item the set of rank one isometries in $\G$ is Zariski dense in $G$, 
\item $\rkonelimset=\overline{\{ \g ^+: \g \text{ is a rank one isometry } \}}$.
\end{enumerate} 
\end{lemma}
\begin{corollary}
\label{cor:hypoth_star_implies_rank_one}
 Suppose $\G \leq G$ satisfies the \hyperlink{assumption}{\bf Assumption} and $(\Omega,\G)$ satisfies \hyperlink{hypoth_star}{\emph{\bf Hypothesis $(\star)$}}. Then $(\Omega,\Gamma)$ is a rank one Hilbert geometry.
\end{corollary}

\begin{proof}[Proof of Lemma \ref{lem:limit-set-rank-one-hil-geom}]
The key idea of this proof is in \cite{SAM2019} and it relies on results of Benoist \cite{benoist_auto_convex_cone}. Before starting the proof, we informally outline the main idea. The key technical point is to find a seqeunce $\{g_n\}$ of biproximal elements in $\G$ such that $a_{g_n^+} \to a'$ and $a_{g_n}^- \to b'$. A direct way to find such a $\{g_n\}$ is: using Zariski density, find a pair $g,h \in \Gamma$ of transversally biproximal elements \cite[Chapter 7]{BQ2016} such that $a_g^+$ and $a_h^-$ are arbitrarily close to $a'$ and $b'$ respectively. Then, for large enough $n$,  $g^nh^n$ is a biproximal element whose attracting and repelling fixed points are close to $a'$ and $b'$. However, in this proof, we will take a more  indirect approach by passing to the limit set in $G/P$ and using a result of Benoist. We rely on \cite[Lemma 2.6 (c)]{benoist_auto_convex_cone}: given two distinct points $\Xb_{+}, \Xb_- \in \Lambda_{\G}^{G/P}$, there exist loxodromic elements $g_n \in \Gamma$ such that $\Xb_{g_n}^{\pm} \to \Xb_{\pm}$.   Once we have this sequence $\{g_n\}$, Claim \ref{claim:many_rank_ones} implies that all but finitely many of them are rank one isometries.

Now we begin the formal proof. Equip $G/P$ and $G/Q$ with $K$-invariant Riemannian metrics and denote the corresponding Riemannian distance functions by $d_P$ and $d_Q$ respectively. We remark that this specific choice of Riemannian metrics will be insignificant as $G/P$ and $G/Q$ are compact manifolds. Let $\G_{\rm lox}$ be the set of loxodromic elements in $\G$. Since $\G$ is Zariski dense in $G$, \cref{rem:proximal_limset_ZD} implies that $\Pi_{PQ}(\Lambda_{\G}^{G/P})=\rkonelimset$. Then pick $\Xb_a, \Xb_b \in \Lambda_{\G}^{G/P}$ such that $\Pi_{PQ}(\Xb_a)=a'$ and $\Pi_{PQ}(\Xb_b)=b'$.  For any $\e>0$, 
 \begin{equation*}
\G_{\e}:=
\{ g \in \G_{\rm lox} : \dist_{P}(\Xb_g^+, \Xb_a) < \e, \dist_{P}(\Xb_g^-,\Xb_b) < \e \}
\end{equation*} 
 is Zariski dense in $G$ \cite[Lemma 2.6 (c)]{benoist_auto_convex_cone}. 
 
For any $g \in \G_{\rm lox}$, $a_g^{\pm}=\Pi_{PQ}(\Xb_g^{\pm})$ and $a_{g}^{\pm} \in \bdry$. Moreover, $\Pi_{PQ}$ is continuous and $(a',b') \subset \Om$.  Thus there exists $\e'$ such that: if $\e \in (0,\e')$,  then $(a_g^+,a_g^-) \subset \Om$ for any $g \in \G_{\e}$. In fact $(a_g^+,a_g^-) \subset \Om$ is the unique axis in $\Om$ for any such $g \in \G_{\e}$. Indeed, the uniqueness follows from \cref{cor:biproximal_with_axis_implies_unique_axis}  because $g$ has an axis $(a_g^+,a_g^-) \subset \Om$, $g$ is loxodromic and $\tau_{\Om}(g)>0$. We now claim that: 
\begin{claim}
\label{claim:many_rank_ones}
If $\e \in (0,\e')$ is small enough, then $g$ is a rank one isometry for all $g \in \G_{\e}$.
\end{claim}
\noindent\emph{Proof of Claim.} Suppose the claim is false. Then there exist a sequence $\{\e_n\}$ in $(0,\e')$ with $\e_n \to 0$ and $g_n \in \G_{\e_n}$ such that $g_n$ is not a rank one isometry. Then $\Xb_{g_n}^+ \to \Xb_a$ and $\Xb_{g_n}^- \to \Xb_b$. Since $\Pi_{PQ}$ is continuous, $a_{g_n}^+ \to a'$ and $a_{g_n}^-\to b'$. 

By the paragraph before the claim, each $g_n$ has a unique axis $(a_{g_n}^+,a_{g_n}^-) \subset \Om$.  Moreover, $(a_{g_n}^+,a_{g_n}^-) \to (a',b')$. But since $g_n$ is not a rank one isometry by assumption, this implies that there exists $\{c_n\}$ with $c_n \in \bdry - \{ a_{g_n}^+,a_{g_n}^-\}$ such that $$[a_{g_n}^+,c_n] \cup [c_n,a_{g_n}^-] \subset \bdry.$$ Up to passing to a subsequence, we can assume that $c_n \to c$ in $\bdry$. Then $[a',c] \cup [c,b' ] \subset \bdry$ while $(a',b') \subset \Om$. Thus $(a',b')$ cannot be a rank one geodesic and we have a contradiction. This finishes the proof of this claim.

Now we finish the proof of the lemma. Let us choose an $\e \in (0,\e')$ as in the above claim.

\noindent (1) The result follows by choosing $g_n \in \G_{\e/n}$ for all $n \geq 1$. 

\noindent (2) This follows from (1), since there is at least one rank one isometry in $\G$.

\noindent (3) The set $\G_{\e}$ is a subset of the set of rank one isometries of $\G$ and $\G_{\e}$ is Zariski dense. 

\noindent (4) By \cref{lem:proximal_limset_inside_bdry}, $\rkonelimset \subset \bdry $ is a minimal, closed $\G$-invariant set which contains the unique attracting fixed points of all proximal elements. Since a rank one isometry is necessarily proximal, 
$\overline{ \{ \g^+ : \g \text{ is a rank one isometry } \}} \subset \rkonelimset.$ Since $\overline{ \{ \g^+ : \g \text{ is a rank one isometry } \}}$ is a closed $\G$-invariant set,  the equality then follows from minimality of $\rkonelimset$. 
\end{proof}

We now observe that \hyperlink{hypoth_star}{\bf Hypothesis ($\star$)} gives an \hypertarget{ans:answer_to_ques_in_sec_8}{\textbf{answer to \cref{ques:does_rk1_geod_imply_rk1}}} (cf. \ref{lem:rank_one_implies_hypoth_star} and \ref{cor:hypoth_star_implies_rank_one}):
\begin{quote} 
if $\G \leq \SL_{d+1}(\Rb)$ is a discrete Zariski dense subgroup of $\SL_{d+1}(\Rb)$ that preserves a properly convex domain $\Om$, then \emph{$(\Om,\G)$ is a rank one Hilbert geometry} if and only if \emph{$\Om$ contains a rank one geodesic $(a',b') \subset \Om$ with $a',b' \in \rkonelimset \cap \bdry$}.
\end{quote}

We finish the section with an example where \hyperlink{hypoth_star}{\bf Hypothesis ($\star$)} fails. Recall \cref{example:proximal_limset_proj_torus}. In that case, $\G'$ preserves the standard $d$-simplex $T_d$, $T_d/\G'$ is homeomorphic to a $d$-torus, and $T_d$ does not contain any rank one geodesics. Thus $(T_d,\G')$ does not satisfy \hyperlink{hypoth_star}{\bf Hypothesis ($\star$)}. However, in this example, the group $\G'$ is not Zariski dense in $\SL_{d+1}(\Rb)$ and one may wonder if that is the reason why \hyperlink{hypoth_star}{\bf Hypothesis ($\star$)} fails. So, we ask the following question.
\begin{question}
Suppose $\Gamma\leq G$ is a discrete subgroup that preserves a properly convex domain $\Om$. If $\G$ is Zariski dense in $G$, then does $(\Om,\G)$ satisfy  \hyperlink{hypoth_star}{\bf Hypothesis ($\star$)}?
\end{question}
To the best of the author's knowledge, the answer to this question is not known unless one makes other assumptions, e.g. say $\Om/\G$ is compact and $\Om$ is irreducible. Under these assumptions, \cref{rem:proximal_limit_equals_bdry} and \cref{thm:rank_rigidity_detailed} together provide an answer.

\part{Contracting elements in Hilbert geometry}
\label{part:contracting_element_hil_geom}

In this part of the paper, we prove our main results  - Theorems \ref{thm:contracting-iff-rank-one} and \ref{thm:acy-hyp}. The outline of this part of the paper is as follows. We recall the notion of contracting elements in Section \ref{sec:defn-contracting-element}. The proof of Theorem \ref{thm:contracting-iff-rank-one} is split into two sections: Section \ref{sec:rank-one-implies-contracting} and Section \ref{sec:contracting-implies-rank-one}. In Section \ref{sec:acy_hyp}, we introduce the notion of acylindrically hyperbolic groups  and prove Theorem \ref{thm:acy-hyp}.

\section{Contracting Elements: Definition and Properties}
\label{sec:defn-contracting-element}

Suppose $K \geq 1$ and $C \geq 0$. A function $F:(X,\dist_X) \to (Y,\dist_Y)$ is called a $(K,C)$-quasi-isometric embedding if for any $x_1,x_2 \in X$,
\begin{equation*}
\dfrac{1}{K} \dist_X(x_1,x_2)-C \leq \dist_Y(F(x_1),F(x_2)) \leq K \dist_X(x_1,x_2)+C.
\end{equation*}

Fix a proper geodesic metric space $(X,\dist)$ and a group $G$ that acts properly and by isometries on $X$. If $K \geq 1$ and $C \geq 0$, then a $(K,C)$-path in $(X,\dist)$ is a set $F(\Rb)$ where $F:(\Rb,|\cdot|) \to (X,\dist)$ is a $(K,C)$-quasi-isometric embedding.  A subpath of the path $F(\Rb)$ is $F(I)$ where $I \subset \Rb$ is an interval, possibly unbounded. 

\begin{definition}
\label{defn-paths}
Let $K \geq 1$ and $C\geq 0$. Let $\PS$ be a collection of $(K,C)$-paths in $X$. Then: 
\begin{enumerate}
\item $\PS$ is called a \emph{path system} on $X$ if: 
\begin{enumerate}
\item any subpath of a path in $\PS$ is also in $\PS$ and
\item any pair of points in $X$ can be connected by a path in $\PS$.
\end{enumerate}
\item $\PS$ is called a \emph{geodesic path system} if all paths in $\PS$ are geodesics in $(X,\dist)$.
\item If $G$ preserves $\PS$, then $(X,\PS)$ is called a \emph{path system for the group $G$}.
\end{enumerate}
\end{definition}

\begin{definition}[Contracting subsets \cite{sisto_contracting_rw}]
\label{defn-contraction-sisto}
If $\PS$ is a path system on $X$, then $\Ac \subset X$ is called \emph{$\PS$-contracting}  (with constant $C$) if there exists a map $\pi_\Ac: X \to \Ac$ such that:
\begin{enumerate}
\item if $x \in \Ac$, then $\dist \big( x, \pi_\Ac(x) \big) \leq C$
\item if $x, y \in X$ and  $\dist \big( \pi_\Ac(x),\pi_\Ac(y) \big) \geq C, $ then for any path $\sigma \in \PS$ joining $x$ and $y$,  
\begin{equation*}
\dist \big( \sigma, \pi_\Ac(x) \big) \leq C ~~\text{ and }~~ \dist \big( \sigma, \pi_\Ac(y) \big) \leq C.
\end{equation*}
\end{enumerate}
\end{definition}

A prototypical example of a contracting subset is a bi-infinite geodesic in $\Hb^2$ (with the map $\pi_{\Ac}$ being the closest point projection on the geodesic). Generally speaking, one should think of the projection map $\pi_{\Ac}$ as an analogue of the closest-point projection. In fact, the following lemma makes this analogy concrete in the context of geodesic path systems. We will use the notation
$$\pid_{\Ac}(x):=\{a \in \Ac : \dist(x,a)=\dist(x,\Ac)\}$$  
for set-valued closest-point projection map on  $\Ac$.

\begin{lemma}\label{lem:coarse-equiv-of-proj}
Suppose $\PS$ is a geodesic path system and $\Ac \subset X$ is $\PS$-contracting (with constant $C$) with the projection map $\pi_\Ac: X \to \Ac$. Then $\left(\sup_{a \in \pid_{\Ac}(x)}\dist(\pi_\Ac(x), a) \right) \leq 2C$ for all $x \in X$.
\end{lemma}
\begin{proof}
Suppose there exist $x \in X$ and $a \in \pid_{\Ac}(x)$ such that 
\begin{align*}
\dist(\pi_\Ac(x), a)>2C.
\end{align*}
Since $\Ac$ is $\PS$-contracting and $a \in \Ac$, $\dist(\pi_\Ac(a), a) \leq C.$ Then 
\begin{equation*}
\dist(\pi_\Ac(x), \pi_\Ac(a) \geq \dist(\pi_\Ac(x), a)- \dist(\pi_\Ac(a), a)> C.
\end{equation*}
Let $\sigma_{x,a}$ be a geodesic path in $\PS$ joining $x$ and $a$. Since $\Ac$ is $\PS$-contracting, there exists $z \in \sigma_{x,a}$ such that $\dist(z, \pi_\Ac(x)) \leq C.$ As $z \in \sigma_{x,a}$, $\dist(a, z) = \dist(a,x)- \dist(z,x)$. As $a \in \rho_{\Ac}(x)$, $\dist(x,a) \leq \dist(\pi_{\Ac}(x),x)$. Then
\begin{align*}
\dist(a, z) \leq \dist(\pi_{\Ac}(x),x) -\dist(x,z) \leq \dist(\pi_\Ac(x),z)+ \dist(z,x)- \dist(z,x) \leq C.
\end{align*}
Then $\dist(\pi_{\Ac}(x),a) \leq \dist(\pi_{\Ac}(x),z)+\dist(z,a) \leq 2C$, a contradiction. 
\end{proof}

Using the notion of contracting subsets, one introduces the notion of  contracting group elements. A prototypical example of a contracting element is $g=\begin{pmatrix} \ev & 0 \\ 0 & 1/\lambda \end{pmatrix}$ for some $\lambda>1$, that acts on $\Hb^2$ by a translation along a bi-infinite geodesic in $\Hb^2$. 
\begin{definition}[Contracting elements \cite{sisto_contracting_rw}]\label{defn-contracting-element-sisto}
If $(X, \PS)$ is a path system for $G$, then $g \in G$ is a  \emph{contracting element} for $(X, \PS)$ provided for some (hence any) $x_0 \in X$: 
\begin{enumerate}
\item $g$ has infinite order and $\langle g \rangle \cdot x_0$ is a quasi-isometric embedding of $\Zb$ in $X$,
\item there exists $\Ac \subset X$ containing $x_0$ that is $\langle g \rangle$-invariant, $\PS$-contracting and has co-bounded $\langle g \rangle$ action.
\end{enumerate}
\end{definition}

\begin{remark}
We remark that if $g \in G$ is a contracting element and $\PS$ is a geodesic path system, then $\pi_{\Ac}$ is coarsely $\langle g \rangle$-equivariant. This is immediate from \cref{lem:coarse-equiv-of-proj} since $\pi_{\Ac}$ is coarsely equivalent to $\pid_{\Ac}$ and $\pid_{\Ac}$ is clearly $\langle g \rangle$-equivariant.
\end{remark}

In the definition of a contracting element, the set $\Ac$ is not necessarily a $\langle g \rangle$-orbit in $X$. We will now explain that we can always replace $\Ac$ by a $\langle g \rangle$-orbit. Moreover, we also show $g$ has positive translation length for its action on $X$. We remark that the following observation does not require that $\PS$ is a geodesic path system.

\begin{observation}\label{prop:on-contracting-elements}
Suppose $(X,\PS)$ is a path system for $G$ and $g \in G$ is a contracting element for $(X,\PS)$. Then: 
\begin{enumerate}
\item $\tau_X(g):=\inf_{x \in X} \dist (x,gx)$ is positive. 
\item for any $x_0 \in \Ac$, $\Amin:=\langle g \rangle x_0$ is the minimal $\PS$-contracting, $\langle g \rangle$-invariant subset of $X$ containing $x_0$ with a co-bounded $\langle g \rangle$ action. 
\end{enumerate}
\end{observation}
\begin{proof}
(1)  Recall the definition of stable translation length $$\tau^{\st}_X(g):=\lim_{n \to \infty}\frac{\dist(x,g^nx)}{n}.$$ Then $\tau_X(g) \geq \tau^{\st}_X(g)$ and it suffices to show $\tau^{\st}_X(g)>0$. Fix any $x_0 \in X$. Since $g$ is contracting, $\langle  g \rangle x_0$ is a quasi-geodesic, that is, there exists $K \geq 1$ and  $C \geq 0$ such that $\dist(x_0,g^nx_0) \geq \frac{1}{K}|n|-C$ for every $n \in \Zb$. Then, $\tau^{\st}_X(g)\geq 1/K>0.$

\noindent (2) Let $\Ac$ be $\PS$-contracting with constant $C_\Ac$ and the map $\pi_\Ac:X \to \Ac$. Fix any $x_0 \in \Ac$ and set $R_\Ac := \diam \big( \Ac/ \langle g \rangle \big)$, $C_0:=C_\Ac +2R_\Ac$ and $\Amin:= \langle g \rangle x_0$. 

Since $\Amin \subset \Ac$, if $x \in X$, then there exists $m \in \Zb$ such that $\dist(\pi_\Ac(x),g^mx_0) \leq R_\Ac$. Define $\pi_{\min}: X \to \Amin$ by setting $\pi_{\min}(x)=g^mx_0.$ Then,  if $x \in \Amin$, $\pi_{\min}(x)=x$. If $x,y \in X$ and $\dist(\pi_{\min}(x),\pi_{\min}(y)) \geq C_0$, then $\dist(\pi_\Ac(x), \pi_\Ac(y)) \geq C_\Ac$. Thus, if $\sigma \in \PS$ is a path from $x$ to $y$,   $\dist(\pi_\Ac(x), \sigma) \leq C_\Ac$ and $\dist(\pi_\Ac(y), \sigma) \leq C_\Ac.$ Hence, 
\begin{equation*}
\dist(\pi_{\min}(x), \sigma) \leq  C_0  ~~\text{ and }~~  \dist(\pi_{\min}(y), \sigma) \leq C_0. \qedhere
\end{equation*}
\end{proof}

There are many other  notions of contracting subsets in geometric group theory. We will require one such notion in Section \ref{sec:counting-conjugacy} for proving our  \cref{thm:counting-conjugacy}. We will call this notion \emph{contraction in the sense of $\BF$} -- it was introduced by Bestvina-Fujiwara for CAT(0) spaces \cite{bestvina_fujiwara_symm_space} and by Gekhtman-Yang in general \cite{gekhtman_wenyuan_counting}. We defer all further discussion about this to the Appendix \ref{appendix:contracting-BF} and only remark that in our case, this notion of contraction will be equivalent to \cref{defn-contraction-sisto}.

\begin{remark}\label{rem:contraction-equiv}
If $\Om$ is a Hilbert geometry, we will always use the geodesic path system $\PS^\Om:=\{ [x,y]: x,y \in \Om\}$ given by projective geodesics. We use the $\PS$-contracting notion everywhere in the paper except in Section \ref{sec:counting-conjugacy} (where we use \emph{contraction in the sense of BF}, cf. \ref{defn-contracting-BF}). Proposition \ref{prop:notion-of-contraction-equiv} below implies that these two notions of contraction are equivalent in our setup. Hence, in the rest of the paper, we will use the term  contracting subset (and element) without additional labels.
\end{remark}

\begin{proposition}\label{prop:notion-of-contraction-equiv}
Suppose $(X,\PS)$ is a geodesic path system. Then: 
\begin{enumerate}
\item $\Ac \subset X$ is $\PS$-contracting if and only if $\Ac$ is contracting in the sense of $\BF$. 
\item If $G$ preserves $\PS$, then $g \in G$ is a contracting element for $(X,\PS)$ if and only if $g \in G$ is a contracting element in the sense of $\BF$.
\end{enumerate}
\end{proposition}
\begin{proof}
See Appendix \ref{appendix:contracting-BF}.
\end{proof}

\section{Rank one isometries are contracting}
\label{sec:rank-one-implies-contracting}

 In this section, we prove one implication in Theorem \ref{thm:contracting-iff-rank-one}. Fix a Hilbert geometry $\Om$ and let $\PS^\Om:=\big\{ [x,y]: x, y \in \Om \big\}$.

\begin{theorem}\label{rank1-implies-contraction}
If $\g \in \Aut(\Omega)$ is a rank one isometry, then $\g$ is a contracting element for $(\Omega, \PS^\Om).$
\end{theorem}

The key step  will be part (2) of the following Lemma \ref{lem:rank1-axis-contracting} which shows that a rank one axis is $\PS^\Om$-contracting. First, we construct suitable projection maps on a rank one axis. Recall the notion of closest-point projection on closed convex subsets, particularly Corollary \ref{cor:P-sigma}.

\begin{definition}\label{defn:projection-map}
Suppose $\Om$ is a Hilbert geometry, $\ell$ is a bi-infinite projective geodesic in $\Om$ and $\sigma: \Rb \to \ell$ is its unit speed parametrization. Then $\Pi_{\ell}(x)=[\sigma(T_x^-),\sigma(T_x^+)]$ for $T_x^-,T_x^+ \in \Rb$. We define the projection map $\pi_{\ell}: \Omega \to \ell$ as
\begin{equation*}
\pi_{\ell}(x):=\sigma \Big( \dfrac{T_x^-+T_x^+}{2}\Big).
\end{equation*}
\end{definition}

\noindent We now establish some properties of the map $\pi_{\ell}$ when $\ell$ is a rank one axis.

\begin{lemma}\label{lem:rank1-axis-contracting}
If $\ell \subset \Om$ is a rank one axis, then there exists $\Cell \geq 0$ such that
\begin{enumerate}
\item if $x \in \Om$ and $z \in \ell$,  then there exists $p_{xz} \in [x,z]$ such that 
\begin{equation*}
\hil \big( \pi_\ell(x),p_{xz}\big) \leq 3 \Cell.
\end{equation*}
\item $\ell$ is $\PS^\Om$-contracting with constant $\Cell$ and the map $\pi_\ell$.
\end{enumerate} 
\end{lemma}
\begin{proof} 
(1) ~~ Let $x \in \Om$ and $z \in \ell$. Choose any $C_{\ell} \geq D_\ell$, where $D_\ell$ is the constant from  \cref{prop:Dl-thin}. \cref{prop:Dl-thin} implies that $\Delta(x, \pi_\ell(x),z)$ is $\Dell$-thin.  By \cref{obs:thinness_implies_special_points}, there exists $p \in [x,\pi_{\ell}(x)]$, $q \in [\pi_\ell(x), z]$ and $r \in [z,x] $ such that 
\begin{equation*}
\hil(q,p) \leq \Dell ~~\text{ and }~~ \hil(q,r) \leq \Dell.
\end{equation*}
Then
\begin{align*}
\hil \big( \pi_\ell(x),p \big) & = \hil \big( \pi_\ell(x),x \big) - \hil \big( p,x \big) \\
& \leq \hil \big( q,x \big) - \hil \big( p,x \big) \leq \hil \big( p,q \big) \leq \Dell.
\end{align*}
Thus
\begin{align*}
\hil \big( \pi_\ell(x),q \big) \leq \hil \big( \pi_\ell(x),p \big) + \hil \big( q,p \big) \leq 2 \Dell.
\end{align*}
Set $p_{xz}:=r.$ Then
\begin{align*}
\hil \big( \pi_\ell(x),p_{xz} \big) \leq \hil \big( \pi_\ell(x), q \big)+ \hil \big( q,r \big) \leq 3 \Dell \leq 3 C_\ell.
\end{align*}

(2) ~~  Set $\pi:= \pi_\ell$ for ease of notation. Let us label the endpoints of $\ell$ so that $\ell :=(a,b).$ Observe that we only need to verify (2) in Definition \ref{defn-contraction-sisto}. Suppose, for a contradiction, that it is not satisfied. Then, for every $n \in \Nb$, there exist $x_n,y_n \in \Om$ such that 
\begin{equation*}
\hil \big( \pi(x_n), \pi(y_n) \big) \geq n
\end{equation*}
and
\begin{equation*}
\hil \big( [x_n,y_n],\pi(x_n) \big) \geq n.
\end{equation*} 

Since $\ell$ is a rank one axis, fix a rank one isometry $\g$ whose axis is $\ell$. Then $\g \circ \pi =\pi \circ \g.$ Hence, up to translating $x_n$ and $y_n$ using elements in $\langle \g \rangle$, we can assume that  $\alpha:=\lim_{n \to \infty} \pi(x_n)$ exists in $\ell \subset \Om$. Up to passing to a subsequence, we can further assume that the following limits exist in $\clOm$:  $x:=\lim_{n \to \infty} x_n$, $y:=\lim_{n \to \infty} y_n$, $\beta:= \lim_{n \to \infty} \pi(y_n)$. Then $\lim_{n \to \infty}[x_n,y_n] =[x,y]$. We will now show that 
\begin{equation}\label{eqn:[x,y]-in-bdry}
[x,y] \subset \bdry.
\end{equation}
This follows from the following estimate:
\begin{align*}
\hil(\alpha,[x,y]) =\lim_{n \to\infty} \hil(\alpha,[x_n,y_n]) 
&\geq \lim_{n \to \infty} \left( \hil \big( \pi(x_n),[x_n,y_n] \big)- \hil \big( \pi(x_n),\alpha \big) \right) \\
&\geq \lim_{n \to \infty} \left( n - \hil \big( \pi(x_n),\alpha \big) \right)=
\infty.
\end{align*}
We also observe that:
\begin{align*}
\hil(\alpha, \beta) =\lim_{n \to \infty} \hil \big( \alpha,\pi(y_n) \big)  &\geq \lim_{n \to \infty} \left( \hil \big( \pi(x_n),\pi(y_n) \big) - \hil \big( \pi(x_n),\alpha \big) \right) \\
&\geq \lim_{n \to \infty} \left( n - \hil \big( \pi(x_n),\alpha \big)\right) = \infty. 
\end{align*}
Thus $\beta \in \bdry$. However, since $\beta \in \overline{\ell}=[a,b]$, $ \beta \in \big\{ a, b \big\}$. Thus, up to switching the labels of the endpoints of $\ell$, we can assume that
\begin{equation}\label{eqn:beta=b}
\beta=b.
\end{equation}

\begin{claim}\label{claim:x=y=b}
$x=y=b.$
\end{claim}
\begin{proof}[Proof of Claim]\let\qed\relax
We first show that $y=b$. Since $y_n \in \Om$ and $\alpha \in \ell$, part (1) of Lemma \ref{lem:rank1-axis-contracting} implies that there exists $p_n \in [y_n, \alpha]$ such that 
\begin{equation*}
\hil \big( p_n, \pi(y_n) \big) \leq 3 \Cell.
\end{equation*} 
Up to passing to a subsequence, we can assume that $p:=\lim_{n \to \infty} p_n$ exists in $\clOm.$ Then, by Proposition \ref{line-in-bdry-0}, $p \in F_\Om(\beta).$ By equation \eqref{eqn:beta=b}, $\beta=b$ which implies $p \in F_\Om(b)$. Since $b$ is an endpoint of the rank one axis $\ell$, part (4) of Proposition \ref{prop:rank-one-properties} implies that $F_{\Om}(b)=b.$ Thus $p=b$. On the other hand, since $p_n \in [y_n,\alpha]$,  we have $p \in [y,\alpha].$ Since $p =b $, $p \in \bdry$. Thus, 
\begin{align*}
p \in [\alpha,y] \cap \bdry=\big\{ y\big\}.
\end{align*}
 Hence, 
 \begin{align*}
 y=p=b.
 \end{align*}
 
 We now show that  $x=b$. By equation \eqref{eqn:[x,y]-in-bdry}, $[x,y] \subset \bdry$. But since $y=b$, this contradicts part (4) of Proposition \ref{prop:rank-one-properties} unless $x=y$. Hence $x=y=b$. This concludes the proof of Claim \ref{claim:x=y=b}.
\end{proof}

Consider the points $x_n \in \Om$ and $\pi(y_n) \in \ell$. By part (1) of Lemma \ref{lem:rank1-axis-contracting}, there exists $q_n \in \big[ x_n,\pi(y_n) \big]$ such that $\hil \big( \pi(x_n),q_n \big) \leq 3\Cell.$ Up to passing to a subsequence, we can assume that  $q:=\lim_{n \to \infty} q_n$ exists in $\clOm.$ Then by Proposition \ref{line-in-bdry-0},  $q \in F_{\Om}(\alpha)=\Om.$ Thus $\lim_{n \to \infty} [x_n,\pi(y_n)]$ is a projective line segment containing $q$ and hence intersects $\Om$ . However,  $\lim_{n \to \infty}[x_n,\pi(y_n)]=[x,\beta]=\{b\} \subset \bdry$. This is a contradiction. 
\end{proof}

 We will now apply Lemma \ref{lem:rank1-axis-contracting} to prove Theorem \ref{rank1-implies-contraction}. Suppose $\g \in \Aut(\Om)$ is a rank one isometry. Then $\tau_\Om( \g)>0$ which implies that $\g$ has infinite order. By part (2) of Proposition \ref{prop:rank-one-properties}, $\g$ has a unique axis $\ell_{\g}$ along which $\g$ acts by a translation. Fix $x_0 \in \ell_\g$. As $\langle \g \rangle$ acts co-compactly on $\ell_{\g}$, $\langle \g \rangle \cdot x_0$ is a quasi-isometric embedding of $\Zb$ in $\Om$. Part (2) of Lemma \ref{lem:rank1-axis-contracting} implies that $\ell_\g$ is a $\PS^\Om$-contracting set. Thus $\g$ is a contracting element for $(\Om,\PS^\Om)$ (cf. \ref{defn-contracting-element-sisto}).

\section{Contracting isometries are rank one }
\label{sec:contracting-implies-rank-one}

In this section, we prove the other implication of Theorem \ref{thm:contracting-iff-rank-one}. Fix a Hilbert geometry $\Om$ and let $\PS^\Om:=\big\{ [x,y]: x, y \in \Om \big\}$.

\begin{theorem}\label{thm:contracting-implies-rank-1}
If $\g \in \Aut(\Om)$ is a contracting element for $(\Om, \PS^{\Om})$, then $\g$ is a rank one isometry.
\end{theorem}
We begin by recalling a result of Sisto which says that contracting elements are  `Morse' in the following sense.

\begin{proposition}\cite[Lemma 2.8]{sisto_contracting_rw}\label{sisto-morse} If $\PS$ is a path system on $(X,\dist)$ and $\Ac \subset X$ is $\PS$-contracting with constant $C$, then there exists a constant $M=M(C)$ such that: if $\theta$ is a $(C,C)$-quasi-geodesic with endpoints in $\Ac$, then $\theta \subset \Nc_{M}(\Ac):=\left\{ x \in X : \dist(x, \Ac) < M \right\}.$
\end{proposition}

We use this Morse property to show that a contracting element has at least one axis and none of its axes are contained in half triangles in $\Om$. The first step is the following lemma. Recall the notation $E_\g^+, E_\g^-$ from \cref{defn:E_L_K_plus}.

\begin{lemma}\label{lem:axis-of-contracting-elements}
Suppose $\Om$ is a Hilbert geometry and $\g \in \Aut(\Om)$ is a contracting element for $(\Om, \PS^{\Om})$. If there exist  $x_0 \in \Om$ and two unbounded sequences of positive integers $\{n_k\}_{k \in \Nb}$ and $\{m_k\}_{ k \in \Nb}$ such that 
\begin{equation*}
p:=\lim_{k \to \infty} \g^{n_k}x_0 ~~~ \text{ belongs to } E_\g^+
\end{equation*}
and 
\begin{equation*}
q:=\lim_{k \to \infty} \g^{-m_k}x_0 ~~~\text{ belongs to } E_\g^-, 
\end{equation*} 
then
\begin{enumerate}
\item $(p,q) \subset \Om$, and 
\item $(p,q)$ is not contained in any half triangle in $\Om$.
\end{enumerate}
\end{lemma}
\begin{proof} Since $\g$ is a contracting element, \cref{prop:on-contracting-elements} implies that $\tau_{\Omega}(\g)>0.$  Thus $p \neq q$.

\noindent \textbf{(1)} \, Suppose this is false. Then $[p,q] \subset \bdry$. Choose any $r \in (p,q).$ Set $L_k:=\big[\g^{-m_k}x_0,\g^{n_k}x_0 \big]$. Then $L_{\infty}:=\lim_{k \to \infty} L_k=[q,p].$ Thus we can choose $r_k \in L_k$ such that $\lim_{k \to \infty} r_k=r.$ 

Since $\g$ is a contracting element, part (2) of \cref{prop:on-contracting-elements} implies that $\Amin:=\langle \g \rangle x_0$ is $\PS^\Om$-contracting. Since the $L_k$ are geodesics with endpoints in $\Amin$, Proposition \ref{sisto-morse} implies that there exists a constant $M$ such that for all $k \geq 1$, $L_k \subset \Nc_{M}(\Amin)$. Thus for every $k \geq 1$, there exists $\g^{t_k}x_0 \in \Amin$ such that 
\begin{equation}\label{eqn:ak-rk-M}
\hil(r_k,\g^{t_k}x_0) \leq M.
\end{equation}
Up to passing to a subsequence, we can assume that 
\begin{equation*}
t:= \lim_{k \to \infty} \g^{t_k}x_0
\end{equation*}
exists in $\clOm$. Since $r_k$ leaves every compact subset of $\Om$, $\{t_k\}$ is an unbounded sequence. Then by Proposition \ref{prop:omega-limit-gamma} part (1), $t \in \big( E_\g^+ \sqcup E_\g^- \big)$. On the other hand, by Proposition \ref{line-in-bdry-0} and  equation \eqref{eqn:ak-rk-M}, 
\begin{align}
\label{eqn:t_in_face}
t \in F_{\Omega}(r) \subset \bdry.
\end{align}  We now analyze the two possibilities:

\noindent \textbf{Case 1.} If possible, suppose $t \in E_{\g}^-.$ Then consider the sequence $\{\g^{n_k}r\}_{k\in \Nb}$. Up to passing to a subsequence, we can assume that $r_\infty:=\lim_{k \to \infty}\g^{n_k}r$ exists in $\bdry.$ Since $p \in E_\g^+$, $q \in E_\g^-$ and $r \in (p,q)$ with $n_k>0$, \cref{obs:iteration_jordan_blocks} part (2) implies that 
\begin{equation}\label{eqn:r-inf-in-E+}
r_\infty =\lim_{k \to \infty} \g^{n_k}r \in E_\g^+.
\end{equation}
To sum up, we have $r \in F_{\Omega}(t)$ where $t \in E_\g^-$ and $r_\infty=\lim_{k \to \infty}\g^{n_k}r$ (cf. \ref{eqn:t_in_face} and \ref{eqn:r-inf-in-E+}). Now we apply part (2) of \cref{no-face-shared-by-E+-E-} with $t, r$ and $\{n_k\}$ taking the role of $y,z$ and $\{i_k\}$ respectively. Then the conclusion is that $r_\infty \in E_\g^-$. This contradicts equation \eqref{eqn:r-inf-in-E+}.

\noindent \textbf{Case 2:} If possible, suppose $t \in  E_\g^+$. We can repeat the same arguments as in Case 1 by considering the sequence $\{\g^{-m_k}r\}_{k\in \Nb}$ and arrive at a contradiction (we need a version of \cref{no-face-shared-by-E+-E-} with $\g$ replaced by $\g^{-1}$; see \cref{rem:analogues_of_face_results}).

The contradiction to both of these possibilities finishes the proof of (1).

\medskip

\noindent \textbf{(2)} \, By part (1), $(p,q)\subset \Om.$ Suppose there exists $z \in \bdry$ such that $p, z, q$ form a half triangle in $\Om$. Choose any sequence of points $z_k \in [x_0,z] \cap \Om$ such that $\lim_{k \to \infty} z_k=z.$ Since $\g$ is contracting, part (2) of \cref{prop:on-contracting-elements} implies that $\Amin= \langle \g \rangle x_0$ is $\PS^\Om$-contracting (with  constant, say $C$). Thus there exists a projection $\pi: \Om \to \Amin$ that satisfies Definition \ref{defn-contraction-sisto}. We will analyze the sequence $\pi(z_k)$.  Since $\pi(z_k) \in \Amin$, there exists a sequence of integers $\{ i_k\}$ such that  $\pi(z_k)=\g^{i_k}x_0.$ Up to passing to a subsequence, we can assume that the following limit exists in $\clOm$: 
\begin{equation}\label{eqn:defn-w}
w := \lim_{k \to \infty} \pi(z_k) =\lim_{k \to \infty} \g^{i_k} x_0.
\end{equation}

\begin{claim}\label{claim:w-in-eigenspace}
$w \in \bdry$ and $w \in (E_\g^+ \sqcup E_\g^-).$
\end{claim}
\begin{proof}[Proof of Claim]\let \qed \relax
Recall that $\Gamma$ acts properly discontinuously on $\Omega$. Moreover, $\omega(\gamma,\Omega) \cup \omega(\gamma^{-1},\Omega) \subset  E_{\g}^+ \cup E_{\g}^-$ (cf. \cref{prop:omega-limit-gamma}). Thus it suffices to show that $\{i_k\}$ is an unbounded sequence. Suppose, on the contrary, that $\{ i_k\}$ is a bounded sequence. Then $w \in \Om$ and $\lim_{k \to \infty} \hil(w,\pi(z_k))=0$, see \eqref{eqn:defn-w}.

Recall that $\{n_k\}$ is the sequence such that $\g^{n_k} x_0 \to p \in \bdry$. We will prove this claim by comparing $\g^{i_k}x_0(=\pi(z_k))$ with $\g^{n_k}x_0$. We claim that $\hil(\pi(z_k),\pi(\g^{n_k}x_0)) \to \infty$ as $k \to \infty$. To prove this subclaim, first note that (1) of Definition \ref{defn-contraction-sisto} implies that $$\hil(\g^{n_k}x_0, \pi(\g^{n_k}x_0)) \leq C$$ because $\g^{n_k}x_0 \in \Amin$. The subclaim then follows from the next equation:
\begin{align*}
\lim_{k \to \infty} \hil(\pi(z_k),\pi(\g^{n_k}x_0) & \geq \lim_{k \to \infty} \Big( \hil(w,\g^{n_k}x_0) - \hil(w,\pi(z_k)) - \hil(\g^{n_k}x_0,\pi(\g^{n_k}x_0)) \Big) \\
& \geq \liminf_{k \to \infty}  \hil(w,\g^{n_k}x_0)-  C = \infty
\end{align*}
The above equation then implies that for $k$ large enough, $\hil \big( \pi(z_k),\pi(\g^{n_k}x_0) \big) \geq C$. Since $\pi$ is a projection into a $\PS^\Om$-contracting set,  (2) of Definition \ref{defn-contraction-sisto} implies that
\begin{equation*}
\hil \big( \pi(z_k),[z_k,\g^{n_k}x_0] \big) \leq C.
\end{equation*}
Thus 
 \begin{equation*}
 \hil(w,[z,p]) \leq \lim_{k \to \infty} \hil(\pi(z_k),[z_k,\g^{n_k}x_0]) \leq C.
 \end{equation*} 
 Then $[z,p] \cap \Om \neq \emptyset$. But since $p$, $z$, $q$ form a half triangle, $[z,p] \subset \bdry$. This is a contradiction and it concludes the proof of this claim.
\end{proof}

 \begin{claim}\label{claim:w-in-face-of-z}
$w \in F_{\Om}(z)$.
\end{claim}
\begin{proof}[Proof of Claim]\let\qed \relax
First observe that for $k$ large enough, $$\hil \big( \pi(z_k),\pi(x_0) \big) \geq C.$$ Indeed, this follows because $\pi(x_0) \in \Omega$ while $w = \lim_{k \to \infty} \pi(z_k) \in \bdry$. Again, as $\pi$ is a projection into a $\PS^\Om$-contracting set, we have 
 \begin{equation*}
 \hil(\pi(z_k),[ x_0,z_k]) \leq C.
 \end{equation*} 

 Choose $\eta_k \in [x_0, z_k]$ such that $\hil(\pi(z_k),\eta_k) \leq C.$ Up to passing to a subsequence, we can assume that $\eta:=\lim_{k \to \infty} \eta_k$ exists. By Proposition \ref{line-in-bdry-0}, $\eta \in F_{\Omega}(w).$ Since $w \in \bdry$, $\eta \in \bdry$ (Proposition \ref{prop:faces} part (1)). But $\eta \in [ x_0,z]$, which intersects $\bdry$ at exactly one point, namely $z$. Thus, $\eta =z$ implying $z \in F_{\Omega}(w)$, or equivalently, $w \in F_{\Omega}(z)$. This concludes the proof of Claim \ref{claim:w-in-face-of-z}.
\end{proof}

\medskip

Since $p,z,q$ form a half triangle, $[p,z] \cup [z,q] \subset \bdry$. By Claim \ref{claim:w-in-face-of-z}, $w \in F_{\Omega}(z)$. Then part (4) of Proposition \ref{prop:faces}  implies that 
\begin{equation}\label{eqn:pwq-in-bdry}
[p, w] \cup [q,w] \subset \bdry.
\end{equation}
Recall from Claim \ref{claim:w-in-eigenspace} that $\{i_k\}$ is an unbounded sequence and that $w = \lim_{k\to \infty} \gamma^{i_k}x_0$ lies in $E_\g^+ \sqcup E_\g^-$. We will now show that \eqref{eqn:pwq-in-bdry} contradicts this. 

Suppose, up to passing to a subsequence, that $\{i_k\}$ is a seqeunce of positive integers. Then $w \in E_\g^+$.  Since $\lim_{k\to\infty} \g^{i_k}x_0=w \in E_\g^+$ and $\lim_{k \to \infty} \g^{-m_k}x_0 = q \in E_\g^-$, then part (1) of Lemma \ref{lem:axis-of-contracting-elements} implies that $(w,q) \subset \Om$. This contradicts \eqref{eqn:pwq-in-bdry}.  On the other hand, if we suppose that $\{i_k\}$ is a sequence of negative integers, then $w \in E_\g^-$.  Then, by a similar reasoning, $(p,w) \subset \Om$ which again contradicts \eqref{eqn:pwq-in-bdry}.  These contradictions show that $p$, $z$, $q$ cannot form a half triangle. 
\end{proof}

We now prove Theorem \ref{thm:contracting-implies-rank-1} using the above lemma. Let $\g \in \Aut(\Om)$ be a contracting element for $(\Om, \PS^\Om)$. By part (1) of \cref{prop:on-contracting-elements}, $\tau_{\Om}(\g)>0$. The following will imply that $\g$ is a rank one isometry.\\

 \noindent  \textbf{$\mathbf{\g}$ has an axis:}  By Proposition \ref{prop:pseudo_axis}, there exists $(a,b) \subset \clOm$ with $a,b$ fixed points of $\g$ such that $a \in E_{\g}^+$ and $b \in E_{\g}^-$.  We will show that $(a,b) \subset \Om$; hence it  is  an axis of $\g$.
 
Fix $x_0 \in \Omega$. Proposition \ref{prop:omega-limit-gamma} part (1) implies  $\{\g^n x_0 :  n \in \Nb \}$ has an accumulation point $p$ in $E_\g^+$ and $\{\g^{-n} x_0: n \in \Nb\}$ has accumulation point $q$ in $E_\g^-$. By part (1) of Lemma \ref{lem:axis-of-contracting-elements}, $(p,q) \subset \Om.$ 

Note that $E_{\g}^+ \cap \clOm \subset \bdry$ (cf. \cref{claim:about_jordan_subspace_intersections}). Thus $[a,p] \subset \bdry$ as $a, p \in E_\g^+ \cap \bdry$. Similarly, $[b,q] \subset \bdry$. By Part (2) of Lemma \ref{lem:axis-of-contracting-elements}, $(p,q) \subset \Om$ is not contained in any half triangle in $\Om$. Since $[b,q] \subset \bdry$, this implies that $(p,b) \subset \Omega.$ 

We will use $(p,b) \subset \Om$ to derive that $(a,b) \subset \Om$. First note that since $b \in E_{\g}^-$ is the endpoint of a pseudo-axis, $b$ is a fixed point of $\g$.  Thus $\lim_{k \to \infty} \g^{-k} y' =b \in  E_\g^-$ for any $y' \in (p,b)$. We then note that $p$ is an `almost-fixed' point of $\g$, i.e. there exists $\{n_k\}$ with $n_k \to \infty$ such that $\lim_{k \to \infty} \g^{n_k}p=p \in E_{\g}^+$. Indeed, Proposition \ref{prop:omega-limit-gamma} part (3) implies that there exists a sequence of positive integers $\{n_k\}$ with $n_k\to \infty$ such that $\lim_{k \to \infty} \g|_{E_{\g}^+}^{n_k}=\id_{E_\g^+}$, i.e. $\lim_{k \to \infty} \g^{n_k}p=p$. Now pick $y_0 \in (p,b) \in \Om$.  The above discussion implies that $\lim_{k \to \infty} \g^{n_k}y_0=p \in E_\g^+$ while $\lim_{k \to \infty} \g^{-k} y_0 =b \in  E_\g^-.$ Then, by Part (2) of Lemma \ref{lem:axis-of-contracting-elements}, $(p,b) \subset \Om$ cannot be contained in a half triangle in $\Om$. But we know that $[a,p] \subset \bdry$. Thus, $(a,b) \subset \Omega$. \\

\noindent \textbf{None of the axes of $\g$ are contained in a half triangle in $\Om$:}  Let $(a',b') \subset \Omega$ be any axis of $\g$ with $a' \in E_{\g}^+$ and $b' \in E_{\g}^-$. If  $z_0 \in (a',b')$, then $\lim_{k \to \infty}\g^k z_0=a'$ and $\lim_{k \to \infty} \g^{-k}z_0=b'. $ Then, by Part (2) of Lemma \ref{lem:axis-of-contracting-elements}, $(a',b')$ cannot be contained in a half triangle in $\Om$.

\section{Acylindrical Hyperbolicity: Proof of Theorem \ref{thm:acy-hyp}}
\label{sec:acy_hyp}

Acylindrically hyperbolic groups are a generalization of non-elementary Gromov hyperbolic groups with many interesting examples, like mapping class groups of most finite-type surfaces, rank one $\CAT(0)$ groups that are not virtually cyclic, outer automorphisms of free groups on at least two generators and relatively hyperbolic groups with proper peripheral subgroups that are not virtually cyclic \cite[Appendix]{osin_acy_hyp}. In this section, we will add a new class of examples by showing that discrete groups acting on Hilbert geometries with at least one rank one isometry are  either virtually cyclic or acylindrically hyperbolic.

\subsection{Acylindrically hyperbolic groups}
We first recall some basic definitions about Gromov hyperbolic metric spaces (not necessarily proper) and we refer to  \cite{MG1987} for details. A geodesic metric space $(Y,\dist_Y)$ is called \emph{Gromov hyperbolic} if there exists $\delta \geq 0$ such that every geodesic triangle in $Y$ is $\delta$-thin (recall \cref{defn:thin_triangles}). If $(Y,\dist_Y)$ is Gromov hyperbolic, let $\partial Y$ denote the \emph{boundary of $Y$} defined via equivalence classes of sequences in $Y$ ``convergent at infinity", see \cite[Section 1.8]{MG1987}. We remark that this definition of $\partial Y$ does not require that $Y$ is a proper metric space. 

If $G$ acts isometrically on a Gromov hyperbolic space $(Y,\dist_Y)$, let $\Lambda_G(Y) \subset \partial Y$ denote the \emph{limit set} of the $G$-action (i.e, $\Lambda_G(Y)$ is the set of accumulation points  in $\partial Y$ of any $G$ orbit in $Y$). The action is called \emph{non-elementary} if $\#(\Lambda_G(Y))=\infty$; see \cite{osin_acy_hyp} for details.

Finally we define the notion of acylindrical actions on a metric space (not necessarily Gromov hyperbolic). An isometric action of a group $G$ on a metric space $(Y,\dist_Y)$ is called \emph{acylindrical} if: for every $\e >0$, there exists $R_\e,N_\e >0$ such that if $x,y \in Y$ with $\dist_Y (x,y) \geq R_\e$, then
\begin{equation*}
\# \left\{ g \in G : \dist_Y (x,gx) \leq \e \text{ and } \dist_Y (y,gy)\leq \e \right\} \leq N_\e. 
\end{equation*}

\begin{definition}\label{defn-acy-hyp-1}
A group $G$ is called \emph{acylindrically hyperbolic} if it admits an isometric non-elementary acylindrical action on a (possibly non-proper) Gromov hyperbolic metric space $(Y,\dist_Y)$.
\end{definition}

A motivating example of acylindrically hyperbolic groups is a non-elementary Gromov hyperbolic group. Indeed, if $H$ is a finitely generated non-elementary Gromov hyperbolic group, then it has a non-elementary acylindrical action on its Cayley graph which is a Gromov hyperbolic metric space. More generally if $H$ is a finitely generated non-elementary relatively hyperbolic group with proper peripheral subgroups, then $H$ has an acylindrical action on its coned-off Cayley graph. Another interesting example is the mapping class group of a closed hyperbolic surface. It acts acylindrically and non-elementarily on the curve graph of the surface, which is a (non-proper) Gromov hyperbolic space.

Although \cref{defn-acy-hyp-1} of acylindrically hyperbolic groups  is perhaps the cleanest to state,  a characterization of acylindrically hyperbolic groups using contracting elements will be particularly well-suited for our purpose. We state such a characterization now which  follows directly from work of Osin and Sisto; a proof is included because we could not find a result stated in this form. 

\begin{theorem}\cite{osin_acy_hyp, sisto_contracting_rw}
\label{thm:contracting-equiv-acy-hyp}
Suppose $G$ has a proper isometric action on a geodesic metric space $(X, \dist)$, $(X,\PS)$ is a path system for $G$ and $g \in G$ is a contracting element for $(X,\PS)$. Then, either $G$ is virtually cyclic or $G$ is acylindrically hyperbolic. 
\end{theorem}
\begin{proof}[Sketch of proof of \cref{thm:contracting-equiv-acy-hyp}]
In \cite{osin_acy_hyp}, Osin introduces several characterizations of acylindrically hyperbolic groups that are equivalent to Definition \ref{defn-acy-hyp-1}.  The one that we will use (Proposition \ref{defn-acy-hyp-2}) requires the notion of  \emph{hyperbolically embedded subgroups}.  Results due to Osin and Sisto (cf. \ref{defn-acy-hyp-2} and \ref{prop:Eg}) will allow us to use this notion without defining it precisely. See \cite[Definition 2.8]{osin_acy_hyp}  or \cite[Definition 4.6]{sisto_contracting_rw} for details. 

In \cite{osin_acy_hyp}, Osin proves:
\begin{proposition}[{\cite[Theorem 1.2 and Definition 1.3]{osin_acy_hyp}} and Remark \ref{rem:acy-hyp-2}]
\label{defn-acy-hyp-2} 
A group $G$ is acylindrically hyperbolic if $G$ contains a proper infinite hyperbolically embedded subgroup.
\end{proposition}

So in order to prove that a group is acylindrically hyperbolic, it suffices to produce a proper infinite hyperbolically embedded subgroup. For this, we rely on a result of Sisto. 
\begin{proposition}[\protect{\cite[Theorem 4.7]{sisto_contracting_rw}}]\label{prop:Eg} Suppose $g \in G$ is a contracting element for $(X,\PS)$ and $\Ac \subset X$ is $\langle g \rangle$-invariant, $\PS$-contracting and has co-bounded $\langle g \rangle$ action. Then 
\begin{align*}
E(g):=  \{ h \in G : d^{\Haus} (\pi_{\Ac}(h\Ac),\Ac) < \infty \}
\end{align*}
is a hyperbolically embedded subgroup of $G$ which is infinite and contains $\langle g \rangle$ as a finite index subgroup, i.e. $E(g)$ is virtually cyclic.
\end{proposition}

Now let us summarize how these results give us our desired conclusion. Suppose $g \in G$ is a contracting element.  By \cref{prop:Eg}, $E(g)$ is an infinite hyperbolically embedded subgroup of $G$ which is virtually cyclic. Now note that if $G$ is virtually cyclic, there is nothing to prove. So suppose that $G$ is not virtually cyclic. Then $E(g) \subsetneq G$ as $E(g)$ is virtually cyclic. Thus $E(g)$ is a proper infinite hyperbolically embedded subgroup and Proposition \ref{defn-acy-hyp-2} implies that $G$ is an acylindrically hyperbolic group. See the following remark for further comments on the proof.
\end{proof}

\begin{remark}
\label{rem:acy-hyp-2}
Recall the alternate definition of an acylindrically hyperbolic group given by \cref{defn-acy-hyp-2}. A subgroup $H \leq G$ is \emph{proper infinite} if $H \subsetneq G$ and $H$ is infinite. Such proper infinite hyperbolically embedded subgroups are sometimes called \emph{non-degenerate} hyperbolically embedded subgroups in the terminology of \cite{osin_acy_hyp, DGO2017}. Notably, the existence of one such non-degenerate hyperbolically embedded subgroup $H \leq G$ implies the existence of non-Abelian free subgroups in $G$,  see \cite[Lemma 5.12, 5.4]{osin_acy_hyp} or \cite[Theorem 6.14]{DGO2017}. Thus $G$ is not virtually cyclic and contains infinitely many `independent loxodromic' elements \cite{osin_acy_hyp,DGO2017}. Roughly speaking, this is akin to producing non-Abelian free subgroups in any non-elementary Gromov hyperbolic group. 
\end{remark}

\subsection{Proof of \cref{thm:acy-hyp}}

We first recall the theorem. 

\noindent \textbf{Theorem \ref{thm:acy-hyp}.} \emph{If $(\Om, \G)$ is a rank one Hilbert geometry, then either $\G$ is virtually cyclic or $\G$ is an acylindrically hyperbolic group.}

The proof of \cref{thm:acy-hyp} will be immediate from \cref{thm:contracting-equiv-acy-hyp}, thanks to the well-developed machinery of acylindrically hyperbolic groups due to the work of many authors, see for instance \cite{osin_acy_hyp, DGO2017, sisto_contracting_rw, BBFS2019}. In case the proof seems a bit opaque to a reader, we will first give an informal sketch of the underlying idea before providing a formal proof. 

Our result \cref{thm:contracting-iff-rank-one} implies that rank one isometries in $\Aut(\Om)$ are contracting elements for $(\Om,\PS^{\Om})$. Thus a rank one Hilbert geometry $(\Om,\G)$ contains contracting elements by definition. Now it is possible that $\Gamma$ is virtually cyclic in which case $\Gamma$, up to passing to a finite index subgroup, is generated by a single rank one isometry. But if $\G$ is not virtually cyclic, then there will be infinitely many rank one isometries $\g_1,\g_2,\dots$ which are `independent loxodromics' (i.e., there exists an abstract Gromov hyperbolic space $X$ on which each $\g_i$ acts `loxodromically' with exactly two distinct fixed points $\g_i^{\pm}$ and the sets $\{ \g_i^{\pm}\}$ and $\{\g_j^{\pm}\}$ are pairwise disjoint whenever $i \neq j$). This last conclusion follows from results in \cite{DGO2017} and  \cite{sisto_contracting_rw} that we referred to in \cref{rem:acy-hyp-2}. These infinitely many independent rank one isometries $\g_i$ generate non-Abelian free subgroups of $\G$ and the $\g_i$ lie in distinct hyperbolically embedded subgroups $E(\g_i)$, see \cref{prop:Eg}. 

Now let us give the formal proof. 

\begin{proof}[Proof of \cref{thm:acy-hyp}]
Since $(\Om, \G)$ is a rank one Hilbert geometry, $\G$ contains a rank one isometry. Then Theorem \ref{thm:contracting-iff-rank-one} implies that $\G$ contains a contracting element for $(\Om, \PS^\Om)$. The result follows from Theorem \ref{thm:contracting-equiv-acy-hyp}. 
\end{proof}

\begin{remark}
 By \cref{thm:acy-hyp}, a rank one Hilbert geometry $(\Om,\G)$ where $\G$ is not virtually cyclic gives an example of an acylindrically hyperbolic group $\G$. A natural question is: what is an example of a Gromov hyperbolic metric space $X$ on which $\G$ acts acylindrically and non-elementarily? Is there a way to understand this space $X$ in terms of the Hilbert geometry $\Om$?
 
 It seems that one might be able to apply the projection complex construction in \cite{BBFS2019} (see also \cite{BBF15, sisto_contracting_rw}) to construct such a space $X$  from the Hilbert geometry $\Om$. Roughly speaking, this will be a metric space obtained by collecting all rank one axes in $\Om$ and adding edges between them depending on diameters of images of some projection maps. We do not pursue this  direction in this paper and this remark is mostly speculative in nature.
 \end{remark}

\part{Applications}
\label{part:applications}

\section{Second Bounded Cohomology and Quasi-morphisms}
\label{sec:second_bdd_cohom}

\subsection{Definitions} We first introduce some definitions following \cite[Section 1]{BBF13}. Suppose $G$ is a group, $(E,||\cdot||)$ is a complete normed $\Rb$-vector space and  $\rho: G \to \Uc (E)$ is a unitary representation. Let $C(G,E)$ be the space of all functions from $G$ to $E$. 

A function $F \in C(G,E)$ is called a quasi-cocycle if   
\begin{equation*}
\Delta(F) := \sup_{g,g' \in G}||F(g g')-F(g) - \rho(g)F(g')|| < \infty.
\end{equation*}
Let $V$ be the vector sub-space of $C(G,E)$ that consists of all quasi-cocyles. Let $V_0$ be the sub-space of $V$ generated by bounded functions and the set  $\{F:G \to E | F(gg')=F(g)+\rho(g)F(g') \forall g,g ' \in G\}$. Define  $$\T{QC}(G;\rho):=V/V_0.$$

If $\rho$ is the trivial representation $\rho_{\rm triv}: G \to \Rb$, then $V$ is the space of quasi-morphisms of $G$ while $V_0$ is the space generated by bounded functions and group homomorphisms from $G$ to $\Rb$. In this case, $\T{QC}(G;\rho_{\rm triv})$ recovers a classical object called the space of `non-trivial' quasi-morphisms of $G$, usually denoted by $\T{QH}(G)$ (see the definitions preceding Theorem \ref{thm:rk-1-hil-infinite-QH}). 

 Group cohomology of $G$ (twisted by the representation $\rho$) affords an interesting interpretation of $\T{QC}(G;\rho)$. If $F$ is a quasi-coycle, then $dF(g,g'):=F(gg')-F(g)-\rho(g)F(g')$ defines a class in the second bounded cohomology group $H_b^2(G;\rho)$. This class $dF$ is trivial in the ordinary cohomology group $H^2(G;\rho)$.  On the other hand, the class $dF$ is non-trivial in $H_b^2(G;\rho)$ whenever $F$ is non-trivial in $V/V_0$. Thus $\wt{QC}(G;\rho)$ is  the kernel of the comparison map $H_b^2(G, \rho) \rightarrow H^2(G ; \rho)$. For a more detailed discussion, we refer the reader to \cite[Section 1]{BBF13} or  \cite{bounded_cohom_book}.

\subsection{Results} Infinite dimensionality of $\T{QH}(G)$ (and $\T{QC}(G;\rho$) is often related to geometric phenomena. For example, \cite{bestvina_fujiwara_symm_space} shows that a compact irreducible non-positively curved Riemannian manifold $M$ is (Riemannian) rank one if and only if $\dim \big( \T{QH}(\pi_1(M))\big) =\infty$. Now, in the same spirit as in Riemannian non-positive curvature, we prove a cohomological characterization of rank one Hilbert geometries. We will only consider unitary representations on uniformly convex Banach spaces\footnote{A Banach space $E$ is uniformly convex if for any $\e'>0$, there exists $\delta'>0$ such that if $u,v\in E$, $||u|| \leq 1$, $||v|| \leq 1$, $||u-v|| \geq \e'$, then $||\frac{u+v}{2}|| \leq 1-\delta'$.} (e.g. $\Rb$ or $\ell^p(G)$ where $G$ is a discrete group and $1< p < \infty$).

\begin{theorem}\label{thm:rk-1-hil-infinite-QC-rho}
If $(\Omega,\G)$ is a rank one Hilbert geometry, $\G$ is torsion-free and $\rho$ is any unitary representation of $~\G$ on a uniformly convex Banach space $E \neq 0$, then either $\G$ is virtually cyclic or $\dim \big( \T{QC}(\G;\rho) \big)=\infty$.
\end{theorem}
The proof follows directly from the following general result about acylindrically hyperbolic groups.
\begin{theorem}{\cite[Corollary 1.2]{BBF13}}
\label{thm:BBF13}
If $G$ is an acylindrically hyperbolic group, $E \neq 0$ is a  uniformly convex Banach space, $\rho: G \to \Uc(E)$ is a unitary representation  and any maximal finite normal subgroup of $G$ has a non-zero fixed vector, then $\dim \big( \T{QC}(G; \rho) \big) = \infty.$  \qedhere
\end{theorem}
\begin{proof}[Proof of Theorem \ref{thm:rk-1-hil-infinite-QC-rho}] If $\G$ is not virtually cyclic, then Theorem \ref{thm:acy-hyp} implies that $\G$ is an acylindrically hyperbolic group. Since $\G$ is torsion-free, there are no finite normal subgroups. The claim then follows from Theorem \ref{thm:BBF13}.
\end{proof}

\noindent  We will now apply \cref{thm:rk-1-hil-infinite-QC-rho} to two specific choices of $\rho$ and $E$ to get Theorem \ref{thm:rk-1-hil-infinite-QH}. For the first one, $\rho=\rho_{\rm triv}$ and $E=\Rb$ in which case $\wt{QC}(\Gamma;\rho)=\wt{QH}(\G)$, the space of non-trivial quasi-morphisms. For the second one, $E=\ell^p(\Gamma)$ with $1<p <\infty$ and $\rho=\rho^p_{\rm reg}$ is the regular representation (i.e. $\rho^p_{\rm reg}(\g)  f(x)=f(\g^{-1}x)$ for any $f \in \ell^p(\G)$ and $x \in \G$).

\textbf{Theorem \ref{thm:rk-1-hil-infinite-QH}.} \emph{If $(\Omega,\G)$ is a rank one Hilbert geometry, $\G$ is torsion-free and $\G$ is not virtually cyclic, then $\dim \big( \T{QH} (\G) \big) = \infty$ and $\dim \big( \T{QC}(\G;\rho^p_{\rm reg}) \big) =\infty$ if $1< p < \infty$. }
\vspace{-10pt}
\begin{proof}
Immediate from \cref{thm:rk-1-hil-infinite-QC-rho} and the fact that $\Rb$ and $\ell^p(\G)$ with $1 < p <\infty$ are uniformly convex Banach spaces, see \cite[Section 3]{BBF13}.
\end{proof}

\textbf{Corollary \ref{cor:qh-rigid}.} \emph{If $(\Omega,\G)$ is a divisible Hilbert geometry and $\Om$ is irreducible, then $\dim (\T{QH} (\G)) = \infty$ if and only if $(\Om,\G)$ is a rank one Hilbert geometry. Otherwise $\dim (\T{QH}(\G))=0$.}
\vspace{-10pt}
\begin{proof}
If $(\Om,\Gamma)$ is a rank one Hilbert  geometry, then Theorem \ref{thm:rk-1-hil-infinite-QH} implies that $\dim (\T{QH} (\G)) = \infty$. If $(\Om,\G)$ is not rank one, then \cref{thm:rank_rigidity_detailed} implies that $\Aut(\Om)$ is locally isomorphic to a simple Lie group of real rank at least two (i.e. $\Om$ is an irreducible symmetric domain of rank at least two). Thus $\G$ is isomorphic to a uniform lattice in a higher rank simple Lie group which implies that $\dim (\T{QH}(\G))=0$ \cite[Theorem 21]{BM2002}.
\end{proof}

\section{Counting of conjugacy classes}
\label{sec:counting-conjugacy}
Suppose $(\Omega,\G)$ is a rank one Hilbert geometry. Recall the notions of translation length and stable translation length of a conjugacy class in $\G$ (cf. \ref{subsec:intro-counting}). We now introduce the notion of pointed length for a conjugacy class $[c_g]$ of  $g \in \G$ \cite{gekhtman_wenyuan_counting}. Fix a base point $p\in \Omega$. The pointed length of $[c_g]$ is 
\begin{equation*}
 \Lc_p([c_g]):=\inf_{g'\in [c_g]} \hil(p,g'p).
\end{equation*}

We first show that $$\tau_\Omega([c_g])=\tau_{\Omega}^{\st}([c_g]).$$ Indeed, triangle inequality implies $\tau_{\Omega}^{\st}(g) \leq \tau_{\Omega}(g)$. On the other hand, by Proposition  \ref{prop:translation_length},
\begin{align*}
\tau_{\Omega}^{\st}(g) \geq \lim_{n \to \infty} \dfrac{\tau_{\Omega}(g^n)}{n}=\dfrac{1}{n}\log \dfrac{\ev_{\max}(\wt g^n)}{\ev_{\min}(\wt g^n)}= \log \dfrac{\ev_{\max}(\wt g)}{\ev_{\min}(\wt g)} = \tau_{\Omega}(g).
\end{align*}
Next, we show that if  $\Omega/\G$ is compact and $R:= \diam (\Omega/\G)$, then \begin{equation*}
\tau_{\Omega}([c_g]) \leq \Lc_p([c_g]) \leq \tau_{\Omega}([c_g])+2R.
\end{equation*} 
Clearly $\tau_{\Omega}([c_g])\leq \Lc_p([c_g])$. On the other hand, if $x \in \Omega$ then there exists $h_x \in \G$ such that $\hil(x,h_xp) \leq R$. Then
\begin{align*}
\Lc_p([c_g]) \leq \hil(p,h_x^{-1}gh_xp)\leq 2\hil(h_xp,x) + \hil(x,gx) \leq 2R +\hil(x,gx).
\end{align*}
Thus $\Lc_p([c_g]) \leq \tau_{\Omega}([c_g])+2R$.

Now let us consider the following counting functions for conjugacy classes in $\G$:
\begin{align*}
\Cc(t)&:=\#\{ [c_g] : g \in \G, \tau_{\Omega}([c_g])\leq t\},\\
\Cc^{\st}(t)&:=\#\{ [c_g] : g \in \G, \tau_{\Omega}^{\st}([c_g])\leq t\} \text{ and } \\
\Cc^{\Lc_p}(t)&:=\#\{ [c_g] : g \in \G, \Lc_p([c_g]) \leq t \}
\end{align*}
Based on the above discussion, 
\begin{equation}
\label{eqn:cc-ccst-equal}
\Cc(t)=\Cc^{\st}(t).
\end{equation} If $\Omega/\G$ is compact and $R=\diam(\Omega/\G)$, then 
\begin{equation}
\label{eqn:count-lc}
\Cc^{\Lc_p}(t) \leq \Cc(t) \leq \Cc^{\Lc_p}(t+2R). 
\end{equation}
We now prove asymptotic growth formula for these functions. It is a direct consequence of the Main Theorem in \cite{gekhtman_wenyuan_counting}. Recall that the critical exponent of $\G$ (cf. \ref{subsec:intro-counting}) is defined by 
\begin{equation*}
\omega_{\G}:=\limsup_{n \to \infty}\dfrac{\log \# \{ g \in \G: \hil(x,gx) \leq n\}}{n}.
\end{equation*}

\noindent {\bf Theorem \ref{thm:counting-conjugacy}.}
\emph{Suppose $(\Om, \G)$ is a divisible rank one Hilbert geometry and $\G$ is not virtually cyclic.  Then there exists a constant $D'$ such that for all $t \geq 1$,}
\begin{equation}\label{eqn:count-closed-geod}
\dfrac{1}{D'} \dfrac{\exp (t\omega_{\G})}{t} ~\leq~ \Cc(t) ~\leq~ D' \dfrac{\exp (t\omega_{\G})}{t}.
\end{equation}
\emph{The functions $\Cc^{\st}(t)$, $\Cc^{\Lc_p}(t)$ and $\Cc_{\prim}(t)$ (cf. \ref{rem:counting}) also satisfy similar growth formula.} 
\begin{proof}
The Main Theorem part (1) in \cite{gekhtman_wenyuan_counting} implies that if $\G$ is a non-elementary group with a co-compact action (more generally, statistically convex co-compact action) on a geodesic metric space and $\G$ contains a contracting element (in the sense of BF; cf. \ref{appendix:contracting-BF}), then $\Cc^{\Lc_p}(t)$ satisfies the growth formula in \eqref{eqn:count-closed-geod}. If $(\Omega, \Gamma)$ is as above, then it satisfies all of these conditions (cf. \ref{thm:contracting-iff-rank-one} and \ref{rem:contraction-equiv}). Then $\Cc^{\Lc_p}(t)$ satisfies equation \eqref{eqn:count-closed-geod}. By equations \eqref{eqn:cc-ccst-equal} and \eqref{eqn:count-lc}, $\Cc(t)$ and $\Cc^{\st}(t)$ also satisfies equation \eqref{eqn:count-closed-geod}. 

For proving Remark \ref{rem:counting} part (2), set $$\Cc^{\Lc_p}_{\prim}(t):=\#\{ [c_g]: g \in \G \text{ is primitive}, \Lc_p([c_g])\leq t\}.$$ Main Theorem part (1) in \cite{gekhtman_wenyuan_counting} implies that the $\Cc^{\Lc_p}_{\prim}(t)$  satisfies a  similar growth formula as \eqref{eqn:count-closed-geod}. Since $\tau_{\Om}([c_g]) \leq \Lc_p([c_g]) \leq \tau_{\Omega}([c_g])+2R$, this implies the result for $\Cc_{\prim}(t)$.
\end{proof}

\section{Proofs of Proposition \ref{prop:generic-rw}, Proposition \ref{prop:appl-acy-hyp} and Proposition \ref{prop:rank-one-axis-morse}}
\label{sec:other-proofs}

For the proofs in this section, recall the following implication of Theorem \ref{thm:acy-hyp}: if $(\Omega,\G)$ is a rank one Hilbert geometry and $\G$ is not virtually cyclic, then $\G$ is an acylindrically hyperbolic group. \\

\noindent {\bf Proposition \ref{prop:generic-rw}.}
\emph{If $(\Omega,\G)$ is a rank one Hilbert geometry, $\G$ is not virtually cyclic and $\G$ is finitely generated, then the rank one isometries in $\G$ are exponentially generic:  if $(X_n)_{n \in \Nb}$ is a simple random walk on $\G$, then there exists a constant $C \geq 1$ such that for all $n \geq 1$}
\begin{equation*}\Pb \big[~X_n \text{ is not a rank one isometry} ~\big] \leq C e^{-n/C}.
\end{equation*}
\begin{proof}
Under the hypotheses, $\G$ is an acylindrically hyperbolic group. The result then follows from \cite[Theorem 1.6]{sisto_contracting_rw}.
\end{proof}

\noindent {\bf Proposition \ref{prop:appl-acy-hyp}.} 
\emph{If $(\Om,\G)$ is a rank one Hilbert geometry and $\G$ is not virtually cyclic, then: 
\begin{enumerate}
\item $\G$ is SQ-universal, i.e. every countable group embeds in a quotient of $\G$. 
\item if $\G$ is the Baumslag-Solitar group $\BS(m,n)$, then $m=n=0$ and $\G$ is the free group on two generators.
\end{enumerate}}
\begin{proof} Under the hypotheses, $\G$ is an  acylindrically hyperbolic group. Then SQ-universality follows from \cite[Theorem 8.1]{osin_acy_hyp}. The second part follows from \cite[Example 7.4]{osin_acy_hyp}, where Osin proves that $\BS(m,n)$ is acylindrically hyperbolic if and only if $m=n=0$. But $\BS(0,0)=F_2.$ 
\end{proof}

\noindent {\bf Propostion \ref{prop:rank-one-axis-morse}.}
 \emph{If $\Om$ is a Hilbert geometry and $\g \in \Aut(\Om)$ is a rank one isometry, then the axis $\ell_\g$ of $\g$ is $\Kc$-Morse for some Morse gauge $\Kc: [1,\infty)\times[0,\infty)\to [0,\infty)$, i.e. if $\alpha$ is a $(\lambda,\varepsilon)$-quasi-geodesic  with endpoints on $\ell_\g$, then $\alpha \subset \Nc_{\Kc(\lambda,\varepsilon)}(\ell_\g)$.}
\begin{proof}
Since $\g$ is a rank one isometry, Theorem \ref{thm:contracting-iff-rank-one} implies that $\g$ is a contracting element for $(\Om,\PS^{\Om})$. Then the axis of $\ell_\g$ of $\g$ is $\PS^\Om$-contracting. Thus \cite[Lemma 2.8]{sisto_contracting_rw} (Proposition \ref{sisto-morse} in this paper) implies that $\ell_\g$ is a Morse geodesic. 
\end{proof}

\appendix

\section{Rank one Hilbert geometries: Generalization, Examples and Non-examples}
\label{sec:eg-rank-one}

This section is devoted to the discussion of examples and non-examples of rank one Hilbert geometries (cf. \ref{defn:rank-one-hil-geom}) and generalizing the notion of rank one to convex co-compact actions.  

\subsection{Strictly convex examples.} If $\Omega$ is a strictly convex Hilbert geometry, then $\bdry$  does not contain any line segments. Thus, if $g \in \Aut(\Om)$ with $\tau_{\Om}(g)>0$, then $g$ is a rank one isometry (as it has an axis and there are no half triangles in $\Om$).  Then, all the  strictly convex divisible examples in Section \ref{sec:eg-div-hil-geom} are rank one.

\subsection{Non-strictly convex examples.}
\label{subsec:non-strict-eg-rk-1}

Suppose $(\Omega,\Gamma)$ is a divisible Hilbert geometry where $\Gamma$ infinite and not virtually abelian. Assume that $\Gamma$ is a relatively hyperbolic group with respect to a finite collection of free abelian subgroups of rank at least two.  Then we claim that: \emph{$(\Omega, \Gamma)$ is a divisible rank one Hilbert geometry}. The proof of this follows from Remark \ref{rem:rk-one-cc}(B) and Proposition \ref{prop:iso-sim-implies-rk-1} (see below). This claim implies that the divisible non-strictly convex examples discussed in Section \ref{sec:eg-div-hil-geom}, that are neither simplices nor symmetric domains of rank at least two, are all examples of rank one Hilbert geometries.

\subsection{Non-examples}
\label{subsec:non-eg-rk-1}
The $d$-simplices $T_d$ for $d \geq 2$ are clearly non-examples of rank one Hilbert geometries. If $\Om$ is an irreducible symmetric domain of rank at least two and $\Gamma \leq \Aut(\Om)$ acts co-compactly on $\Om$, then $(\Omega, \G)$ \emph{cannot be a rank one Hilbert geometry}, see \cref{thm:rank_rigidity_detailed}.

\subsection{Generalization of rank one to convex co-compact actions}
\label{subsec:app-conv-cocpt} 
The notion of convex co-compact actions on Hilbert geometries \cite{DGK2017} generalize divisible Hilbert geometries. Suppose $\Omega$ is a Hilbert geometry and $\G \leq \Aut(\Om)$ is a discrete subgroup. The \emph{full orbital limit set} is defined as $\Lc^{\orb}_{\Omega}(\G):= \bigcup_{x \in \Omega}\left( \overline{ \G \cdot x} \cap \bdry \right)$ and let $\convcore:=\CH_{\Omega}(\Lc^{\orb}_{\Omega}(\G)).$ 

\begin{definition}
\label{defn:cc}An infinite discrete group $\G\leq \Aut(\Om)$ is \emph{convex co-compact} if $\convcore \neq \emptyset$ and $\convcore/\G$ is compact. 
\end{definition}

The \emph{ideal boundary} of $\convcore$ is given by $\partiali \convcore:=\bdry~ \cap \overline{\convcore}$. For convex co-compact groups, $\partiali \convcore$ is the only part of $\bdry$ `visible' to the group acting on $\Om$. Thus it is natural to modify the notion of rank one isometries by considering half triangles in $\convcore$ instead of $\Om$. We say that the projective geodesic $(a,b) \subset \convcore$ is \emph{not contained in any half triangle in $\convcore$} if either $(a,z) \subset \convcore$ or $(z,b)\subset \convcore$ for any $z \in \partiali \convcore$.

\begin{definition}
Suppose $\Om \subset \RP$ is a Hilbert geometry and $\G \leq \Aut(\Om)$ is a convex co-compact group. 
\begin{enumerate}
\item An element $\g \in \G$ is a \emph{convex co-compact rank one isometry} if:
\begin{enumerate}
\item $\log \left| \dfrac{\ev_1}{\ev_{d+1}}(\g) \right|>0$ and $~\g$ has an axis (cf. \ref{defn:axis}),
\item none of the axes $\ell_{\g}$ of $\g$ are contained in a half triangle in $\convcore$.
\end{enumerate}  
\item We say that $\Gamma$ is a \emph{rank one convex co-compact group} if $\G$ contains  a   convex co-compact rank one isometry. 
\end{enumerate}
\end{definition}

 \begin{remark}\label{rem:rk-one-cc}\
 \begin{enumerate}[(A)]
 \item The notion of a \emph{convex co-compact rank one isometry} differs from the notion of a \emph{rank one isometry} (cf. \ref{defn:rank-one-isometry}) only in condition (1b): for convex co-compact actions, we consider half triangles in $\convcore$ instead of $\Om$. 
 \item  If $\G$ acts co-compactly on $\Om$, then $(\Omega, \G)$ is a \emph{divisible rank one Hilbert geometry} if and only if $\Gamma$ is a \emph{rank one convex  co-compact group}. This is because divisibility implies $\convcore=\Omega$.
 \end{enumerate}
 \end{remark}
 
If $\G \leq \Aut(\Om)$ is a convex co-compact group and $\g \in \G$ is a convex co-compact rank one isometry, then the analogues of \cref{prop:rank-one-properties}, \ref{cor:ns_dynamics} and \ref{prop:biprox-equiv-no-half-T} hold. But now we need to replace $\Om$ with $\convcore$ and $\bdry$ with $\partiali \convcore$. In particular, we have: $\g \in  \G$ is a convex co-compact rank one isometry if and only  if $\g$ is biproximal and has an axis.

We will briefly sketch the proof ideas of these analogues, see \cite{MI2021} for details. Observe that if $\left| \frac{\ev_1}{\ev_{d+1}}(\g) \right|>0$, then $E_{\g}^{\pm}\cap \bdry=E_{\g}^{\pm} \cap \partiali \convcore$. Recall that if $x \in \partiali \convcore$, then $$F_{\convcore}(x):=\{x\} \cup \{ y \in \overline{\convcore} : \exists \text{ open projective line segment in } \overline{\convcore} \text{ containing } x \text{ and } y \}.$$
Convex co-compact groups have a special property: if $x \in \partiali \convcore$, $F_{\convcore}(x)=F_{\Om}(x)$ \cite[Corollary 4.13]{DGK2017}. Using these properties, one can now see that the proofs in Section \ref{sec:rank-one-isom} go through verbatim after replacing $\Om$ by $\convcore$ and $\bdry$ by $\partiali \convcore$. Thus the analogues of Proposition  \ref{prop:rank-one-properties}, \ref{cor:ns_dynamics} and \ref{prop:biprox-equiv-no-half-T} hold; also see \cite{MI2021}.

\subsection{Convex co-compact examples - Hyperbolic groups.}

Suppose $\G \leq \Aut(\Om)$ is a convex co-compact group  that is word hyperbolic. We claim that: \emph{$\G$ is a rank one convex co-compact group}. Indeed, \cite[Theorem 1.15]{DGK2017} implies that word hyperbolicity of $\G$ is equivalent to the property that  $\partiali \convcore$ does not contain any non-trivial projective line segments. Then there are no half triangles in $\convcore$. Moreover, any infinite-order element $\g$ has an axis \cite[Corollary 7.4]{DGK2017}.  Thus every such $\g$ is a convex co-compact rank one isometry and the claim follows.

\subsection{Convex co-compact examples - Relatively hyperbolic groups.} 
\label{sec:cc-rel-hyp}
\begin{proposition}\label{prop:iso-sim-implies-rk-1}
Suppose $\G\leq \Aut(\Om)$ is a convex co-compact group that is relatively hyperbolic with respect to $\{ A_1, A_2, \ldots, A_m\}$ where each $A_i$ is a virtually free abelian group of rank at least two. Then $\Gamma$ is either a rank one convex co-compact group or a virtually abelian group.
\end{proposition}
 This proposition shows that the divisible examples of Section \ref{subsec:non-strict-eg-rk-1} and  their convex co-compact deformations produce relatively hyperbolic examples that are rank one convex co-compact. We will spend the rest of this subsection proving this proposition. We will rely on results from \cite{IZ2019}. 
 
 \medskip

\noindent \emph{Proof of Proposition \ref{prop:iso-sim-implies-rk-1}.}
Let $\smax$ be the collection of all maximal properly embedded simplices in $\convcore$ of dimension $\geq 2$. Since $\G$ is relatively hyperbolic with respect to virtually abelian subgroups of rank at least two, \cite[Theorem 1.7]{IZ2019} implies that $(\convcore, \hil)$ is a \emph{Hilbert geometry with isolated simplices}, i.e. $\smax$ is closed and discrete in the local Hausdorff topology induced by $\hil$. In this case, \cite[Theorem 1.18]{IZ2019} implies that for each $i \in \{1, \ldots, m\}$, we can assume $A_i=\Stab_{\G} (S_i)$ where $S_i$ is a maximal properly embedded simplex in $\convcore$ of dimension $\geq 2$ and $\smax=\sqcup_{i=1}^m \G \cdot S_i$.  We will require the following result regarding simplices in $\smax$.
\begin{proposition}[{\cite[Theorem 1.8]{IZ2019}}]
\label{prop:facts-iso-simplices}
Suppose $\G$ and $\smax$ are as above. Then:
 \begin{enumerate}
 \item if $[x,y] \subset \partiali \convcore$ with $x \neq y$, then there exists $S \in \smax$ such that $[x,y] \subset \partial S$. 
 \item if $S_1 \neq S_2 \in \smax$, then $\#(S_1 \cap S_2) \leq 1$ and $\partial S_1 \cap \partial S_2=\emptyset$.
 \end{enumerate}
\end{proposition}

Since $\G$ is relatively hyperbolic with respect to $\{ A_1, A_2,\ldots, A_m\}$, \cite[Lemma 2.3]{DAHG2018} implies:
\begin{enumerate}
\item[{\bf (Case 1)}] either $\G$ is virtually $gA_ig^{-1}$ for some $g \in \G$  and $1 \leq i \leq m$,
\item[{\bf (Case 2)}] or there exists $\g \in \G$ such that $\g \not \in \bigcup_{g \in \G}\bigcup_{i=1}^m gA_ig^{-1}=\bigcup_{S \in \smax} \Stab_{\G}(S).$
\end{enumerate}  

In Case 1, $\G$ is a virtually abelian group. So we can now assume that we are in Case 2.

\vspace{0.5em}
\noindent \textbf{Claim: } \emph{If $\g$ is as in Case 2, then $\g$ is a convex co-compact rank one isometry.}
\vspace{0.5em}

From this claim, Proposition \ref{prop:iso-sim-implies-rk-1} is immediate. So all that remains is to prove this claim.

\begin{proof}[Proof of Claim] As $\G$ is a convex co-compact group, $\log \left|\frac{\ev_1}{\ev_{d+1}}(\g)\right|=\tau_{\convcore}(\g)>0$. We first show that $\g$ has an axis in $\convcore$. Let $\Cc^+:=\overline{E_{\g}^+ \cap \convcore}$ and $\Cc^-:=\overline{E_{\g}^- \cap \convcore}$. Then $\Cc^+$ and $\Cc^-$ are disjoint, non-empty, compact, convex, $\g$-invariant subsets of $\Rb^d$. Then the Brouwer fixed point theorem implies the existence of distinct fixed points $\g^{\pm}$ of $\g$ in $\Cc^{\pm}$. If $[\g^+,\g^-] \subset \partiali \convcore$, then Proposition \ref{prop:facts-iso-simplices} implies that there exists $S \in \smax$ such that $[\g^+,\g^-] \subset \partial S$. Then $\partial (\g S) \cap \partial S \supset [\g^+,\g^-]$ and Proposition \ref{prop:facts-iso-simplices} implies that $\g S= S$. Thus, $\g \in \Stab_{\G}(S)$. This contradiction implies that $(\g^+,\g^-) \subset \convcore$ and is an axis of $\g$. 

Suppose $A_\g:=[A_\g^+,A_\g^-]$ is an axis of $\g$ contained in a half triangle in $\convcore$: $[A_{\g}^+,z]\cup[z,A_{\g}^-] \subset  \partiali \convcore$. Then, by Proposition \ref{prop:facts-iso-simplices}, there exist $S^{\pm} \in \smax$ such that $[z,A_\g^{\pm}] \subset \partial S^{\pm}$. Since $z \in \partial S^+ \cap \partial S^-$, Proposition \ref{prop:facts-iso-simplices} implies that $S:=S^+=S^-$ and $A_\g \subset S$. Since $\g$ acts by a translation along $A_\g$, $\g S \cap S \supset A_\g$ which implies $\#(\g S \cap S)=\infty$.  Then by Proposition \ref{prop:facts-iso-simplices}, $\g S= S.$ Thus $\g \in \Stab_{\G}(S)$, a contradiction. Thus $A_{\g}$ is not contained in any half triangle in $\convcore$. This proves the claim.
\end{proof}

\section{Contracting Elements}
\label{appendix:contracting-BF}

Fix a proper geodesic metric space $(X, \dist)$ and a group $G$ that acts properly isometrically on $X$.  If $x \in X$ and $R>0$, let 
$B(x,R):=\{ y \in X : \dist(x,y) < R \}.$ If $\Ac \subset X$ and $x \in X$, let the closest point projection onto $\Ac$ be defined by $\pid_{\Ac}(x):= \{ y \in \Ac : \dist(x,y)=\dist(x,\Ac)\}.$ We let $\Nc_r(\Ac):=\{ y \in X : \dist(y,\Ac)<r\}$ and $\overline{\Nc_r(\Ac)}:=\{ y \in X: \dist(y,\Ac) \leq r\}$) denote the open and the closed $r$-neighborhoods of $\Ac$ respectively. 

In \cite{bestvina_fujiwara_symm_space}, Bestvina-Fujiwara introduced the following notion of contracting subsets. 

\begin{definition}
\label{defn-BF-original} 
A set $\Ac \subset X$ is $B$-contracting  if there exists a constant $B$ such that: if $x \in X$, $R >0$ and $B(x,R) \cap \Ac =\emptyset$, then $\diam \left( \pid _\Ac ( B(x,R) ) \right) \leq B.$
\end{definition}

We will, however, use a related but stronger notion of contracting subsets introduced in \cite{gekhtman_wenyuan_counting}.

\begin{definition}[\cite{gekhtman_wenyuan_counting}]
\label{defn-contracting-BF} 
Fix a geodesic path system $\PS$ on $X$ (cf. \ref{defn-paths}). A set $\Ac \subset X$ is a \emph{contracting subset in the sense of $\BF$} if there exists a constant $C$ such that: if $\sigma \subset X$ is a geodesic in $\PS$ for which $\dist(\sigma,\Ac)>C$, then
\begin{equation*}
\diam \left( \pid _\Ac (\sigma) \right) \leq C.
\end{equation*}
Suppose $G$ preserves $\PS$. Then $g \in G$ is a \emph{contracting element in the sense of $\BF$} if for any $x_0 \in X$:
\begin{enumerate}
\item $g$ has infinite order and $\langle g \rangle \cdot  x_0$ is a quasi-isometric embedding of $\Zb$ in $X$, and
\item $\langle g \rangle  \cdot  x_0$ is a contracting subset in the sense of $\BF$.
\end{enumerate}
\end{definition}

\begin{remark}
If $X$ is a proper CAT(0) geodesic metric space, \cite[Corollary 3.4]{bestvina_fujiwara_symm_space} prove that the definitions \ref{defn-BF-original} and \ref{defn-contracting-BF} are equivalent. In fact, they prove this equivalence for any metric space that satisfies their axioms DD and FT, see \cite{bestvina_fujiwara_symm_space}. However, it is unclear whether these two definitions \ref{defn-BF-original} and \ref{defn-contracting-BF} are equivalent in complete generality. We will discuss this in \cref{prop:comparison_B1_and_B2} below. \cref{prop:comparison_B1_and_B2} suggests that it is unlikely that these definitions are equivalent in general. 
\end{remark}

We will now prove \cref{prop:notion-of-contraction-equiv}: contraction in the sense of $\BF$ is equivalent to Sisto's notion (cf. \ref{defn-contraction-sisto} and \ref{defn-contracting-element-sisto}), in the context of a geodesic path system. Before starting the proof, we  record the following immediate consequence of \cref{defn-contracting-BF}.

\begin{lemma}
\label{lem-contraction-BF-obs1}
Suppose $A\subset X$ is \emph{contracting in the sense of $\BF$} with constant $C$. Let $\sigma_{x,x'} \in \PS$ be a geodesic joining $x$ and $x'$ such that $\sigma_{x,x'} \cap \overline{\Nc_{2C}(A)}=\{x'\}$. Then $\left( \sup_{a \in \pid_{A}(x)}\dist(a, x') \right) \leq 3C.$
\end{lemma}
\begin{proof}
Since $\dist(\sigma_{x,x'},A)\geq 2C>C$,  $\dist(y,y')\leq C$ for any $y \in \pid_{A}(x)$ and any $y' \in \pid_A(x')$. But $\dist(a',x') \leq 2C$ for any $a' \in \pid_A(x')$. Hence the conclusion. 
\end{proof}

\begin{proposition}[\cref{prop:notion-of-contraction-equiv}]
\label{prop:equiv_sisto_and_BF}
Suppose $(X,\PS)$ is a geodesic path system. Then: 
\begin{enumerate}
\item $\Ac \subset X$ is $\PS$-contracting if and only if $\Ac$ is contracting in the sense of $\BF$. 
\item If $G$ preserves $\PS$, then $g \in G$ is a contracting element for $(X,\PS)$ if and only if $g \in G$ is a contracting element in the sense of $\BF$.
\end{enumerate}
\end{proposition}
\begin{proof}
It suffices to prove only part (1) as (2) then follows from definitions. We will use $\sigma_{p,q}$ to denote a geodesic path in $\PS$  joining $p$ and $q$. We now start the proof of (1).

 $(\implies)$ Suppose $\Ac$ is contracting in the sense of BF. Define a projection map $\pi:X \to \Ac$ by choosing $\pi(x) \in \pid_{\Ac}(x)$ for each $x \in X$. We remark that such a map $\pi$ is coarsely unique, i.e. any other such map $\pi'$ has the property that $\dist(\pi(x),\pi'(x)) \leq 2C$ for any $x \in X$. Indeed, for any $x$ such that $\dist(x,\Ac)>C$, $\diam(\pid_{\Ac}(x)) \leq C$ as $\{x\}$ is a geodesic in $\PS$. On the other hand, if $\dist(x,\Ac)\leq C$, then $\sup_{a \in \pid_{\Ac}(x)}\dist(x,a) \leq C$ and hence $\diam(\pid_{\Ac}(x)) \leq 2C$. Hence the remark.

We now show that $\pi$ satisfies \cref{defn-contraction-sisto} with constant $3C$.  Clearly, if $x \in \Ac$, then $\dist(x,\pi(x))=0$. Now suppose that $x,y \in X$ is such that $\dist(\pi(x),\pi(y)) \geq 3C$. Then $\diam(\pid_{\Ac}(\sigma_{x,y})) \geq 3C > C$. Then $d(\sigma_{x,y},\Ac) \leq C$ and thus $\sigma_{x,y} \cap \overline{\Nc_{2C}(\Ac)} \neq \emptyset$. Let $x'$ be the first point along $\sigma_{x,y}$ that intersects $\overline{\Nc_{2C}(\Ac)}$ (assume that $\sigma_{x,y}$ is continuously parametrized in the direction from $x$ to $y$). If $x=x'$, then $\dist(\pi(x),x') \leq 2C$. Otherwise apply Lemma \ref{lem-contraction-BF-obs1} to $\sigma_{x,x'} \subset \sigma_{x,y}$ to see that $\dist(\pi(x),x') \leq 3C$. Similarly, if $y'$ is the last point along $\sigma_{x,y}$ where $\sigma_{x,y}$ intersects $\overline{\Nc_{2C}(A)}$, then $\dist(y',\pi(y)) \leq 3C$. Thus $\pi$ is a contracting projection with constant $3C$.

$(\impliedby)$ Suppose $\pi:X \to \Ac$ is a contracting projection with constant $C$. By Lemma \ref{lem:coarse-equiv-of-proj} $\left( \sup_{a \in \pid_{\Ac}(x)}\dist(a,\pi(x))  \right)\leq 2C$ for any $x \in X$. Let $\sigma_{x,y} \in \PS$ be such that $\dist(\sigma_{x,y},\Ac)>5C$. If possible, let there exist $a_1 \in \pid_{\Ac}(x)$ and $b_1\in \pid_{\Ac}(y)$ such that $\dist(a_1,b_1)>5C$. Then $$\dist(\pi(x),\pi(y)) \geq \dist(a_1,b_1)-\dist(a_1,\pi(x))-\dist(b_1,\pi(y)) > C.$$ Then, $\sigma_{x,y}$ must intersect  $\Nc_C(\Ac)$, a contradiction.  
\end{proof}

\subsection{Comparison between Definitions \ref{defn-BF-original} and \ref{defn-contracting-BF}}

To discuss the relationship between \ref{defn-BF-original} and \ref{defn-contracting-BF} for a general metric space, we need the following condition $(\blacklozenge)$. We will say that \emph{$A \subset X$ satisfies $(\blacklozenge)$ if there exists a  constant $C$ such that: for any $x \in X$, $z \in A$ and $a \in \pid_{A}(x)$}, 
\begin{equation*}
\dist(x,z) \geq \dist(x,a) + \dist(a,z)-C.
\end{equation*}
We will now show that:
\begin{proposition}
\label{prop:comparison_B1_and_B2}
Fix a proper geodesic metric space $X$ and a geodesic path system $\PS$ on $X$. Then $A \subset X$ is contracting in the sense of $\BF$ if and only if it satisfies $(\blacklozenge)$ and \cref{defn-BF-original}.
\end{proposition}
The implication $(\implies)$ follows from \cite[Lemma 2.10]{S2013}. Note that in  \cite{S2013}, condition $(\blacklozenge)$ is called (AP1) while \ref{defn-BF-original} is  called (AP2). This direction is then immediate from \cite[Lemma 2.10]{S2013}. The proof of the converse $(\impliedby)$ follows from the next two lemmas. For $p,q \in X$, we will denote by $\sigma_{p,q}$ a geodesic in $\PS$ joining $p$ and $q$.

\begin{lemma}
Suppose $A \subset X$ satisfies \cref{defn-contracting-BF}. Then $A$ satisfies $(\blacklozenge)$. 
\end{lemma}
\begin{proof}
Fix any $x \in X$, $z \in A$ and $a \in \pid_A(x)$. It suffices to only consider the case when $\dist(x,A)>2C$. Suppose $x' \in \sigma_{x,z}$ be the first point along $\sigma_{x,z}$ that intersects $\overline{\Nc_{2C}(A)}$ (assume that $\sigma_{x,z}$ is continuously parametrized in the direction from $x$ to $z$). By \cref{lem-contraction-BF-obs1}, $\dist(x',a) \leq 3C$. Then $\dist(x,a)-\dist(x,x') \leq \dist(x',a) \leq 3C$ and $\dist(z,a)-\dist(z,x') \leq \dist(x',a) \leq 3C$. Since $x' \in \sigma_{x,z}$, $\dist(x,z)=\dist(x,x')+\dist(x',z)$. Thus
\begin{equation*}
\dist(x,z)-\dist(x,a)-\dist(a,z)=(\dist(x,x')-\dist(x,a)) + (\dist(x',z)-\dist(a,z)) \geq -6C.\qedhere 
\end{equation*}
\end{proof}

\begin{lemma}
Suppose $\Ac \subset X$ satisfies \cref{defn-contracting-BF}. Then $\Ac$ also satisfies \cref{defn-BF-original}. 
\end{lemma}
\begin{proof}
 \cref{prop:equiv_sisto_and_BF} implies that $\Ac$ is $\PS$-contracting. Then there exists a projection map $\pi_{\Ac}:X \to \Ac$ with constant $C$ satisfying \cref{defn-contraction-sisto}. Suppose $x \in X$ and $0< R < d(x,\Ac).$ We claim that $\diam(\pid_{\Ac}(B(x,R)))\leq 20C.$ By Lemma \ref{lem:coarse-equiv-of-proj}, it suffices to prove that $\dist(\pi_{\Ac}(x),\pi_{\Ac}(y)) \leq 8C$ for any $y \in B(x,R)$. 
 
 Fix $y \in B(x,R)$ and let $\sigma_{x,y} \in \PS$. Without loss of generality, we can assume that $\dist(\pi_{\Ac}(x),\pi_{\Ac}(y)) \geq C$. Then there exists $x_1 \in \sigma_{x,y}$ such that $\dist(x_1,\pi_{\Ac}(x)) \leq C$. Then, $$|\dist(x,x_1)-\dist(x,\Ac)| \leq \dist(x_1,\pi_{\Ac}(x)) + \sup_{a \in \pid_{\Ac}(x)} \dist(a,\pi_{\Ac}(x)) \leq 3C.$$ Thus, $\dist(y,x_1) =\dist(y,x)-\dist(x,x_1) \leq \dist(y,x)-\dist(x,\Ac)+3C.$ As $\dist(y,x)<\dist(x,\Ac)$, $\dist(y,x_1) \leq 3C$. Then   
 \begin{align*}
&\dist(y,\pi_\Ac(y)) \leq \dist(y,\pi_{\Ac}(x)) \leq \dist(y,x_1)+\dist(x_1,\pi_\Ac(x)) \leq 4C, \text{ which implies that }\\
&\dist(\pi_\Ac(y),\pi_\Ac(x)) \leq \dist(\pi_\Ac(y),y) + \dist(y,x_1) + \dist(x_1,\pi_\Ac(x))  \leq 8C. \qedhere
\end{align*}
\end{proof}

\bibliographystyle{alpha}
\bibliography{list.bib}

\end{document}